\documentclass{amsart}[12pt]
\usepackage{amsmath}
\usepackage{amsfonts}
\usepackage{amssymb}
\usepackage{color}
\usepackage{graphicx}
\usepackage{blkarray}
\usepackage{subfigure}
\usepackage[margin=1.2in]{geometry}

\usepackage[numbers]{natbib} \setcitestyle{open={},close={}}

\usepackage{tikz}
\usetikzlibrary{matrix,arrows,decorations.pathmorphing}

\allowdisplaybreaks[4]
\usepackage[all]{xy}
\newtheorem{theorem}{Theorem}[section]
\newtheorem{proposition}[theorem]{Proposition}
\newtheorem{lemma}[theorem]{Lemma}
\newtheorem{remark}[theorem]{Remark}

\newtheorem{definition}[theorem]{Definition}
\newtheorem{corollary}[theorem]{Corollary}

\newcommand{\be}{\begin{equation}}
\newcommand{\ee}{\end{equation}}
\newcommand{\bea}{\begin{eqnarray}}
\newcommand{\eea}{\end{eqnarray}}
\newcommand{\ben}{\begin{eqnarray*}}
\newcommand{\een}{\end{eqnarray*}}

\begin{document}

\title{Relative orbifold Pandharipande-Thomas theory and the degeneration formula}

\author[Yijie Lin]{Yijie Lin}
\address{Yijie Lin: School of Mathematics (Zhuhai)\\Sun Yat-Sen University\\Zhuhai, 519082, China}
\email{yjlin12@163.com}

\maketitle

\begin{abstract}
We construct relative moduli spaces of semistable pairs	on a family of projective Deligne-Mumford stacks. We define  moduli stacks of stable orbifold Pandharipande-Thomas pairs on stacks of expanded degenerations and pairs, and then show they are separated and proper Deligne-Mumford stacks of finite type.
As an application, we present the degeneration formula for the absolute and relative orbifold Pandharipande-Thomas invariants.	
\end{abstract}




\tableofcontents

\section{Introduction}
The notion of degeneration is  extensively applied in moduli problems  and curve counting theories in  enumerative geometry. Good degenerations are introduced by Jun Li [\cite{Li3}] as an appropriate class of degenerations  to investigate the geometry of moduli spaces, and their constructions depend on the notion of the stack of expanded degenerations. In  [\cite{Li1}], the author constructs the good degeneration of moduli of stable morphisms to a simple degeneration of projective varieties, which leads to a degeneration formula of Gromov-Witten (GW) invariants [\cite{Li2}]. With the construction of good degenerations of Grothendieck's Quot schemes (including Hilbert schemes of curves) and of moduli of coherent systems (or Pandharipande-Thomas (PT) stable pairs) on any simple degeneration  [\cite{LW}],  the authors derive degeneration formulas of Donaldson-Thomas (DT) invariants and of PT invariants (see also [\cite{MPT}]) for projective 3-folds. Degeneration formulas of the above three curve counting theories  are not only useful for computing their invariants and  even determining their structures (e.g., [\cite{BP,OP1,PP3}]), but also important for proving GW/DT and GW/PT correspondence [\cite{BP,OP1,MOOP,MPT}].  

In the generalized orbifold case, or more generally the case of smooth projective Deligne-Mumford stacks, the GW theory has been studied in [\cite{CR1,AGV}]. A degeneration formula of this orbifold GW theory is obtained in [\cite{AF}] by using the notion of stacks of twisted  expanded degenerations and pairs  in [\cite{ACFW}] which generalizes the original one defined by Jun Li [\cite{Li1}]. Without the twisted condition, the author in [\cite{Zhou1}] directly generalizes Jun Li's stacks of expanded degenerations and pairs to the case of smooth projective Deligne-Mumford stacks, which leads to the construction of  good degenerations of Hilbert stacks generalizing the one in [\cite{LW}]. In [\cite{Zhou1}], the author also gives the definition of  the absolute and relative orbifold DT theory (see also  [\cite{BCY,GT}] for the prior definition with Calabi-Yau condition) and derives the corresponding degeneration formulas. In [\cite{Lyj}], the author constructs moduli spaces of orbifold PT stable pairs and defines virtual fundamental classes which are integrated to give the absolute orbifold PT invariants for $3$-dimensional smooth projective Deligne-Mumford stacks  (see also [\cite{BCR}] for the prior definition in the case of Calabi-Yau 3-orbifolds). However, it is naturally expected to have the relative orbifold PT theory and the corresponding degeneration formula. 

In order to  define the relative orbifold PT invariants, one needs to consider the  construction of relative moduli spaces of semistable pairs, which is a relative version of   the  one in [\cite{Lyj}]. We first state the result of this construction as follows. Let $p: \mathcal{X}\to S$ be a family of projective Deligne-Mumford stacks   with a moduli scheme $\pi:\mathcal{X}\to X$ and a relative polarization  $(\mathcal{E}, \mathcal{O}_{X}(1))$, see Definition \ref{polarizations}. Let  $\mathcal{F}_{0}$ be a fixed $S$-flat coherent sheaf on $\mathcal{X}$. Assume that $\delta\in\mathbb{Q}[m]$ is a given polynomial with positive leading coefficient or zero, and $P$ is a given polynomial of degree $d$ where $d\leq\dim\mathcal{X}_{s}$ for any geometric point  $\mathrm{Spec}\,k\xrightarrow{s} S$. We consider the following contravariant functor
\ben
\mathcal{M}^{(s)s}_{\mathcal{X}/S}(\mathcal{F}_{0},P,\delta): (\mathrm{Sch}/S)^\circ\to(\mathrm{Sets})
\een
where for  an $S$-scheme $T$ of finite type,  $\mathcal{M}^{(s)s}_{\mathcal{X}/S}(\mathcal{F}_{0},P,\delta)(T)$ is the set of isomorphism classes of flat families of $\delta$-(semi)stable pairs  $(\mathcal{F},\varphi)$ on $\mathcal{X}\times_{S}T$ with  
modified Hilbert polynomial $P$ parametrized by  $T$, see Definition \ref{moduli-functor} for more details. The existence of relative moduli spaces of $\delta$-(semi)stable pairs for the moduli functor $\mathcal{M}^{(s)s}_{\mathcal{X}/S}(\mathcal{F}_{0},P,\delta)$ is proven in the following
\begin{theorem} [see Theorem \ref{corepresent} and Theorem \ref{fine-coarse}]
\label{relative-moduli}	There exists a projective $S$-scheme $\widetilde{M}^{ss}:=\widetilde{M}^{ss}_{\mathcal{X}/S}(\mathcal{F}_{0},P,\delta)$ which is a moduli space for the moduli functor $\mathcal{M}^{ss}_{\mathcal{X}/S}(\mathcal{F}_{0},P,\delta)$. Moreover, there is an open subscheme  $\widetilde{M}^{s}:=\widetilde{M}^{s}_{\mathcal{X}/S}(\mathcal{F}_{0},P,\delta)$ of $\widetilde{M}^{ss}$ which is a fine moduli space  for the moduli functor $\mathcal{M}^s_{\mathcal{X}/S}(\mathcal{F}_{0},P,\delta)$.
\end{theorem}
As an application of Theorem \ref{relative-moduli}, we consider the special case that $k=\mathbb{C}$, $\mathcal{F}_{0}=\mathcal{O}_{\mathcal{X}}$, $\deg\delta\geq\deg P=1$ and $p: \mathcal{X}\to S$ is a family of  projective Deligne-Mumford stacks of relative dimension three. Then we have the fine moduli space $\mathrm{PT}_{\mathcal{X}/S}^{(P)}:=\widetilde{M}^{s}_{\mathcal{X}/S}(\mathcal{O}_{\mathcal{X}},P,\delta)$ parameterizing orbifold PT stable pairs on $\mathcal{X}/S$, which is a projective $S$-scheme, see the first paragraph in Section 4 for more details. Since $\widetilde{\mathbf{M}}^s:=\mathrm{PT}_{\mathcal{X}/S}^{(P)}$ is a fine moduli space, there is a universal complex 
$\mathbb{I}^\bullet=\{\mathcal{O}_{\mathcal{X}\times_{S} \widetilde{\mathbf{M}}^s}\to\mathbb{F}\}$  on $\mathcal{X}\times_{S} \widetilde{\mathbf{M}}^s$. Let $\tilde{\pi}_{\widetilde{\mathbf{M}}^s}:\mathcal{X}\times_{S} \widetilde{\mathbf{M}}^s\to \widetilde{\mathbf{M}}^s$ and $\tilde{\pi}_{\mathcal{X}}:\mathcal{X}\times_{S} \widetilde{\mathbf{M}}^s\to\mathcal{X}$ be the projections. As in the absolute case (see  [\cite{Lyj}, Theorem 1.6]), if each fiber of  $p$ is smooth, we have the  relative version.
\begin{theorem}[see Theorem \ref{perf-ob1}]
	Let $p: \mathcal{X}\to S$ be a family of smooth projective Deligne-Mumford stacks of relative dimension three   with a moduli scheme $\pi:\mathcal{X}\to X$ and a relative polarization  $(\mathcal{E}, \mathcal{O}_{X}(1))$.
The map 
	\ben
	\widetilde{\Phi}: \widetilde{\mathbf{E}}^\bullet:=R\tilde{\pi}_{\widetilde{\mathbf{M}}^s*}(R\mathcal{H}om(\mathbb{I}^{\bullet},\mathbb{I}^{\bullet})_{0}\otimes\tilde{\pi}_{\mathcal{X}}^*\omega_{\mathcal{X}/S})[2]\to\mathbb{L}_{\widetilde{\mathbf{M}}^s/S}
	\een
	 is a perfect relative obstruction theory for $\widetilde{\mathbf{M}}^s$ over $S$ in the sense of [\cite{BF}]. And there exists a virtual fundamental class $[\widetilde{\mathbf{M}}^s]^{\mathrm{vir}}\in A_{\mathrm{vdim}+\dim S}(\widetilde{\mathbf{M}}^s)$ of virtual dimension $\mathrm{vdim}=\mathrm{rk}(\widetilde{\mathbf{E}}^\bullet)$.
\end{theorem}
However, the moduli space $\mathrm{PT}_{\mathcal{X}/S}^{(P)}$ with its virtual fundamental class $[\mathrm{PT}_{\mathcal{X}/S}^{(P)}]^\mathrm{vir}$ is not the right space to   define directly relative orbifold PT invariants and hence
to derive a degeneration formula of orbifold PT theory, but it provides an elementary basis to construct the desired good degeneration of moduli spaces of PT stable pairs, see Section 5.
Next, we will deal with the orbifold PT side parallel to the DT theory in [\cite{LW,Zhou1}]. We  briefly describe the construction of this good degeneration as follows. Now, let $\pi:\mathcal{X}\to\mathbb{A}^1$ be a simple degeneration of relative dimension three (see Definition \ref{simple-degeneration}) and also a family of projective Deligne-Mumford stacks with a coarse moduli scheme $X$ and a relative polarization $(\mathcal{E}_{\mathcal{X}},\mathcal{O}_{X}(1))$. Following the method in [\cite{Li1,LW}],  a good degeneration constructed here is to fill some suitable central fiber in the following family \ben\coprod_{c\neq0}\mathrm{PT}_{\mathcal{X}_{c}/\mathbb{C}}^{P},
\een where $\mathrm{PT}_{\mathcal{X}_{c}/\mathbb{C}}^{P}$  is defined  in Remark \ref{absolute} as an open and closed subscheme of $\mathrm{PT}_{\mathcal{X}_{c}/\mathbb{C}}$  parameterizing orbifold PT stable pairs on the smooth fiber $\mathcal{X}_{c}$ with a fixed Hilbert homomorphism $P\in \mathrm{Hom}(K^0(\mathcal{X}_{c}),\mathbb{Z})$, see Section 5.3 for  Hilbert homomorphisms here and below. Here, $\mathrm{PT}_{\mathcal{X}_{c}/\mathbb{C}}$ is a disjoint union of  $\mathrm{PT}_{\mathcal{X}_{c}/\mathbb{C}}^{(P)}$, see Section 5.1. And we adopt the notations  that the superscript $(P)$ means the associated modified Hilbert polynomial is $P$ while the superscript $P$ without parentheses is used to indicate the Hilbert homomorphism $P$, see also Section 4. The orbifold PT stable pairs on the central fiber which is any possible expanded degeneration $\mathcal{X}_{0}[k]$ with $\mathcal{X}_{0}:=\mathcal{Y}_{-}\cup_{\mathcal{D}}\mathcal{Y}_{+}$ (see Definition \ref{expanded-deg/pair}) are required to satisfy some stability condition in Definition \ref{stable} and have the fixed Hilbert homomorphism $P\in \mathrm{Hom}(K^0(\mathcal{X}_{0}),\mathbb{Z})$. To avoid the confusion of two notions of stability appeared above, we make the convention that an orbifold PT stable pair is simply called an orbifold PT pair in this paper, see also the first paragraph in Section 5. Then  the good degeneration is constructed  as the moduli stack of stable orbifold PT pairs with fixed Hilbert homomorphism $P\in \mathrm{Hom}(K^0(\mathcal{X}),\mathbb{Z})$, denoted by $\mathfrak{PT}_{\mathfrak{X}/\mathfrak{C}}^{P}$ in the degeneration case, see Sections 5.2 and 5.3. Similarly, if $(\mathcal{Y},\mathcal{D})$ is a smooth pair (see Definition \ref{simple-degeneration}) and $\mathcal{Y}$
is  a $3$-dimensional projective Deligne-Mumford stack with a coarse moduli scheme $Y$ and  a  polarization $(\mathcal{E}_{\mathcal{Y}},\mathcal{O}_{Y}(1))$, one can define the moduli stack $\mathfrak{PT}_{\mathfrak{Y}/\mathfrak{A}}^{P}$ parameterizing stable orbifold PT pairs with  a fixed Hilbert homomorphism   $P\in\mathrm{Hom}(K^0(\mathcal{Y}),\mathbb{Z})$  in the relative case. Here, $\mathfrak{C}$ and $\mathfrak{A}$ are stacks of expanded degenerations and pairs with $\mathfrak{X}$ and $\mathfrak{Y}$ as their universal families defined in [\cite{Zhou1}], see also Definitions \ref{degen-stacks} and \ref{pair-stacks}. We obtain the properties of these stacks  as follows.

\begin{theorem}[see Theorem \ref{stack1} and Propositions \ref{boundedness}, \ref{separatedness}, \ref{properness}]
	The stack	$\mathfrak{PT}_{\mathfrak{X}/\mathfrak{C}}^{P}$ $($resp. $\mathfrak{PT}_{\mathfrak{Y}/\mathfrak{A}}^{P}$$)$ is a  proper Deligne-Mumford stack of finite type over $\mathbb{A}^1$ $($resp. over $\mathbb{C}$$)$.
\end{theorem}

Let $\Lambda:=\mathrm{Hom}(K^0(\mathcal{X}),\mathbb{Z})$. Taking elements in $\Lambda$ as the weight assignments on  expanded degenerations $\mathcal{X}_{0}[k]$ is useful for the decomposition of the central fiber of the good degeneration	$\mathfrak{PT}_{\mathfrak{X}/\mathfrak{C}}^{P}$ over $\mathbb{A}^1$, see Definiton \ref{weight-assignment}. Denote by
\ben
\Lambda_{P}^{\mathrm{spl}}:=\{\varrho=(\varrho_{-},\varrho_{+},\varrho_{0})|\varrho_{\pm},\varrho_{0}\in\Lambda, \varrho_{-}+\varrho_{+}-\varrho_{0}=P\}
\een
the splitting data of $P\in\Lambda$. For $\varrho\in\Lambda_{P}^{\mathrm{spl}}$, stacks $\mathfrak{PT}_{\mathfrak{X}_{0}^\dagger/\mathfrak{C}_{0}^\dagger}^P$, $\mathfrak{PT}_{\mathfrak{X}_{0}^\dagger/\mathfrak{C}_{0}^\dagger}^\varrho$, $\mathfrak{PT}_{\mathfrak{Y}_{-}/\mathfrak{A}}^{\varrho_{-},\varrho_{0}}$, and  $\mathfrak{PT}_{\mathfrak{Y}_{+}/\mathfrak{A}}^{\varrho_{+},\varrho_{0}}$ are defined in Section 6.1. Similarly, one can also define perfect relative obstruction theories on them and prove the existence of virtual fundamental classes, see Section 6.2.
Now, we have the following three Cartesian diagrams:
\ben
\xymatrixcolsep{3pc}\xymatrix{
	\mathrm{PT}_{\mathcal{X}_{c}/\mathbb{C}}^P\ar[d] \ar[r] &  \mathfrak{PT}^P_{\mathfrak{X}/\mathfrak{C}} \ar[d] & \mathfrak{PT}_{\mathfrak{X}_{0}^\dagger/\mathfrak{C}_{0}^\dagger}^\varrho\ar@{^{(}->}[r]^-{\varsigma_{\varrho}}	& \mathfrak{PT}_{\mathfrak{X}_{0}^\dagger/\mathfrak{C}_{0}^\dagger}^P\ar[d] \ar[r] & \mathfrak{PT}^P_{\mathfrak{X}/\mathfrak{C}} \ar[d]\\
	\{c\}\ar@{^{(}->}[r]^-{i_{c}} & \mathbb{A}^1   & 	&\{0\}\ar@{^{(}->}[r]^-{i_{0}} & \mathbb{A}^1
}
\een
and
\ben
\xymatrixcolsep{3pc}\xymatrix{
	\mathfrak{PT}_{\mathfrak{X}_{0}^\dagger/\mathfrak{C}_{0}^\dagger}^\varrho 	&\ar[l]_-{\Phi_{\varrho}}^-{\cong}\mathfrak{PT}_{\mathfrak{Y}_{-}/\mathfrak{A}}^{\varrho_{-},\varrho_{0}}\times_{\mathrm{Hilb}_{\mathcal{D}}^{\varrho_{0}}}\mathfrak{PT}_{\mathfrak{Y}_{+}/\mathfrak{A}}^{\varrho_{+},\varrho_{0}}\ar[d] \ar[r] &  	\mathfrak{PT}_{\mathfrak{Y}_{-}/\mathfrak{A}}^{\varrho_{-},\varrho_{0}}\times \mathfrak{PT}_{\mathfrak{Y}_{+}/\mathfrak{A}}^{\varrho_{+},\varrho_{0}} \ar[d] \\
	& \mathrm{Hilb}_{\mathcal{D}}^{\varrho_{0}}\ar[r]^-{\Delta}  &  \mathrm{Hilb}_{\mathcal{D}}^{\varrho_{0}}\times\mathrm{Hilb}_{\mathcal{D}}^{\varrho_{0}} &  
}
\een
where maps $\varsigma_{\varrho}$, $i_{0}$, $i_{c}$ $(c\neq0)$ are the inclusions and  $\mathrm{Hilb}_{\mathcal{D}}^{\varrho_{0}}$ is the Hilbert stack on $\mathcal{D}$, see Section 6.3. Then we have the  following cycle version of degeneration formula.
\begin{theorem}[see Theorem \ref{cycle-degenerate}]\label{intr-cyc-deg}
	We  have 
	\ben
	&&i_{c}^![\mathfrak{PT}^P_{\mathfrak{X}/\mathfrak{C}}]^\mathrm{vir}=[	\mathrm{PT}_{\mathcal{X}_{c}/\mathbb{C}}^P]^\mathrm{vir},\\
	&&i_{0}^![\mathfrak{PT}^P_{\mathfrak{X}/\mathfrak{C}}]^\mathrm{vir}=\sum_{\varrho\in\Lambda_{P}^\mathrm{spl}}\varsigma_{\varrho*}\Delta^!([\mathfrak{PT}_{\mathfrak{Y}_{-}/\mathfrak{A}}^{\varrho_{-},\varrho_{0}}]^\mathrm{vir}\times[\mathfrak{PT}_{\mathfrak{Y}_{+}/\mathfrak{A}}^{\varrho_{+},\varrho_{0}}]^\mathrm{vir}).
	\een
\end{theorem}
Let $\mathcal{Z}$ be a  smooth projective Deligne-Mumford stack of dimension 3 over $\mathbb{C}$. Assume  $K(\mathcal{Z}):=K^0(\mathrm{Coh}(\mathcal{Z}))_{\mathbb{Q}}$ is the Grothendieck group of $\mathcal{Z}$ over $\mathbb{Q}$. Then  $\mathrm{Hom}(K(\mathcal{Z}),\mathbb{Q})$ can be identified with  $K(\mathcal{Z})$, see Section 6.2. Denote by $F_{i}K(\mathcal{Z})$  the subgroup of $K(\mathcal{Z})$ generated by all sheaves $\mathcal{G}$ with $\dim\mathrm{Supp}\,\mathcal{G}\leq i$. On the one hand, let $P\in F_{1}K(\mathcal{X}_{c})$, one can define the absolute orbifold PT invariants of the form $\langle\prod_{i=1}^r\tau_{k_{i}}(\gamma_{i})\rangle_{\mathcal{X}_{c}}^P$on $\mathcal{X}_{c}$, see Definition \ref{absolute-def}. Here, $\gamma_{i}\in A_{\mathrm{orb}}^*(\mathcal{X})$ and denote again by $\gamma_{i}$ the restriction of $\gamma_{i}$ on $\mathcal{X}_{c}$. On the other hand, if $P\in F_{1}K(\mathcal{Y})$, one has the moduli stack $\mathfrak{PT}^{P,P_{0}}_{\mathfrak{Y}/\mathfrak{A}}$ with  $P_{0}:=i^*P\in F_{0}K(\mathcal{D})$ where $i:\mathcal{D}\to\mathcal{Y}$ is the inclusion, which is the stack we use to define the relative orbifold PT invariants $\langle\prod_{i=1}^r\tau_{k_{i}}(\gamma_{i})|C\rangle_{\mathcal{Y},\mathcal{D}}^P$ where $\gamma_{i}\in A^*_{\mathrm{orb}}(\mathcal{Y})$ and $C\in A^*(\mathrm{Hilb}^{P_{0}}_{\mathcal{D}})$, see Definition \ref{relative-def}. It follows from Theorem \ref{intr-cyc-deg} that we have the numerical version of degeneration formula.
\begin{theorem}[see Theorem \ref{numerical-deg1}]
	Assume $\gamma_{i}\in A_{\mathrm{orb}}^*(\mathcal{X})$ and  $\gamma_{i,\pm}$ are the restriction of $\gamma_{i}$ on $\mathcal{Y}_{\pm}\subset\mathcal{X}_{0}$  for $1\leq i\leq r$. Suppose  $\gamma_{i,\pm}$ are disjoint with $\mathcal{D}$. For a fixed $P\in F_{1}K(\mathcal{X}_{c})$, we have
	\ben
	\bigg\langle\prod_{i=1}^r\tau_{k_{i}}(\gamma_{i})\bigg\rangle^P_{\mathcal{X}_{c}}=\sum_{\substack{\varrho_{-}+\varrho_{+}-P_{0}=P\\T\subset\{1,\cdots,r\},k,l}}\bigg\langle\prod_{i\in T}\tau_{k_{i}}(\gamma_{i,-})\bigg|C_{k}\bigg\rangle_{\mathcal{Y}_{-},\mathcal{D}}^{\varrho_{-}}g^{kl}\bigg\langle\prod_{i\notin T}\tau_{k_{i}}(\gamma_{i,+})\bigg|C_{l}\bigg\rangle_{\mathcal{Y}_{+},\mathcal{D}}^{\varrho_{+}}
	\een
	where $\varrho_{\pm}\in F_{1}K(\mathcal{Y}_{\pm})$ in the sum is taken over all possible splitting data satisfying $\varrho_{-}+\varrho_{+}-P_{0}=P$, and the set $\{C_{k}\}$ is a basis of $A^*(\mathrm{Hilb}^{P_{0}}_{\mathcal{D}})$ satisfying
	$
	\int_{\mathrm{Hilb}^{P_{0}}_{\mathcal{D}}}C_{k}\cup C_{l}=g_{kl}.
	$
Here, $(g^{kl})$ is the inverse matrix of $(g_{kl})$. 
\end{theorem}
As an application, we follow [\cite{Zhou1}, Section 8.2] to consider the multi-regular classes, see Definition \ref{multi-regular-class}. Let $F_{1}^{\mathrm{mr}}K(\mathcal{Z})$ be the subgroup of $F_{1}K(\mathcal{Z})$ spanned by multi-regular classes.
Then we have

\begin{theorem}[see Theorem \ref{multi-deg-form}]\label{intro-multi-deg-form}
	Assume $\gamma_{i}\in A_{\mathrm{orb}}^*(\mathcal{X})$ and  $\gamma_{i,\pm}$ are disjoint with $\mathcal{D}$  for $1\leq i\leq r$. For a fixed $\beta\in F_{1}^{\mathrm{mr}}K(\mathcal{X}_{c})/F_{0}K(\mathcal{X}_{c})$, we have
	\ben
	\bigg\langle\prod_{i=1}^r\tau_{k_{i}}(\gamma_{i})\bigg\rangle_{\mathcal{X}_{c}}^{(\beta,\epsilon)}=\sum_{\substack{\beta_{-}+\beta_{+}=\beta\\ \epsilon_{-}+\epsilon_{+}=\epsilon+n\\ T\subset\{1,\cdots,r\},k,l}}\bigg\langle\prod_{i\in T}\tau_{k_{i}}(\gamma_{i,-})\bigg|C_{k}\bigg\rangle_{\mathcal{Y}_{-},\mathcal{D}}^{(\beta_{-},\epsilon_{-})}g^{kl}\bigg\langle\prod_{i\notin T}\tau_{k_{i}}(\gamma_{i,+})\bigg|C_{l}\bigg\rangle_{\mathcal{Y}_{+},\mathcal{D}}^{(\beta_{+},\epsilon_{+})}.
	\een
where classes $(\beta_{\pm},\epsilon_{\pm})\in \left(F_{1}^{\mathrm{mr}}K(\mathcal{Y}_{\pm})/F_{0}K(\mathcal{Y}_{\pm})\right)\oplus F_{0}K(\mathcal{Y}_{\pm})$ and  $(\beta,\epsilon)\in \left(F_{1}^{\mathrm{mr}}K(\mathcal{X}_{c})/F_{0}K(\mathcal{X}_{c})\right)\oplus F_{0}K(\mathcal{X}_{c})$ satisfy $\beta_{-}+\beta_{+}=\beta$, $\epsilon_{-}+\epsilon_{+}=\epsilon+n$ and $\beta\cdot\mathcal{D}=n[\mathcal{O}_{x}]$. Here, $x\in\mathcal{D}$ is the preimage of a point in the coarse moduli space of $\mathcal{D}$.
\end{theorem}
By Theorem \ref{intro-multi-deg-form} and  the (relative) descendent PT generating function in Definition \ref{PT-gen-func}, we have 

\begin{theorem}[see Theorem \ref{mul-deg-form-genfun}]
	Assume $\gamma_{i}\in A_{\mathrm{orb}}^*(\mathcal{X})$ and  $\gamma_{i,\pm}$ are disjoint with $\mathcal{D}$  for $1\leq i\leq r$. For a fixed $\beta\in F_{1}^{\mathrm{mr}}K(\mathcal{X}_{c})/F_{0}K(\mathcal{X}_{c})$, we have
	\ben
	&&Z_{\beta}^{\mathrm{PT}}\bigg(\mathcal{X}_{c};q\bigg|\prod_{i=1}^r\tau_{k_{i}}(\gamma_{i})\bigg)\\
	&=&\sum_{\substack{\beta_{-}+\beta_{+}=\beta\\T\subset\{1,\cdots,r\},k,l}}\frac{g^{kl}}{q^n}Z_{\beta_{-},C_{k}}^{\mathrm{PT}}\bigg(\mathcal{Y}_{-},\mathcal{D};q\bigg|\prod_{i\in T}\tau_{k_{i}}(\gamma_{i,-})\bigg)\times Z_{\beta_{+},C_{l}}^{\mathrm{PT}}\bigg(\mathcal{Y}_{+},\mathcal{D};q\bigg|\prod_{i\notin T}\tau_{k_{i}}(\gamma_{i,+})\bigg).
	\een
where we use the same notation as in Theorem \ref{intro-multi-deg-form}.	
\end{theorem}
In [\cite{Zhou2,ZZ}], the authors compute the relative orbifold DT invariants and relative orbifold GW invariants of $[\mathbb{C}^2/\mathbb{Z}_{n+1}]\times\mathbb{P}^1$  respectively and prove the crepant resolution conjecture relating with their corresponding theories of $\mathcal{A}_{n}\times\mathbb{P}^1$, where $\mathcal{A}_{n}\to\mathbb{C}^2/\mathbb{Z}_{n+1}$ is a crepant resolution. The GW/DT/Hilb/Sym correspondence is proved for the surface $\mathcal{A}_{n}$ in [\cite{Mau,MO,MO1,CA}] and also  for the stacky quotient $[\mathbb{C}^2/\mathbb{Z}_{n+1}]$ in [\cite{Zhou2,ZZ}], see the beautiful diagram in [\cite{ZZ}, Section 6.2] for more details. We will  explore the computation of the relative orbifold PT invariants of $[\mathbb{C}^2/\mathbb{Z}_{n+1}]\times\mathbb{P}^1$ elsewhere as in DT case [\cite{Zhou2}] and consider  putting this orbifold PT theory in the GW/DT/Hilb/Sym correspondence for $[\mathbb{C}^2/\mathbb{Z}_{n+1}]$.

This paper is organized as follows. In Section 2, we recall the notions of (families) of projective Deligne-Mumford stacks with some topological data and of stacks of expanded degenerations and pairs, and then recall the admissibility of  sheaves and its numerical criterion. We construct  relative moduli spaces of $\delta$-(semi)stable pairs  in Section 3 with moduli spaces of orbifold PT stable pairs as a special case where  we construct a perfect relative obstruction theory and prove the existence of the virtual fundamental class in Section 4. On stacks of expanded degenerations and pairs, we define  moduli stacks of stable orbifold PT pairs which are proved to be separated and  proper Deligne-Mumford stacks of finite type in Section 5. Finally, in Section 6, we derive the degeneration formula of the orbifold PT theory.

{\bf Acknowledgements.} The author thank an anonymous referee for his helpful  comments. This work was partially  supported by Chinese Universities Scientific Fund (74120-31610010)
and Guangdong Basic and Applied Basic Research Foundation (2019A1515110255).

\section{Preliminaries}
In this section, we first briefly recall the definition of (families of) projective Deligne-Mumford stacks and of the modified Hilbert polynomial and Hilbert homomorphism of sheaves. See [\cite{OS03,Kre,Nir1}] for more details. We next recall definitions and properties of stacks (families) of expanded degenerations and pairs on smooth Deligne-Mumford stacks in [\cite{Zhou1}], which are important for obtaining degeneration formulas. Then we recall the definition of admissible sheaves and the associated numerical criterion, which is useful for showing the properness of good degenerations in Section 5.4. 
\subsection{Projective Deligne-Mumford stacks}
Let $k$ be an algebraically closed field of characteristic 0.
Let the base scheme $S$ be connected, noetherian and  of finite type over $k$. 
Each Deligne-Mumford $S$-stack is assumed to be separated, notherian and of finite type over $S$. One may refer to [\cite{DM,LMB,Ols,Vis}] for the notion of stacks and their properties.
It is shown in [\cite{KM,Con}] that a Deligne-Mumford stack $\mathcal{X}$ over  $S$ has a coarse moduli space $\pi:\mathcal{X}\to X$ which is proper and quasi-finite. When $X$ is a scheme, we simply call it a moduli scheme. Denote by $\mathrm{QCoh}$ the category of quasi-coherent sheaves.  
\begin{definition}(see [\cite{OS03}, Definition 5.1] and [\cite{Nir1}, Definition 2.4 and 2.6])
Let $\mathcal{X}$ be a Deligne-Mumford stack over $S$ with a moduli scheme $\pi:\mathcal{X}\to X$. A locally free sheaf $\mathcal{E}$ is said to be a generator for the quasi-coherent sheaf $\mathcal{F}$ if the left adjoint of the identity $\pi_{*}(\mathcal{F}\otimes\mathcal{E}^\vee)\xrightarrow{id}\pi_{*}(\mathcal{F}\otimes\mathcal{E}^\vee)$, i.e.,
\ben
\theta_{\mathcal{E}}(\mathcal{F}): G_{\mathcal{E}}\circ F_{\mathcal{E}} (\mathcal{F}) =\pi^{*}\pi_{*}\mathcal{H}om_{\mathcal{O}_{\mathcal{X}}}(\mathcal{E},\mathcal{F})\otimes\mathcal{E}\to\mathcal{F}
\een
is surjective where two functors $F_{\mathcal{E}}$ and $G_{\mathcal{E}}$ are defined as
\ben
&&F_{\mathcal{E}}:\mathrm{QCoh}(\mathcal{X}/S)\to\mathrm{QCoh}(X/S),\;\;\;\;\mathcal{F}\mapsto \pi_{*}\mathcal{H}om_{\mathcal{O}_{\mathcal{X}}}(\mathcal{E},\mathcal{F});\\
&&G_{\mathcal{E}}:\mathrm{QCoh}(X/S)\to\mathrm{QCoh}(\mathcal{X}/S),\;\;\;\;
F\mapsto\pi^*F\otimes\mathcal{E}.
\een
The sheaf $\mathcal{E}$ is called a generating sheaf for $\mathcal{X}$ if it is a generator for every quasicoherent sheaf on $\mathcal{X}$.
\end{definition}
\begin{remark}
As in [\cite{Nir1}], we also have the right adjoint of the identity $\pi^{*}F\otimes\mathcal{E}\xrightarrow{id}\pi^{*}F\otimes\mathcal{E}$ defined by
$\varphi_{\mathcal{E}}(F): F\to\pi_{*}\mathcal{H}om(\mathcal{E},\pi^*F\otimes\mathcal{E})=F_{\mathcal{E}}(G_{\mathcal{E}}(F))$ for any quasi-coherent sheaf $F$ on $X$.
\end{remark}	
\begin{definition}(see [\cite{Nir1}, Definition 2.20 and 2.23] and [\cite{Kre}, Corollary 5.4])
A Deligne-Mumford stack $\mathcal{X}$ over $k$ is a  projective stack if $\mathcal{X}$  has a projective moduli scheme and possesses a generating sheaf.

Let $p:\mathcal{X}\to S$ be a Deligne-Mumford $S$-stack with a moduli scheme $X$ and assume $\mathcal{X}$ possesses a generating sheaf. We call $p:\mathcal{X}\to S$ a family of projective Deligne-Mumford stacks if $p$ factorizes as $\pi:\mathcal{X}\to X$  followed by a projective morphism $\rho:X\to S$. 
\end{definition}

\begin{remark}\label{quo-gen}
By [\cite{EHKV}, Theorem 2.14] and [\cite{OS03}, Theorem 5.5], if a Deligne-Mumford $S$-stack $\mathcal{X}$ has a projective coarse moduli scheme, then $\mathcal{X}$ is global quotient stack over $S$ if and only if $\mathcal{X}$ possesses a generating sheaf.
\end{remark}

\begin{definition}(see [\cite{Nir1}, Section 3] and [\cite{BS}, Definition 2.16])\label{polarizations}
Let $\mathcal{X}$ be a projective Deligne-Mumford stack over $k$ with a moduli scheme $X$. Suppose $\mathcal{E}$ is a generating sheaf for $\mathcal{X}$ and $\mathcal{O}_{X}(1)$ is a very ample invertible sheaf on $X$ relative to $\mathrm{Spec}\,k$,  we define   $(\mathcal{E}, \mathcal{O}_{X}(1))$ as a polarization of $\mathcal{X}$. 
	
 A relative polarization of a family of projective Deligne-Mumford stacks $p:\mathcal{X}\to S$ is a pair $(\mathcal{E}, \mathcal{O}_{X}(1))$ where $\mathcal{E}$ is a generating sheaf for $\mathcal{X}$ and $\mathcal{O}_{X}(1)$ is a very ample invertible sheaf on $X$ relative to $S$.
\end{definition}
Now let $\mathcal{X}$ be a projective Deligne-Mumford stack over $k$ with a moduli scheme $\pi:\mathcal{X}\to X$ and a polarization $(\mathcal{E},\mathcal{O}_{X}(1))$. 
\begin{definition}(see [\cite{Nir1}, Definition 3.1 and 3.2] or [\cite{Lieb}, Section 2.2.6])
	Let $\mathcal{F}$ be a coherent sheaf on $\mathcal{X}$. Define the support of $\mathcal{F}$, denoted by $\mathrm{Supp}(\mathcal{F})$, to be the closed substack associated to the ideal $\mathcal{I}$:
	\ben
	0\to\mathcal{I}\to \mathcal{O}_{\mathcal{X}}\to \mathcal{E}nd_{\mathcal{O}_{\mathcal{X}}}(\mathcal{F}).
	\een	
	The dimension of $\mathcal{F}$ is defined as the dimension of its support.
	The $d$-dimensional sheaf $\mathcal{F}$  is called pure   if $\mathrm{dim}(\mathcal{G})=d$ for every nonzero subsheaf $\mathcal{G}\subset \mathcal{F}$.
\end{definition}
We have the following topological data.
\begin{definition}([\cite{Nir1}, Definition 3.10, Remark 3.11 and Definition 3.12]) \label{m-h-p}Let  
 $\mathcal{F}$ be a coherent sheaf   on $\mathcal{X}$.  The modified Hilbert polynomial of $\mathcal{F}$ is defined as
	\ben
	P_{\mathcal{E}}(\mathcal{F})(m):=\chi(\mathcal{X},\mathcal{F}\otimes\mathcal{E}^\vee\otimes\pi^*\mathcal{O}_{X}(m))=\chi(X,F_{\mathcal{E}}(\mathcal{F})(m))=P(F_{\mathcal{E}}(\mathcal{F})(m)).
	\een
Suppose $\mathcal{F}$ is a coherent sheaf of dimension $d$, we have 	
\ben
P_{\mathcal{E}}(\mathcal{F})(m)=\sum_{i=0}^{d}\alpha_{\mathcal{E},i}(\mathcal{F})\frac{m^i}{i!},
\een
where $r(F_{\mathcal{E}}(\mathcal{F})):=\alpha_{\mathcal{E},d}(\mathcal{F})$ is the multiplicity of $F_{\mathcal{E}}(\mathcal{F})$. Then the reduced Hilbert polynomial of $\mathcal{F}$ is defined as
\ben
p_{\mathcal{E}}(\mathcal{F})=\frac{P_{\mathcal{E}}(\mathcal{F})}{\alpha_{\mathcal{E},d}(\mathcal{F})}
\een
and the slope of $\mathcal{F}$ is defined by
\ben
\hat{\mu}_{\mathcal{E}}(\mathcal{F})=\frac{\alpha_{\mathcal{E},d-1}(\mathcal{F})}{\alpha_{\mathcal{E},d}(\mathcal{F})}.
\een
\end{definition}
Let $K^0(\mathcal{X})$ be the Grothendieck group of $\mathcal{X}$. There is another topological data in the following
\begin{definition}([\cite{Zhou1}, Definition 4.10])\label{h-h}
Let $\mathcal{F}$ be a coherent sheaf on $\mathcal{X}$. The Hilbert homomorphism of $\mathcal{F}$ with respect to $\mathcal{E}$ is  a group homomorphism $P_{\mathcal{F}}^{\mathcal{E}}: K^0(\mathcal{X})\to\mathbb{Z}$ defined by
\ben
[V]\mapsto\chi(\mathcal{X}, V\otimes_{\mathcal{O}_{\mathcal{X}}}\mathcal{F}\otimes_{\mathcal{O}_{\mathcal{X}}}\otimes\mathcal{E}^\vee)
\een
where $V$ is a vector bundle on $\mathcal{X}$, and extended to $K^0(\mathcal{X})$ additively.	
\end{definition} 
By Definitions \ref{m-h-p} and \ref{h-h}, we have $P_{\mathcal{E}}(\mathcal{F})(m)=P_{\mathcal{F}}^{\mathcal{E}}(H^{\otimes m})$, where $H:=\pi^*\mathcal{O}_{X}(1)$.

\subsection{Stacks of expanded degenerations and pairs}
In  [\cite{Li1}], Jun Li introduces the notion of stacks of expanded degenerations and pairs for variety cases, whose appearance is stimulated by Gieseker's degeneration of some moduli space  and Floer-Donaldson theory as pointed out in [\cite{Li3}]. These stacks are  crucial  to constructing good degenerations of moduli spaces of stable morphisms, ideal sheaves and stable pairs [\cite{Li1,Li2,LW}], and are furtherly investigated through several comparable approaches and generalized to the twisted version in [\cite{ACFW}]. In this subsection, we will recall the construction of stacks of expanded degenerations and pairs in [\cite{Zhou1}] for the case of Deligne-Mumford stacks, which is a direct generalization of Li's approach [\cite{Li1,Li3}].

\begin{definition}([\cite{Zhou1}, Definition 2.1 and Definition 2.2])\label{simple-degeneration}
	Suppose $\mathcal{Y}$ is a separated Deligne-Mumford stack of finite type over $\mathbb{C}$, and $\mathcal{D}\subset\mathcal{Y}$ is a  connected effective divisor. We call a pair $(\mathcal{Y}, \mathcal{D})$ locally smooth if $\acute{e}$tale locally near a point on $\mathcal{D}$,  $(\mathcal{Y}, \mathcal{D})$ is a pair of a smooth stack and a smooth divisor. A pair $(\mathcal{Y}, \mathcal{D})$ is called a  smooth pair if both $\mathcal{Y}$ and $\mathcal{D}$ are smooth.
	
	Let $\mathcal{X}$ be a separated Deligne-Mumford stack of finite type over $\mathbb{C}$, and 
	$0\in C$ be a smooth pointed curve over $\mathbb{C}$. A locally simple degeneration is a flat morphism $\pi:\mathcal{X}\to C$ with central fiber
	$\mathcal{X}_{0}=\mathcal{Y}_{-}\cup_{\mathcal{D}}\mathcal{Y}_{+}$, where $\mathcal{Y}_{\pm}\subset \mathcal{X}_{0}$ are closed substacks, $(\mathcal{Y}_{\pm}, \mathcal{D})$ are locally smooth pairs, and the intersection $\mathcal{Y}_{-}\cap \mathcal{Y}_{+}=\mathcal{D}$ is transversal, i.e., for any point $p\in \mathcal{D}$, we have the $\acute{e}$tale neighborhoods $V$ of $0\in C$ and $U$ of $p\in \mathcal{X}$ such that the following  diagram is  commutative:
	\ben
	\xymatrixcolsep{4pc}\xymatrix{
		\mathrm{Spec}\left(\frac{\mathbb{C}[x,y,t]}{(xy-t)}\right)\times \mathbb{A}^{n} \ar[d]^{\pi}  &    U \ar[l]_-{\acute{e}t} \ar[d]^{\pi|_{U}}\ar[r]^{\acute{e}t} & \mathcal{X} \ar[d]^{\pi} \\
		\mathrm{Spec}\,\mathbb{C}[t] & \ar[l]_-{\acute{e}t} V \ar[r]^{\acute{e}t}& C &
	}
	\een
	where $\acute{e}t$ means the $\acute{e}tale$ map. We call $\pi:\mathcal{X}\to C$ a simple degeneration if in addition $(\mathcal{Y}_{\pm}, \mathcal{D})$ are smooth pairs and for any $0\neq c\in C$, the fiber $\mathcal{X}_{c}:=\pi^{-1}(c)$ is smooth.
\end{definition}

As in  [\cite{Zhou1}, Remark 2.4], we simply  assume $C=\mathbb{A}^1$ from now on since we will only investigate the local behavior for degenerations.

\begin{definition}([\cite{Zhou1}, Section 2.1])\label{expanded-deg/pair}
For $\mathcal{X}_{0}=\mathcal{Y}_{-}\cup_{\mathcal{D}}\mathcal{Y}_{+}$ in a locally smooth simple degeneration, let $\mathcal{D}_{\pm}$ be the divisor $\mathcal{D}$ in $\mathcal{Y}_{\pm}$ and  $\mathcal{N}_{\pm}:=N_{\mathcal{D}_{\pm}/\mathcal{Y}_{\pm}}$  be the normal line bundles of $\mathcal{D}_{\pm}$ in $\mathcal{Y}_{\pm}$. There is a $\mathbb{P}^1$-bundle over $\mathcal{D}$
\ben
\Delta:=\mathbb{P}_{\mathcal{D}}(\mathcal{O}_{\mathcal{D}}\oplus \mathcal{N}_{+})\cong\mathbb{P}_{\mathcal{D}}(\mathcal{N}_{-}\oplus\mathcal{O}_{\mathcal{D}})
\een
with two distinguished sections $\mathcal{D}_{-}=\mathbb{P}_{\mathcal{D}}(\mathcal{O}_{\mathcal{D}}\oplus\, 0)$ and $\mathcal{D}_{+}=\mathbb{P}_{\mathcal{D}}(\mathcal{N}_{-}\oplus\,0)$. Here, we have $N_{\mathcal{D}_{+}/\Delta}\cong \mathcal{N}_{+}$, $N_{\mathcal{D}_{-}/\Delta}\cong \mathcal{N}_{-}$ and $\mathcal{N}_{-}\otimes\mathcal{N}_{+}=\mathcal{O}_{\mathcal{D}}$. For any integer $k\geq0$, let $\Delta_{0}:=\mathcal{Y}_{-}$ and $\Delta_{k+1}:=\mathcal{Y}_{+}$.  An expanded degeneration of length $k$ with respect to $\mathcal{X}_{0}$ is defined by
\ben
\mathcal{X}_{0}[k]:=\mathcal{Y}_{-}\cup_{\mathcal{D}}\Delta_{1}\cup_{\mathcal{D}}\cdots\cup_{\mathcal{D}}\Delta_{k}\cup_{\mathcal{D}}\mathcal{Y}_{+}
\een
where $\Delta_{i}$ is the $i$-th copy of $\Delta$ for $1\leq i\leq k$, and $\mathcal{D}_{-}\subset \Delta_{i}$ is glued with $\mathcal{D}_{+}\subset \Delta_{i+1}$  for $0\leq i\leq k$. The divisors in $\mathcal{X}_{0}[k]$ are denoted in order by $\mathcal{D}_{0}, \cdots, \mathcal{D}_{k}$. 

Similarly, for a locally smooth pair $(\mathcal{Y}, \mathcal{D})$,  the $\mathbb{P}^1$-bundle $\Delta:=\mathbb{P}_{\mathcal{D}}(\mathcal{O}_{\mathcal{D}}\oplus \mathcal{N})$  also has two distinguished sections $\mathcal{D}_{-}=\mathbb{P}_{\mathcal{D}}(\mathcal{N})$ and $\mathcal{D}_{+}=\mathbb{P}_{\mathcal{D}}(\mathcal{O}_{\mathcal{D}})$  where $\mathcal{N}:=N_{\mathcal{D}/\mathcal{Y}}$. For any integer $l\geq0$, define with  the same gluing rule as above
\ben
\mathcal{Y}[l]:=\mathcal{Y}\cup_{\mathcal{D}}\Delta_{1}\cup_{\mathcal{D}}\cdots\cup_{\mathcal{D}}\Delta_{l}.
\een
Let $\mathcal{D}[l]$ be the last divisor $\mathcal{D}_{-}\subset\Delta_{l}$. Now the pair $(\mathcal{Y}[l],\mathcal{D}[l])$ is also a locally smooth pair. An expanded pair of length $l$ with respect to $(\mathcal{Y},\mathcal{D})$ is defined as $(\mathcal{Y}[l],\mathcal{D}[l])$. The divisors in $(\mathcal{Y}[l],\mathcal{D}[l])$ are denoted in order by $\mathcal{D}_{0}, \cdots, \mathcal{D}_{l-1}$, $\mathcal{D}_{l}=\mathcal{D}[l]$. 
\end{definition}

We recall  the construction of standard families of expanded degenerations and pairs  in [\cite{Zhou1}, Section 2.3].  They are related to the expanded degeneration family of a standard local model  which is originally described  in  [\cite{Li1}] (see also [\cite{Zhou1}, Section 2.2]). Consider the standard local model
\ben
\underline{\pi}:\underline{U}=\mathrm{Spec}\,\frac{\mathbb{C}[x,y,t]}{(xy-t)}\to \mathbb{A}^1=\mathrm{Spec}\,\mathbb{C}[t],
\een
with singular divisor $D=\mathrm{pt}$. Let 
\ben
\mathbf{m}:\mathbb{A}^2\to\mathbb{A}^1,\;\;\;\;(t_{0},t_{1})\mapsto t_{0}t_{1},
\een
and hence we have the following map
\ben
\underline{\pi}^\prime:\underline{U}\times_{\mathbb{A}^1}\mathbb{A}^2=\mathrm{Spec}\,\frac{\mathbb{C}[x,y,t_{0},t_{1}]}{(xy-t_{0}t_{1})}\to \mathbb{A}^2=\mathrm{Spec}\,\mathbb{C}[t_{0},t_{1}].
\een
We first blow up $\underline{U}\times_{\mathbb{A}^1}\mathbb{A}^2$  along the singular divisor $D\times_{\mathbb{A}^1}0$ to get $\widetilde{\underline{U}(1)}$ which is verified to be isomorphic to $\mathcal{O}_{\mathbb{P}^1\times\mathbb{P}^1}(-1,-1)$, and then contract one factor of $\mathbb{P}^1\times\mathbb{P}^1$ to obtain $\underline{U}(1):=\mathcal{O}_{\mathbb{P}^1}(-1)^{\oplus2}$. This process gives the associated projection $\underline{\pi}_{1}:\underline{U}(1)\to\mathbb{A}^2$, which is called the expanded degeneration family of the standard local model $\underline{\pi}:\underline{U}\to\mathbb{A}^1$. One can refer to [\cite{Li1}, Section 1.1] for more details.

The standard families of expanded degenerations are constructed in [\cite{Zhou1}, Section 2.3] via the associated stacks of the groupoid induced by  some $\acute{e}$tale covering of a Deligne-Mumford stack by schemes as follows. For a locally simple  degeneration
$\pi:\mathcal{X}\to \mathbb{A}^1$, if $p\in \mathcal{D}$, one can choose a local family $\pi_{p}:U_{p}\to V_{p}$, where $V_{p}$ is an $\acute{e}tale$ neighborhood of $0\in \mathbb{A}^1$, and $p\in U_{p}$ is a common $\acute{e}tale$ neighborhood of $\mathcal{X}\times_{\mathbb{A}^1}V_{p}$ and $\mathrm{Spec}\left(\frac{\mathbb{C}[x,y,t,\mathbf{z}]}{(xy-t)}\right)\times_{\mathrm{Spec}\,\mathbb{C}[t]}V_{p}$ due to Definition \ref{simple-degeneration} with $\mathbf{z}=(z_{1},\cdots,z_{n})$ and $n=\dim \mathcal{D}$.
Define
$U_{p}(1)$ to be the resulting space obtained by  restricting the above construction of expanded degeneration of the standard local model $\mathrm{Spec}\left(\frac{\mathbb{C}[x,y,t,\mathbf{z}]}{(xy-t)}\right)\to\mathrm{Spec}\,\mathbb{C}[t]$   to $U_{p}$. If $p\notin \mathcal{D}$, one can select an  $\acute{e}tale$ neighborhood $U_{p}$ of $p$ such that $U_{p}\cap \mathcal{D}=\emptyset$, and denote by $U_{p}(1)=U_{p}\times_{\mathbb{A}^1}\mathbb{A}^{2}$. Then $\{U_{p}|p\in \mathcal{X}\}$ provides an  $\acute{e}tale$
covering $U=\coprod_{p}U_{p}\to \mathcal{X}$. And  two natural maps $U(1)\to U$ and $U(1)\to\mathbb{A}^2$ are  defined by using the disjoint union of all $U_{p}(1)$.
Define the following two relations:
\ben
R:=U\times_{\mathcal{X}}U\rightrightarrows
U,\;\;\;\;R(1):=R\times_{q_{1},U}U(1)\rightrightarrows U(1),
\een
where in the first relation the upper and lower arrows are denoted by two projection maps $q_{1}$ and $q_{2}$, and   the upper and lower arrows in the second relation are defined by $\mathrm{pr}_{2}:R(1):=R\times_{q_{1},U}U(1)\rightarrow U(1)$ and $\mathrm{pr}_{2}\circ i$.  Here $i$ is the map $R(1)\to R(1)$ induced by switching two factors of $R:=U\times_{\mathcal{X}}U$. Then we have $\mathcal{X}\cong[R\rightrightarrows U]$ as a stack-theoretic quotient.
It  is shown in [\cite{Zhou1}, Section 2.3] that the stack-theoretic quotient $\mathcal{X}(1):=[R(1)\rightrightarrows U(1)]$ is a Deligne-Mumford stack of finite type and we have the following commutative diagram
\ben
\xymatrix{
	\mathcal{X}(1)\ar[d]_{\pi_{1}} \ar[r]^{p} &  \mathcal{X} \ar[d]^{\pi}  \\
	\mathbb{A}^2 \ar[r]^{\mathbf{m}}& \mathbb{A}^1 &
}
\een
where $p:\mathcal{X}(1)\to\mathcal{X}$ is the projection, and 
the map $\pi_{1}:\mathcal{X}(1)\to \mathbb{A}^2$ is called the standard family of length 1-expanded degenerations with respect to $\mathcal{X}\to\mathbb{A}^1$. Now $U$ and $U(1)$ can be glued globally on the stack $\mathcal{X}$ and $\mathcal{X}(1)$ respectively. Assume that $\pi_{k}:\mathcal{X}(k)\to\mathbb{A}^{k+1}$ is constructed, $\pi_{k+1}:\mathcal{X}(k+1)\to\mathbb{A}^{k+2}$ can be constructed by induction (starting with $k=1$) as follows. Let $\mathrm{pr}_{k+1}: \mathbb{A}^{k+1}\to\mathbb{A}^1$ be the projection which maps to the last factor. Since the composition map
$\mathcal{X}(k)\xrightarrow{\pi_{k}} \mathbb{A}^{k+1}\xrightarrow{\mathrm{pr}_{k+1}}\mathbb{A}^1$  is also a locally simple degeneration, by repeating the above argument for constructing $\mathcal{X}(1)\to\mathbb{A}^2$ from $\mathcal{X}\to\mathbb{A}^1$, one has the map $\overline{\pi}_{k+1}:\mathcal{X}(k)(1)\to\mathbb{A}^{2}$. Now we define $\mathcal{X}(k+1):=\mathcal{X}(k)(1)$ and the map 
\ben
\pi_{k+1}:=(\mathrm{pr}_{[k]}\circ \pi_{k}\circ p_{k},\overline{\pi}_{k+1}):\mathcal{X}(k+1)\to\mathbb{A}^k\times\mathbb{A}^2\cong\mathbb{A}^{k+2}
\een
where $p_{k}:\mathcal{X}(k)(1)\to\mathcal{X}(k)$ is the projection map and  $\mathrm{pr}_{[k]}:\mathbb{A}^{k+1}\to\mathbb{A}^k$ is the projection mapping to the first $k$ factors.
Then for any $k\geq1$, we have the standard family of length $k$-expanded degenerations
$\pi_{k}:\mathcal{X}(k)\to \mathbb{A}^{k+1}$ with respect to $\mathcal{X}\to\mathbb{A}^1$ and  the following commutative diagram
\be \label{sim-deg-comm}
\xymatrix{
	\mathcal{X}(k)\ar[d]_{\pi_{k}} \ar[r]^{p} &  \mathcal{X}\ar[d]^{\pi}  \\  
	\mathbb{A}^{k+1} \ar[r]^{\mathbf{m}}& \mathbb{A}^1 &
}
\ee
where $p:\mathcal{X}(k)\to\mathcal{X}$ is the projection and  $\mathbf{m}:\mathbb{A}^{k+1}\to\mathbb{A}^1$ is defined as $(t_{0},\cdots, t_{k})\mapsto t_{0}t_{1}\cdots t_{k}$.

For the standard families of expanded pairs, they are constructed by two equivalent approaches in [\cite{Zhou1}, Section 2.3] as follows. Let $(\mathcal{Y}, \mathcal{D})$ be a locally  smooth pair. One approach is applying $\mathcal{X}(k)$-construction to the locally simple degeneration $\pi_{1}:\mathcal{Y}(1):=\mathrm{Bl}_{\mathcal{D}\times 0}(\mathcal{Y}\times \mathbb{A}^1)\to\mathbb{A}^1$ with the central fiber $\mathcal{Y}\cup_{\mathcal{D}}\Delta$. Define $\mathcal{Y}(k):=\mathcal{Y}(1)(k-1)$ for $k\geq2$. Let $\mathcal{D}(1)$ be the proper transform of $\mathcal{D}\times\mathbb{A}^1$, and $\mathcal{D}(k)\subset\mathcal{Y}(k)$ be the base change of $\mathcal{D}(1)\subset\mathcal{Y}(1)$ in the following commutative diagram
\ben
\xymatrix{
	\mathcal{Y}(1)(k-1)=\mathcal{Y}(k)\ar[d]_{\pi_{k}} \ar[r] &  	\mathcal{Y}(1)\ar[d]^{\pi_{1}}  \\
	\mathbb{A}^{k} \ar[r]^{\mathbf{m}}& \mathbb{A}^1
}
\een 
Then $(\mathcal{Y}(k), \mathcal{D}(k))$ is a locally smooth pair for any $k\geq1$.
The second approach is to take successive blowup construction, i.e., define  $\mathcal{Y}(k+1)=\mathrm{Bl}_{\mathcal{D}(k)\times0}(\mathcal{Y}(k)\times\mathbb{A}^1)$ inductively from the initial case $\mathcal{Y}(1):=\mathrm{Bl}_{\mathcal{D}\times 0}(\mathcal{Y}\times \mathbb{A}^1)$ where $\mathcal{D}(k+1)\subset\mathcal{Y}(k+1)$ is the proper transform of $\mathcal{D}(k)\times\mathbb{A}^1$ for $k\geq1$. These two approaches are shown to be equivalent in [\cite{Zhou1}, Proposition 2.11].  The obtained family $\pi_{k}: \mathcal{Y}(k)\to\mathbb{A}^k$ is called the standard family of length $k$-expanded pairs with respect to $(\mathcal{Y}, \mathcal{D})$ and $\mathcal{D}(k)$ is called the distinguished divisor of  $\mathcal{Y}(k)$  where $\mathcal{D}(k)\cong\mathcal{D}\times\mathbb{A}^k$. For the later use, there is a standard family of expanded pairs $(\mathcal{Y}(k)^\circ,\pi_{k}^\circ)$ defined in [\cite{Zhou1}, Remark 2.12] where $\mathcal{Y}(k)^\circ:=\mathcal{Y}(k)$ and $\pi_{k}^\circ:=\mathbf{r}\circ \pi_{k}$ with $\mathbf{r}:\mathbb{A}^k\to\mathbb{A}^k$ defined by $(t_{1},\cdots,t_{k})\mapsto(t_{k},\cdots,t_{1})$.

Next, we recollect the properties of standard families of expanded degenerations and pairs in [\cite{Zhou1}, Section 2.4] as follows. For the family $\pi_{k}:\mathcal{X}(k)\to\mathbb{A}^{k+1}$, we define the $(\mathbb{C}^*)^k$-action on $\mathbb{A}^{k+1}$ by
\ben
\lambda\cdot t=(\lambda_{1}t_{0},\lambda_{1}^{-1}\lambda_{2}t_{1},\cdots,\lambda_{k-1}^{-1}\lambda_{k}t_{k-1},\lambda_{k}^{-1}t_{k})
\een
where $\lambda=(\lambda_{1},\cdots,\lambda_{k})\in(\mathbb{C}^*)^k$ and $t=(t_{0},\cdots,t_{k})\in\mathbb{A}^{k+1}$. 
For the base $\mathbb{A}^{k+1}$  with coordinates $(t_{0},\cdots,t_{k})$ and  a subset $I=\{j_{0},\cdots,j_{l}\}\subset\{0,\cdots,k\}$ ($j_{0}<\cdots<j_{l}$), define the embedding $\tau_{I}: \mathbb{A}^{l+1}\to\mathbb{A}^{k+1}$ as 
\[
t_{j}=
\left\{
\begin{aligned}
&t_{j_{i}}, \;\;\mbox{if $j=j_{i}$ for some $i$}, \\
& 1, \;\;\;\;\mbox{if $j\neq j_{i}$ for any $i$}.
\end{aligned}
\right.\]
Set $U_{I}:=\{(t_{0},\cdots,t_{k})\,|\, t_{j}\neq0, \;\forall\; j\notin I\}$, then $U_{I}\cong\mathbb{A}^{l+1}\times(\mathbb{C}^*)^{k-l}$ and we have the natural open immersion $\widehat{\tau}_{I}: U_{I}\to\mathbb{A}^{k+1}$. Given two subsets $I, I^\prime\in\{0,\cdots,k\}$ with $|I|=|I^\prime|$, we have the isomorphism $\widehat{\tau}_{I,I^\prime}: U_{I}\xrightarrow{\cong} U_{I^\prime}$ by reordering the coordinates.
While for the family $\pi_{k}:\mathcal{Y}(k)\to\mathbb{A}^k$, we define the $(\mathbb{C}^*)^k$-action on $\mathbb{A}^k$ for the first approach by
\ben
\lambda\cdot t=(\lambda_{1}t_{1},\lambda_{1}^{-1}\lambda_{2}t_{2},\cdots,\lambda_{k-1}^{-1}\lambda_{k}t_{k})
\een
and for the second approach by
\ben
\lambda\cdot t=(\lambda_{1}t_{1},\lambda_{2}t_{2},\cdots,\lambda_{k}t_{k})
\een
where $\lambda=(\lambda_{1},\cdots,\lambda_{k})\in(\mathbb{C}^*)^k$ and $t=(t_{1},\cdots,t_{k})\in\mathbb{A}^{k}$.  For the base $\mathbb{A}^k$, one can similarly define  $\tau_{I}, \widehat{\tau}_{I}$ and $\widehat{\tau}_{I,I^\prime}$ as above with the  index set $\{0,\cdots,k\}$ replaced by $\{1,\cdots,k\}$.
For the standard families of expanded degenerations $\pi_{k}:\mathcal{X}(k)\to\mathbb{A}^{k+1}$, we have the following properties.
\begin{proposition}([\cite{Zhou1}, Proposition 2.13]) \label{property1}
$(i)$ The stack $\mathcal{X}(k)$ is a separated Deligne-Mumford stack of finite type over $\mathbb{C}$ and of dimension $\dim\mathcal{X}_{0}+k+1$. Moreover, if $\pi:\mathcal{X}\to\mathbb{A}^1$ is proper, $\pi_{k}$ is proper; if $\mathcal{X}$ is smooth, $\mathcal{X}(k)$ is smooth.

$(ii)$ For $t=(t_{0},\cdots,t_{k})\in\mathbb{A}^{k+1}$, if $t_{0}\cdots t_{k}\neq0$, the fiber of $\pi_{k}$ over $t$ is isomorphic to the generic fibers $\mathcal{X}_{c}$ of $\pi:\mathcal{X}\to\mathbb{A}^1$ over $0\neq c\in\mathbb{A}^1$. Let $I\subset\{0,\cdots,k\}$ be a subset with $1\leq|I|=l+1\leq k+1$. If $t_{j}=0$ for any $j\in I$ and $t_{j}\neq0$ for all $j\notin I$, then the fiber over $t$ is isomorphic to $\mathcal{X}_{0}[l]$.

$(iii)$ There  is a $(\mathbb{C}^*)^k$-action on $\mathcal{X}(k)$ such that the diagram \eqref{sim-deg-comm} is $(\mathbb{C}^*)^k$-equivariant. This action gives isomorphisms of fibers in the same strata as described in $(ii)$. The induced action of the stabilizer on a fiber isomorphic to $\mathcal{X}_{0}[l]$ is the induced $(\mathbb{C}^*)^l$-action on $\mathcal{X}_{0}[l]$: the $i$-th factor of $(\mathbb{C}^*)^l$ acts on $\Delta_{i}$ fiberwise and trivially on $\mathcal{Y}_{\pm}$. There are also discrete symmetries in $\mathcal{X}(k)$, away from the central fiber. For two subsets $I, I^\prime\subset\{0,\cdots,k\}$ with $|I|=|I^\prime|=l+1$, the natural isomorphisms $\widehat{\tau}_{I, I^\prime}: U_{I}\xrightarrow{\cong} U_{I^\prime}$ induce isomorphisms on the families $\widehat{\tau}_{I, I^\prime,\mathcal{X}}: \mathcal{X}(k)|_{U_{I}}\xrightarrow{\cong}\mathcal{X}(k)|_{U_{I^\prime}}$ extending those given by $(\mathbb{C}^*)^k$-actions on smooth fibers. Moreover, we have $\mathcal{X}(k)\times_{\mathbb{A}^{k+1},\tau_{I}}\mathbb{A}^{l+1}\cong\mathcal{X}(l)$ after restricting to the embedding $\tau_{I}:\mathbb{A}^{l+1}\to\mathbb{A}^{k+1}$ and we have the identification $\mathcal{X}(k)|_{U_{I}}\cong \mathcal{X}(l)\times(\mathbb{C}^*)^{k-l}$.
\end{proposition}
For the standard families of expanded pairs $\pi_{k}:\mathcal{Y}(k)\to\mathbb{A}^k$, we have the properties as follows.
\begin{proposition}([\cite{Zhou1}, Proposition 2.16])\label{property2}
$(i)$ The stack $\mathcal{Y}(k)$ is a separated Deligne-Mumford stack of finite type over $\mathbb{C}$ and of dimension $\dim\mathcal{Y}+k$. If $\mathcal{Y}$ is proper, then $\pi_{k}$ is proper; if $(\mathcal{Y},\mathcal{D})$ is a smooth pair, so is $(\mathcal{Y}(k),\mathcal{D}(k))$.

$(ii)$ For $t=(t_{1},\cdots,t_{k})\in\mathbb{A}^k$, if $t_{1}\cdots t_{k}\neq0$, then the fiber of the pair $(\mathcal{Y}(k),\mathcal{D}(k))$ over $t$ is isomorphic to the original pair $(\mathcal{Y},\mathcal{D})$. Let $I\subset\{1,\cdots,k\}$ be a subset with $1\leq|I|=l\leq k$. If $t_{j}=0$ for all $j\in I$ and $t_{j}\neq0$ for all $j\notin I$, then the fiber over $t$ is isomorphic to $(\mathcal{Y}[l],\mathcal{D}[l])$.

$(iii)$ There is a $(\mathbb{C}^*)^k$-action on $\mathcal{Y}(k)$ in both approaches, compatible with the corresponding actions on $\mathbb{A}^k$. This action gives isomorphisms of fibers in the same strata as described in $(ii)$. The induced action of the stabilizer on a fiber isomorphic to $\mathcal{Y}[l]$ is the induced $(\mathbb{C}^*)^l$-action on $\mathcal{Y}[l]$: the $i$-th factor of $(\mathbb{C}^*)^l$ acts on $\Delta_{i}$ fiberwise and trivially on $\mathcal{Y}$. There are also discrete symmetries in $\mathcal{Y}(k)$, away from the central fiber. For two subsets $I, I^\prime\subset\{1,\cdots,k\}$ with $|I|=|I^\prime|=l$, the natural isomorphisms  $\widehat{\tau}_{I, I^\prime}: U_{I}\xrightarrow{\cong} U_{I^\prime}$ induce isomorphisms on the families $\widehat{\tau}_{I, I^\prime, \mathcal{Y}}: \mathcal{Y}(k)|_{U_{I}}\xrightarrow{\cong}\mathcal{Y}(k)|_{U_{I^\prime}}$ extending those given by $(\mathbb{C}^*)^k$-actions on smooth fibers. The discrete actions restricted to $\mathcal{D}(k)\cong\mathcal{D}\times\mathbb{A}^k$ are just a reordering of the coordinates of the $\mathbb{A}^k$-factors. Moreover, we have $\mathcal{Y}(k)\times_{\mathbb{A}^k,\tau_{I}}\mathbb{A}^l\cong\mathcal{Y}(l)$ after restricting to the embedding $\tau_{I}:\mathbb{A}^l\to\mathbb{A}^k$ and we have the identification $\mathcal{Y}(k)|_{U_{I}}\cong\mathcal{Y}(l)\times(\mathbb{C}^*)^{k-l}$.
\end{proposition}


There are some relations between standard families of expanded degenerations and standard families of expanded pairs as follows, which are important for formulating the degeneration formula later.

\begin{proposition}([\cite{Zhou1}, Proposition 2.18])\label{decom}
Let $\pi_{k}:\mathcal{X}(k)\to\mathbb{A}^{k+1}$ be the standard family of length $k$-expanded degenerations and $\mathrm{pr}_{i}:\mathbb{A}^{k+1}\to\mathbb{A}^1$ be the projection to the factor $t_{i}$ $($$0\leq i\leq k$$)$. Then the composition $\pi_{k,i}:=\mathrm{pr}_{i}\circ\pi_{k}:\mathcal{X}(k)\to\mathbb{A}^1$ is a locally simple degeneration. Let $\mathcal{D}_{i}(k)$ be the singular divisor of $\pi_{k,i}$ such that $\mathcal{X}(k)\times_{\pi_{k,i},\mathbb{A}^1}0:=\mathcal{X}(k)_{-}^i\cup_{\mathcal{D}_{i}(k)}\mathcal{X}(k)_{+}^i$ is the central fiber decomposition. Then we have
\ben
\mathcal{D}_{i}(k)\cong\mathbb{A}^i\times\mathcal{D}\times\mathbb{A}^{k-i},\;\;\;\mathcal{X}(k)_{-}^i\cong\mathcal{Y}_{-}(i)\times\mathbb{A}^{k-i},\;\;\;\mathcal{X}(k)_{+}^i\cong\mathbb{A}^i\times\mathcal{Y}_{+}(k-i)^\circ.
\een
Moreover, the restrictions of $\pi_{k}:\mathcal{X}(k)\to\mathbb{A}^{k+1}$ to  $\mathcal{X}(k)_{-}^i$ and $\mathcal{X}(k)_{+}^i$ are given as 
\ben
\xymatrix{
	 & \mathcal{X}(k)_{-}^i\ar[d]^{\pi_{k,i}}\ar[dl]_{\pi_{k,\{0,\cdots,i-1\}}} \ar[rd]^{\mathrm{Id}_{\mathbb{A}^{k-i}}} &    &  &\mathcal{X}(k)_{+}^i\ar[d]^{\pi_{k,i}} \ar[dl]_{\mathrm{Id}_{\mathbb{A}^i}}\ar[rd]^{\pi_{k,\{i+1,\cdots,k\}}^\circ}& \\
\mathbb{A}^i	&  0 &\mathbb{A}^{k-i}, &  \mathbb{A}^i & 0   &\mathbb{A}^{k-i}
}
\een
where the corresponding maps from $\mathcal{Y}_{-}(i)$ to $\mathbb{A}^i$ and  from $\mathcal{Y}_{+}(k-i)^\circ$ to $\mathbb{A}^{k-i}$ are the standard families of expanded pairs. Here, the gluing is along  $\mathcal{D}_{i}(k)$ via isomorphisms $\mathcal{D}_{-}(i)\times\mathbb{A}^{k-i}\cong\mathbb{A}^i\times\mathcal{D}\times\mathbb{A}^{k-i}\cong\mathbb{A}^i\times\mathcal{D}_{+}(k-i)^\circ$, and $\mathcal{D}_{i}(k)\cap\mathcal{D}_{j}(k)=\emptyset$ for $i\neq j$.

All the statements above are compatible with the corresponding $(\mathbb{C}^*)^k$-actions.
\end{proposition}

Now, we briefly recall the definition of stacks of expanded degenerations and pairs in [\cite{Zhou1}, Section 2.5] as follows.
Let $\pi:\mathcal{X}\to\mathbb{A}^1$ be a locally simple degeneration, and let  $I, I^\prime\subset\{0,\cdots,k\}$ be two subsets  with $|I|=|I^\prime|=l+1$. 
We first consider the following $\acute{e}$tale equivalence relations:
\ben
R_{\mathrm{discrete},\mathbb{A}^{k+1}}:=\coprod_{1\leq|I|=|I^\prime|\leq k+1}R_{I,I^\prime,\mathbb{A}^{k+1}},\;\;\;\; R_{\mathrm{discrete},\mathcal{X}(k)}:=\coprod_{1\leq|I|=|I^\prime|\leq k+1}R_{I,I^\prime,\mathcal{X}(k)}
\een
where $R_{I,I^\prime,\mathbb{A}^{k+1}}:=\mathbb{A}^{l+1}\times(\mathbb{C}^*)^{k-l}\rightrightarrows\mathbb{A}^{k+1}$ and $R_{I,I^\prime,\mathcal{X}(k)}:=\mathcal{X}(l)\times(\mathbb{C}^*)^{k-l}\rightrightarrows\mathcal{X}(k)$ are the discrete relations induced by open immersions $\mathbb{A}^{l+1}\times(\mathbb{C}^*)^{k-l}\xrightarrow{\cong}U_{I}\hookrightarrow\mathbb{A}^{k+1}$, $\mathbb{A}^{l+1}\times(\mathbb{C}^*)^{k-l}\xrightarrow{\cong}U_{I^\prime}\hookrightarrow\mathbb{A}^{k+1}$ and $\mathcal{X}(l)\times(\mathbb{C}^*)^{k-l}\xrightarrow{\cong}\mathcal{X}(k)|_{U_{I}}\hookrightarrow\mathcal{X}(k)$, $\mathcal{X}(l)\times(\mathbb{C}^*)^{k-l}\xrightarrow{\cong}\mathcal{X}(k)|_{U_{I^\prime}}\hookrightarrow\mathcal{X}(k)$ (see  Proposition \ref{property1}) respectively.
Then we have the discrete relation $R_{\mathrm{discrete},\mathbb{A}^{k+1}}\rightrightarrows\mathbb{A}^{k+1}$, which can be viewed as a subrelation of the $S_{k+1}$-action as follows
\ben
R_{\mathrm{discrete},\mathbb{A}^{k+1}}\hookrightarrow S_{k+1}\times\mathbb{A}^{k+1}\rightrightarrows\mathbb{A}^{k+1}
\een
where $S_{k+1}$ is the symmetric group acting on $\mathbb{A}^{k+1}$ by permutation. Using the  following map
\ben
(\mathbb{C}^*)^k\hookrightarrow(\mathbb{C}^*)^{k+1},\;\;\;\;(\lambda_{1},\cdots,\lambda_{k})\mapsto(\lambda_{1},\lambda_{1}^{-1}\lambda_{2},\cdots,\lambda_{k-1}^{-1}\lambda_{k},\lambda_{k}^{-1}),
\een
one can view the $(\mathbb{C}^*)^k$-action on $\mathbb{A}^{k+1}$ as some  $(\mathbb{C}^*)^{k+1}$-action on $\mathbb{A}^{k+1}$ via the above map. When $k=0$, we have the trivial action on $\mathbb{A}^1$. Then the smooth equivalence relation generated by the $(\mathbb{C}^*)^k$-action and discrete relations is 
\ben
R_{\sim,\mathbb{A}^{k+1}}:=(\mathbb{C}^*)^k\times R_{\mathrm{discrete},\mathbb{A}^{k+1}}\hookrightarrow(\mathbb{C}^*)^{k+1}\rtimes S_{k+1}\times\mathbb{A}^{k+1}\rightrightarrows\mathbb{A}^{k+1},
\een
where the semidirect product $(\mathbb{C}^*)^{k+1}\rtimes S_{k+1}$ is induced by the group $S_{k+1}\subset \mathrm{GL}(k+1)$ acting on $(\mathbb{C}^*)^{k+1}$ by conjugation.
Similarly, one can also define the smooth equivalence relation $R_{\sim,\mathcal{X}(k)}$ on $\mathcal{X}(k)$. Now we have two Artin stacks $[\mathbb{A}^{k+1}/R_{\sim, \mathbb{A}^{k+1}}]$ and $[\mathcal{X}(k)/R_{\sim,\mathcal{X}(k)}]$ with the induced $1$-morphism $\pi:[\mathcal{X}(k)/R_{\sim,\mathcal{X}(k)}]\to[\mathbb{A}^{k+1}/R_{\sim, \mathbb{A}^{k+1}}]$. And we  have embeddings (also open immersions) of Artin stacks $[\mathbb{A}^{l+1}/R_{\sim,\mathbb{A}^{l+1}}]\to[\mathbb{A}^{k+1}/R_{\sim,\mathbb{A}^{k+1}}]$ and $[\mathcal{X}(l)/R_{\sim,\mathcal{X}(l)}]\to[\mathcal{X}(k)/R_{\sim,\mathcal{X}(k)}]$ induced by embeddings $\tau_{I}:\mathbb{A}^{l+1}\hookrightarrow\mathbb{A}^{k+1}$ and $\tau_{I,\mathcal{X}}: \mathcal{X}(l)\cong\mathcal{X}(k)\times_{\mathbb{A}^{k+1},\tau_{I}}\mathbb{A}^{l+1}\hookrightarrow\mathcal{X}(k)$ respectively. Then we have the following defnition of inductive limits.
\begin{definition}([\cite{Zhou1}, Definition 2.22])\label{degen-stacks}
We define
$\mathfrak{C}:=\lim\limits_{\longrightarrow}[\mathbb{A}^{k+1}/R_{\sim,\mathbb{A}^{k+1}}]$ to be the stack of expanded degenerations with respect to the locally simple degeneration $\pi:\mathcal{X}\to\mathbb{A}^1$, and  $\mathfrak{X}:=\lim\limits_{\longrightarrow}[\mathcal{X}(k)/R_{\sim,\mathcal{X}(k)}]
$  to be the universal family of expanded degenerations. We have the following commutative (not Cartesian) diagram
\ben
\xymatrix{
	\mathfrak{X}\ar[d]_{\pi} \ar[r]^{p} &  \mathcal{X} \ar[d]^{\pi}  \\
	\mathfrak{C} \ar[r]^{p}& \mathbb{A}^1 &
}
\een
where the family map $\pi:\mathfrak{X}\to\mathfrak{C}$  is of Deligne-Mumford type and proper.
\end{definition}

For a locally smooth pair $(\mathcal{Y},\mathcal{D})$, one can similarly define smooth equivalence relations $R_{\sim,\mathbb{A}^k}$ on $\mathbb{A}^k$ and $R_{\sim,\mathcal{Y}(k)}$ on $\mathcal{Y}(k)$, and then lead to the following  definition.
\begin{definition}([\cite{Zhou1}, Definition 2.20])\label{pair-stacks}
We define $\mathfrak{A}:=\lim\limits_{\longrightarrow}[\mathbb{A}^{k}/R_{\sim,\mathbb{A}^{k}}]$ to be the stack of expanded pairs with respect to $(\mathcal{Y},\mathcal{D})$, and $\mathfrak{Y}:=\lim\limits_{\longrightarrow}[\mathcal{Y}(k)/R_{\sim,\mathcal{Y}(k)}]$ to be the universal family of expanded pairs. There is a family map $\pi:\mathfrak{Y}\to\mathfrak{A}$, which is of Deligne-Mumford type and proper.
\end{definition}	

The above two definitions have the  categorical interpretations  in [\cite{Zhou1}, Section 2.5] as follows. For a given $\mathbb{A}^1$-map $\xi: S\to\mathbb{A}^{k+1}$, let $\mathtt{X}_{S}:=\xi^*\mathcal{X}(k)$, we call $\pi: \mathtt{X}_{S}\to S$ a family of expanded degenerations over $(S,\xi)$. For the map $\xi$ and  another map obtained from acting on $\xi:S\to\mathbb{A}^{k+1}$  by the equivalence relation $R_{\sim,\mathbb{A}^{k+1}}$, they induce  isomorphic families of expanded degenerations. Given a morphism of $\mathbb{A}^1$-schemes $f:T\to S$ and a $\mathbb{A}^1$-map $\xi: S\to\mathbb{A}^{k+1}$, we have the $\mathbb{A}^1$-map $\xi\circ f: T\to\mathbb{A}^{k+1}$ and the  corresponding family of expanded degenerations $\mathtt{X}_{T}=f^*\mathtt{X}_{S}$ over $T$, which is unique up to $2$-isomorphisms by  the similar argument in the proof of [\cite{Zhou1}, Lemma 2.21]. Similarly, given a map $\xi: S\to\mathbb{A}^k$,  we call $\pi:\mathtt{Y}_{S}:=\xi^*\mathcal{Y}(k)\to S$ a family of expanded pairs over $(S,\xi)$. For a given morphism $f: T\to S$,  the pull-back family of expanded pairs $\mathtt{Y}_{T}=f^*\mathtt{Y}_{S}$ is unique up to $2$-isomorphisms.

\subsection{Admissibility of  sheaves and its numerical criterion}
The admissible sheaves on singular schemes are introduced in [\cite{LW}, Section 3.2], which is an important condition imposed on the objects in moduli stacks of stable quotients and stable pairs (or coherent systems). And there is a numerical criterion to measure the failure of the admissibility of a coherent sheaf in [\cite{LW}, Section 3.3]. In this subsection, we briefly recall the generalized ones for the case of Deligne-Mumford stacks in [\cite{Zhou1}, Section 3 and 5].
Let $\mathcal{U}$ be a separated Deligne-Mumford stack of finite type and $\mathcal{W}\subseteq\mathcal{U}$ be a closed substack.
\begin{definition}([\cite{Zhou1}, Definition 3.1])
We say that a coherent sheaf $\mathcal{G}$ on $\mathcal{U}$ is normal to 	$\mathcal{W}$ if $\mathrm{Tor}_{1}^{\mathcal{O}_{\mathcal{U}}}(\mathcal{G},\mathcal{O}_{\mathcal{W}})=0$. Moreover, $\mathcal{G}$ is said to be normal to $\mathcal{W}$ at a point $p\in\mathcal{W}$ if there is an $\acute{e}$tale neighborhood $i: U_{p}\to\mathcal{U}$ of $p$ such that $i^*\mathcal{G}$ is normal to $\mathcal{W}\times_{\mathcal{U}}U_{p}$.
\end{definition}	
It is shown in [\cite{Zhou1}, Lemma 3.2] that normality is a local property in the $\acute{e}$tale topology. Now, admissible sheaves are defined on (families of) expanded degenerations  and  pairs as follows.
\begin{definition}([\cite{Zhou1}, Definition 3.13])
	A coherent sheaf $\mathcal{G}$ on $\mathcal{X}_{0}[k]$ (resp. $\mathcal{Y}[k]$) is called admissible if it is normal to each $\mathcal{D}_{i}\subset \mathcal{X}_{0}[k]$ (resp. $\mathcal{D}_{i}\subset \mathcal{Y}[k]$) for $0\leq i\leq k$.
	
	Let $\mathtt{X}_{S}\to S$ (resp. $\mathtt{Y}_{S}\to S$) be a family of expanded degenerations (resp. pairs), and let $\mathcal{G}$ be a coherent sheaf on  $\mathtt{X}_{S}$ (resp. $\mathtt{Y}_{S}$) flat over $S$. We say that $\mathcal{G}$ is admissible if for every point  $s\in S$, the fiber $\mathcal{G}_{s}$ is admissible on the fiber $\mathtt{X}_{S,s}$ (resp. $\mathtt{Y}_{S,s}$).
\end{definition}
Admissibility of sheaves is an open condition in the following sense.
\begin{proposition}([\cite{Zhou1}, Proposition 3.16])\label{adm-open}
Let $\mathtt{X}_{S}\to S$ (resp. $\mathtt{Y}_{S}\to S$) be a family of expanded degenerations (resp. pairs), and let $\mathcal{G}$ be a coherent sheaf on  $\mathtt{X}_{S}$ (resp. $\mathtt{Y}_{S}$) and  flat over $S$. Then the set $\{s\in S\;|\; \mathcal{G}_{s} \mbox{ is admissible}\}$ is open in $S$.
\end{proposition}	

Next, we briefly recall the numerical criterion defined in [\cite{Zhou1}, Section 5.3].
Suppose we have the central fiber $\mathcal{X}_{0}=\mathcal{Y}_{-}\cup_{\mathcal{D}}\mathcal{Y}_{+}$ and a smooth pair $(\mathcal{Y},\mathcal{D})$ as in Definition \ref{simple-degeneration}. Assume that $\mathcal{X}_{0}$ and $\mathcal{Y}$ are also projective Deligne-Mumford stacks with moduli schemes $\pi:\mathcal{X}_{0}\to X_{0}$ and $\pi:\mathcal{Y}\to Y$ with polarizations $(\mathcal{E}_{\mathcal{X}_{0}},\mathcal{O}_{X_{0}}(1))$ and $(\mathcal{E}_{\mathcal{Y}},\mathcal{O}_{Y}(1))$ respectively. Let $\mathcal{G}$  be a coherent sheaf on $\mathcal{U}$ and $\mathcal{J}\subset\mathcal{O}_{\mathcal{U}}$ be the ideal sheaf of an effective Cartier divisor $\mathcal{W}\subset\mathcal{U}$. As in [\cite{Zhou1}, Section 3.1],  for an affine $\acute{e}$tale neighborhood $U=\mathrm{Spec}\,A$ of a point in $\mathcal{W}$, the sheaf $\mathcal{G}$ on $U$ is represented by an $A$-module $M$. The sheaf $\mathcal{G}_{\mathcal{J}}$ is defined  locally in an affine neighborhood $U$ with $\mathcal{J}(U)=J$ such that
\ben
\mathcal{G}_{\mathcal{J}}(U):=\{m\in M|\mathrm{ann}(m)\supset J^k \mbox{ for some } k\in\mathbb{Z}_{+}\},
\een
which is the maximal subsheaf of $\mathcal{G}$ supported on $\mathcal{W}$. 
Now, let $\mathcal{H}$ be a coherent sheaf on $\mathcal{X}_{0}[k]$.
Suppose $\mathcal{I}_{i}^{-}$  is the ideal sheaf of $\mathcal{D}_{i-1}\subset\Delta_{i}$ and $\mathcal{I}_{i}^{+}$ is the ideal sheaf of $\mathcal{D}_{i}\subset\Delta_{i}$. Let $\mathcal{I}_{0}^{+}$ be the ideal sheaf of $\mathcal{D}_{0}\subset\mathcal{Y}_{-}$ and $\mathcal{I}_{k+1}^{-}$ be the ideal sheaf of $\mathcal{D}_{k}\subset\mathcal{Y}_{+}$. Assume that $\mathcal{J}_{i}$ is the ideal sheaf of $\mathcal{D}_{i}\subset\mathcal{X}_{0}[k]$. Define $\mathcal{H}^{t.f.}$ to be the quotient in the following short exact sequence
\ben
0\to\bigoplus_{i=0}^{k}\mathcal{H}_{\mathcal{J}_{i}}\to\mathcal{H}\to\mathcal{H}^{t.f.}\to0.
\een
One can similarly define $\mathcal{H}^{t.f.}$ for a coherent sheaf $\mathcal{H}$ on $\mathcal{Y}[k]$, see also [\cite{Zhou1}, Section 5.3].
\begin{definition}([\cite{Zhou1}, Definition 5.12])
For a  sheaf $\mathcal{H}$ on $\mathcal{X}_{0}[k]$, define the $i$-th error of $\mathcal{H}$ by
\ben
\mathrm{Err}_{i}(\mathcal{H})(v):&=&P_{\mathcal{H}_{\mathcal{J}_{i}}}^{\mathcal{E}_{\mathcal{X}_{0}}}(H^{\otimes v})+P_{(\mathcal{H}^{t.f.}|_{\Delta_{i}})_{\mathcal{I}_{i}^{+}}}^{\mathcal{E}_{\mathcal{X}_{0}}}(H^{\otimes v})+P_{(\mathcal{H}^{t.f.}|_{\Delta_{i+1}})_{\mathcal{I}_{i+1}^{-}}}^{\mathcal{E}_{\mathcal{X}_{0}}}(H^{\otimes v})\\
&&-\frac{1}{2}P_{((\mathcal{H}^{t.f.}|_{\Delta_{i}})_{\mathcal{I}_{i}^{+}})|_{\mathcal{D}_{i}}}^{\mathcal{E}_{\mathcal{X}_{0}}}(H^{\otimes v})-\frac{1}{2}P_{((\mathcal{H}^{t.f.}|_{\Delta_{i+1}})_{\mathcal{I}_{i+1}^{-}})|_{\mathcal{D}_{i}}}^{\mathcal{E}_{\mathcal{X}_{0}}}(H^{\otimes v})
\een
for $0\leq i\leq k$, where $H=\pi^*\mathcal{O}_{X_{0}}(1)$ and $P^{\mathcal{E}_{\mathcal{X}_{0}}}_{\bullet}:K^0(\mathcal{X}_{0})\to\mathbb{Z}$ is defined in Section 5.3.

If $\mathcal{H}$ is a sheaf on $\mathcal{Y}[k]$, define the $i$-th error of $\mathcal{H}$ by
\ben
\mathrm{Err}_{i}(\mathcal{H})(v):&=&P_{\mathcal{H}_{\mathcal{J}_{i}}}^{\mathcal{E}_{\mathcal{Y}}}(H^{\otimes v})+P_{(\mathcal{H}^{t.f.}|_{\Delta_{i}})_{\mathcal{I}_{i}^{+}}}^{\mathcal{E}_{\mathcal{Y}}}(H^{\otimes v})+P_{(\mathcal{H}^{t.f.}|_{\Delta_{i+1}})_{\mathcal{I}_{i+1}^{-}}}^{\mathcal{E}_{\mathcal{Y}}}(H^{\otimes v})\\
&&-\frac{1}{2}P_{((\mathcal{H}^{t.f.}|_{\Delta_{i}})_{\mathcal{I}_{i}^{+}})|_{\mathcal{D}_{i}}}^{\mathcal{E}_{\mathcal{Y}}}(H^{\otimes v})-\frac{1}{2}P_{((\mathcal{H}^{t.f.}|_{\Delta_{i+1}})_{\mathcal{I}_{i+1}^{-}})|_{\mathcal{D}_{i}}}^{\mathcal{E}_{\mathcal{Y}}}(H^{\otimes v})
\een
for $0\leq i\leq k-1$ and 
\ben
\mathrm{Err}_{k}(\mathcal{H})(v):=P_{(\mathcal{H}^{t.f.}|_{\Delta_{k}})_{\mathcal{I}_{k}^{+}}}^{\mathcal{E}_{\mathcal{Y}}}(H^{\otimes v})
\een
where $H:=\pi^*\mathcal{O}_{Y}(1)$ and $P^{\mathcal{E}_{\mathcal{Y}}}_{\bullet}:K^0(\mathcal{Y})\to\mathbb{Z}$ is defined in [\cite{Zhou1}, Section 4.3], see also Section 5.3.

For both cases above, define the total error by
\ben
\mathrm{Err}(\mathcal{H})(v):=\sum_{i=0}^k\mathrm{Err}_{i}(\mathcal{H})(v)
\een
\end{definition}
Since  polynomials $P^{\mathcal{E}_{\mathcal{X}_{0}}}_{\bullet}(H^{\otimes v})$ and $P^{\mathcal{E}_{\mathcal{Y}}}_{\bullet}(H^{\otimes v})$  in both cases are Hilbert polynomials due to the  ampleness of $H|_{\mathcal{D}_{i}}$, then $\mathrm{Err}_{i}(\mathcal{H})(v)\geq0$ for $v$ large enough for any $0\leq i\leq k$. And  for each $0\leq i\leq k$, a sheaf $\mathcal{H}$ on $\mathcal{X}_{0}[k]$ is normal to $\mathcal{D}_{i}$ if and only if $\mathcal{H}_{\mathcal{J}_{i}}=(\mathcal{H}^{t.f.}|_{\Delta_{i}})_{\mathcal{I}_{i}^{+}}=(\mathcal{H}^{t.f.}|_{\Delta_{i+1}})_{\mathcal{I}_{i+1}^{-}}=0$, if and only if $\mathrm{Err}_{i}(\mathcal{H})=0$ by [\cite{Zhou1}, Corollary 3.11]. The same result holds for the normality of a sheaf $\mathcal{H}$ on $\mathcal{Y}[k]$ to $\mathcal{D}_{i}$ for each $0\leq i\leq k-1$. While for the distinguished divisor $\mathcal{D}_{k}$ of $\mathcal{Y}[k]$, the sheaf $\mathcal{H}$ is normal to $\mathcal{D}_{k}$ if and only if $(\mathcal{H}^{t.f.}|_{\Delta_{k}})_{\mathcal{I}_{k}^{+}}=0$, if and only if $\mathrm{Err}_{k}(\mathcal{H})=0$  by [\cite{Zhou1}, Proposition 3.5].

Therefore, the numerical criterion is shown  that a coherent sheaf $\mathcal{H}$ on $\mathcal{X}_{0}[k]$ (resp. $\mathcal{Y}[k]$) is admissible if and only if $\mathrm{Err}_{i}(\mathcal{H})=0$ for all $0\leq i\leq k$, if and only if $\mathrm{Err}(\mathcal{H})=0$.

\section{Relative moduli spaces of semistable pairs}
In this section, we will give a construction of relative moduli spaces of semistable pairs, which is a relative version of the one in [\cite{Lyj}]. Let $p: \mathcal{X}\to S$ be a family of projective Deligne-Mumford stacks   with a moduli scheme $\pi:\mathcal{X}\to X$ and a relative polarization  $(\mathcal{E}, \mathcal{O}_{X}(1))$. Assume the stability parameter $\delta\in\mathbb{Q}[m]$ is a polynomial with positive leading coefficient or 0.
\subsection{Semistable pairs on the family of projective Deligne-Mumford stacks}

Let  $\mathcal{F}_{0}$ be a fixed $S$-flat coherent sheaf on $\mathcal{X}$.

\begin{definition}\label{def-pair}
	A pair $(\mathcal{F},\varphi)$  on $\mathcal{X}/S$ consists of  a coherent sheaf $\mathcal{F}$ on $\mathcal{X}/S$ and a morphism $\varphi:\mathcal{F}_{0}\to\mathcal{F}$.  A morphism of pairs $\phi:(\mathcal{F},\varphi)\to(\mathcal{G},\psi)$ is a  morphism of sheaves $\phi:\mathcal{F}\to\mathcal{G}$ such that
	there is an element $\lambda\in k$ making the following  diagram commute
	\ben
	\xymatrix{
		\mathcal{F}_{0}\ar[d]_{\varphi} \ar[r]^{\lambda \cdot \mathrm{id}} &  \mathcal{F}_{0} \ar[d]^{\psi}  \\
		\mathcal{F}\ar[r]_{\phi}&\mathcal{G}  &
	}
	\een
A subpair $(\mathcal{F}^\prime,\varphi^\prime)$ of $(\mathcal{F},\varphi)$ consists of a coherent subsheaf $\mathcal{F}^\prime\subset\mathcal{F}$ and a  morphism $\varphi^\prime:\mathcal{F}_{0}\to\mathcal{F}^\prime$ satisfying
	$i\circ\varphi^\prime=\varphi$ if $\mathrm{im}\varphi\subset\mathcal{F}^\prime$, and $\varphi^\prime=0$ otherwise, where $i$ denotes the inclusion $\mathcal{F}^\prime\hookrightarrow\mathcal{F}$. A quotient pair $(\mathcal{F}^{\prime\prime},\varphi^{\prime\prime})$ consists of a coherent quotient sheaf $q:\mathcal{F}\to\mathcal{F}^{\prime\prime}$ and a  morphism
	$\varphi^{\prime\prime}=q\circ\varphi:\mathcal{F}_{0}\to\mathcal{F}^{\prime\prime}$.
\end{definition}

A pair $(\mathcal{F},\varphi)$ on $\mathcal{X}/S$ or a coherent sheaf $\mathcal{F}$ on $\mathcal{X}/S$ is said to be of dimension $d$ if $\dim\mathcal{F}_{s}=d$ for any geometric point $\mathrm{Spec}\,k\xrightarrow{s} S$. We say a pair $(\mathcal{F},\varphi)$ on $\mathcal{X}/S$ or a coherent sheaf $\mathcal{F}$ on $\mathcal{X}/S$ is  pure if $\mathcal{F}_{s}$ is pure for every geometric point  $s$ of  $S$.
 A pair $(\mathcal{F},\varphi)$ on $\mathcal{X}/S$ is called nontrivial if for each geometric point  $s$ of  $S$, the  morphism $\varphi_{s}: \mathcal{F}_{0,s}\to\mathcal{F}_{s}$ is not zero (also called nontrivial). By [\cite{Lyj}, Remark 2.8],  for each geometric point  $s$ of  $S$, the fiber $\mathcal{X}_{s}$ is a projective Deligne-Mumford stack with a moduli scheme $X_{s}$ and possesses a generating sheaf $\mathcal{E}_{s}$.  Let $P$ be a polynomial of degree $d$, we call  a pair $(\mathcal{F},\varphi)$ on $\mathcal{X}/S$ of type $P$ if  $P_{\mathcal{E}_{s}}(\mathcal{F}_{s})=P$ for each geometric point  $s$ of  $S$. A pair $(\mathcal{F},\varphi)$ on $\mathcal{X}/S$ with $\mathcal{F}$ flat over $S$ is  of type $P$ for some polynomial $P$  due to the following lemma.

\begin{lemma}([\cite{Nir1}, Lemma 3.16])\label{global-hilbpoly}
Let $p: \mathcal{X}\to S$ be a family of projective Deligne-Mumford stacks   with a moduli scheme $\pi:\mathcal{X}\to X$ and a relative polarization  $(\mathcal{E}, \mathcal{O}_{X}(1))$. Assume that $S$ is connected and $\mathcal{F}$ is an $S$-flat sheaf on $\mathcal{X}$. Then there is a polynomial	$P$ such that for every geometric point $\mathrm{Spec}\,k\xrightarrow{s} S$ the modified Hilbert polynomial of the fiber $\chi(\mathcal{X}_{s},\mathcal{F}\otimes\mathcal{E}^\vee\otimes\pi^*\mathcal{O}_{X}(m)|_{\mathcal{X}_{s}})=P(m)$.
\end{lemma}

For every geometric point  $s$ of  $S$, we define the Hilbert polynomial of a pair $(\mathcal{F}_{s},\varphi_{s})$ on $\mathcal{X}_{s}$ as
\ben
P_{(\mathcal{F}_{s},\varphi_{s})}=P_{\mathcal{E}_{s}}(\mathcal{F}_{s})+\epsilon(\varphi_{s})\delta
\een
and the reduced Hilbert polynomial of this pair by
\ben
p_{(\mathcal{F}_{s},\varphi_{s})}=p_{\mathcal{E}_{s}}(\mathcal{F}_{s})+\frac{\epsilon(\varphi_{s})\delta}{r(F_{\mathcal{E}_{s}}(\mathcal{F}_{s}))}
\een
where
\[
\epsilon(\varphi_{s})=
\left\{
\begin{aligned}
&1, \;\;\mbox{if} \;\; \varphi_{s}\neq0; \\
& 0, \;\;\mbox{otherwise}.
\end{aligned}
\right.\]

\begin{definition}\label{semi-sub}
	A pair $(\mathcal{F},\varphi)$ on $\mathcal{X}/S$ is $\delta$-(semi)stable if $\mathcal{F}$ is an $S$-flat coherent sheaf on $\mathcal{X}$ and for each geometric point  $s$ of  $S$, the pair $(\mathcal{F}_{s},\varphi_{s})$ is $\delta$-(semi)stable, that is, $\mathcal{F}_{s}$ is pure and $p_{(\mathcal{F}^\prime_{s},\varphi^\prime_{s})}(\leq) p_{(\mathcal{F}_{s},\varphi_{s})}$ for every proper subpair $(\mathcal{F}_{s}^\prime,\varphi_{s}^\prime)$ of $(\mathcal{F}_{s},\varphi_{s})$. 
\end{definition}
With the above definitions, we also have  the corresponding results in [\cite{Lyj}, Section 2.3] when restricted on the fiber over each geometric point  $s$ of  $S$.
The details are omitted here.

\begin{definition}\label{fam-pair}
	A flat family $(\mathcal{F},\varphi)$ of pairs on $\mathcal{X}/S$ parametrized by a $S$-scheme $T$ consists of a coherent sheaf $\mathcal{F}$ on $\mathcal{X}\times_{S} T$ which is flat over $T$ and a morphism $\varphi: \pi_{\mathcal{X}/S}^*\mathcal{F}_{0}\to\mathcal{F}$, where $\pi_{\mathcal{X}/S}:\mathcal{X}\times_{S} T\to \mathcal{X}$ is the natural projection. Two families $(\mathcal{F},\varphi)$ and $(\mathcal{G},\psi)$ are isomorphic if there is an isomorphism $\Phi:\mathcal{F}\to\mathcal{G}$ such that $\Phi\circ\varphi=\psi$.
\end{definition}

\subsection{Boundedness of the family of semistable pairs}
In this subsection, we will adopt the definition of boundedness of a set-theoretic family of sheaves in [\cite{Nir1}, Definition 4.10].
Since an $S$-flat coherent sheaf $\mathcal{H}$ on $\mathcal{X}$ is also a set-theoretic family of coherent sheaves on $\mathcal{X}$, which is obviously bounded. By Kleiman criterion for stacks in  [\cite{Nir1}, Theorem 4.12], there exists an integer $m$ such that 
$\mathcal{H}_{s}$ is $m$-regular  for every geometric point  $s$ of  $S$ (see  [\cite{Nir1}, Definition 4.1 and 4.2] for   $m$-regularity of $\mathcal{H}_{s}$ on $\mathcal{X}_{s}$). As in [\cite{HL3}, Definition 1.7.3], we define the regularity of an $S$-flat coherent sheaf $\mathcal{H}$   as
\ben
\mathrm{reg}_{\mathcal{E}}(\mathcal{H}):=\inf\{m | \mathcal{H}_{s} \mbox{ is $m$-regular for every geometric point  $s$ of  $S$}\}.
\een
An $S$-flat coherent sheaf $\mathcal{H}$ is said to be $m$-regular if $\mathcal{H}_{s}$ is $m$-regular for any geometric point  $s$ of  $S$.
There is a relative version of Grothendieck lemma in [\cite{Nir1}, Lemma 4.13 and Remark 4.14] as follows.

\begin{lemma}\label{gro}
	Let $p: \mathcal{X}\to S$ be a family of projective Deligne-Mumford stacks   with a moduli scheme $\pi:\mathcal{X}\to X$ and a relative polarization  $(\mathcal{E}, \mathcal{O}_{X}(1))$.  Let $P$ be a  polynomial of degree $d$ $(\leq \mbox{$\dim \mathcal{X}_{s}$ for any geometric point $s$ of  $S$})$ and $\bar{\rho}$ an integer. Assume $\mathcal{F}$ is an $S$-flat  coherent sheaf of dimension $d$ on $\mathcal{X}$ with $P_{\mathcal{E}_{s}}(\mathcal{F}_{s})=P$ for every geometric point  $s$ of  $S$ and $\mathrm{reg}_{\mathcal{E}}(\mathcal{F})\leq \bar{\rho}$. Then there exists a constant
	$C=C(P,\bar{\rho})$ such that  for every  pure $d$-dimensional quotient $\mathcal{F}^\prime_{s}$ $($a quotient of $\mathcal{F}_{s}$$)$ on the fibers of  $\mathcal{X}/S$, we have $\hat{\mu}_{\mathcal{E}_{s}}(\mathcal{F}_{s}^\prime)\geq C$ for every geometric point  $s$ of  $S$. Moreover, the family of pure $d$-dimensional  quotients $\mathcal{F}_{i,s}^\prime$  on the fibers of $\mathcal{X}/S$, $i\in I$ $($for some set of indices $I$$)$ with $\hat{\mu}_{\mathcal{E}_{s}}(\mathcal{F}_{i,s}^\prime)$ for every geometric point  $s$ of  $S$ bounded from above is bounded. 
	
	The similar statement as above is true  for every  saturated subsheaf
	$\mathcal{F}^\prime_{s}\subseteq\mathcal{F}_{s}$ on the fibers of  $\mathcal{X}/S$, that is, the slope $\hat{\mu}_{\mathcal{E}_{s}}(\mathcal{F}^\prime_{s})$  for every geometric point  $s$ of  $S$ is bounded from above, and the family of saturated  subsheaves $\mathcal{F}^\prime_{i,s}\subseteq\mathcal{F}_{s}$ on the fibers of $\mathcal{X}/S$, $i\in I$  with $\hat{\mu}_{\mathcal{E}_{s}}(\mathcal{F}_{i,s}^\prime)$ for every geometric point  $s$ of  $S$ bounded from below such that the  quotient $\mathcal{F}_{s}/\mathcal{F}^\prime_{i,s}$ on the fibers of $\mathcal{X}/S$ is pure  of dimension $d$, is bounded.
\end{lemma}

\begin{proof}
We will deal with the first statement, the second statement can be proved similarly. Since $\mathcal{O}_{X}(1)$ is very ample relative to $S$, we have some embedding $i:X\to \mathbb{P}^N_{S}$. Then one can use the relative version of  the proof of the first part of [\cite{HL3}, Lemma 1.7.9], or alternatively Grothendieck lemma  [\cite{Gro}, Lemma 2.5] for  bounding  $\hat{\mu}$-slope from below.

For the second part of the first statement, following  the  argument in the proof of [\cite{Nir1}, Lemma 4.13], one can firstly prove that there is a coherent sheaf $\mathcal{G}$ on $X\times_{S}T$ ($T$ is an $S$-scheme of finite type) bounding the family $F_{\mathcal{E}_{s}}(\mathcal{F}_{i,s}^\prime)$, $i\in I$ on the fibers of  $X/S$ by using Grothendieck lemma  [\cite{Gro}, Lemma 2.5] again, and  pull back $\mathcal{G}$ and this family by the functor $G_{p_{\mathcal{X}}^*\mathcal{E}}$ where $p_{\mathcal{X}}:\mathcal{X}\times_{S}T\to \mathcal{X}$ to the stacky case such that the family $\mathcal{F}^\prime_{i}$ ($i\in I$) has appeared in the  quotients of $G_{p_{\mathcal{X}}^*\mathcal{E}}(\mathcal{G})$, and then use Kleiman criterion for stacks in  [\cite{Nir1}, Theorem 4.12-(4)] to assert the boundedness of  the family $\mathcal{F}_{i}^\prime$, $i\in I$.  
\end{proof}

Now we have the following relative version of [\cite{Lyj}, Proposition 3.3].
\begin{proposition}\label{bound1}
	Fix a modified Hilbert polynomial $P$ and some $\delta$. The set-theoretic family
	\ben
	\{\mathcal{F}\arrowvert (\mathcal{F},\varphi)\mbox{ is a nontrivial $\delta$-semistable pair on $\mathcal{X}/S$ of type $P$}\}
	\een
	of coherent sheaves on $\mathcal{X}$ is bounded.
\end{proposition}
\begin{proof}
First, we choose the geometric point $\bar{s}\in
 S$  such that $\deg(\mathcal{O}_{X_{\bar{s}}}(1)):=\max\{\deg(\mathcal{O}_{X_{s}}(1))|s\in S\}$ as in the proof of [\cite{Nir1}, Lemma 4.26].  By  Kleiman criterion for stacks in [\cite{Nir1}, Theorem 4.12] for the $S$-flat coherent sheaf $\mathcal{E}$  on $\mathcal{X}$  as in the beginning of this subsection, one can choose an integer $\widetilde{m}>0$  such that $\pi_{*}\mathcal{E}nd_{\mathcal{O}_{\mathcal{X}}}(\mathcal{E})(\widetilde{m})$ is generated by global sections, which is certainly true on each fiber of $X\to S$.
For the case when $\deg\delta<\deg P$, one can  apply the  argument in the proof of [\cite{Lyj}, Lemma 3.1] to get the upper bound (a function of $s$) for each geometric point  $s$ of  $S$. With the choices of $\bar{s}$ and $\widetilde{m}$,  one can
get a uniform upper bound for $\{\hat{\mu}_{\mathrm{max}}(F_{\mathcal{E}_{s}}(\mathcal{F}_{s}))| s\in S\}$ depending on the fixed $S$-flat sheaf $\mathcal{F}_{0}$ and $P$. When $\mathrm{deg}\,\delta\geq\mathrm{deg}\,P$, similarly, follow the  argument  in  
the proof of [\cite{Lyj}, Lemma 3.2] to obtain the lower bound for every geometric point  $s$ of  $S$, where we apply  [\cite{LeP}, Lemma 2.12] to $\hat{\mu}_{\mathrm{min}}(\mathcal{O}_{Y_{s}})$ with $Y_{s}$ being the scheme-theoretic  support of $F_{\mathcal{E}_{s}}(\mathcal{F}_{s})$. Using the flatness of $F_{\mathcal{E}}(\mathcal{F})$ over $S$ by [\cite{Nir1}, Corollary 1.3-(3)] or Lemma \ref{global-hilbpoly}, one can obtain a uniform lower bound for $\{\hat{\mu}_{\mathrm{min}}(F_{\mathcal{E}_{s}}(\mathcal{F}_{s}))| s\in S\}$ depending on $P$ and $X$. Since bounding $\hat{\mu}_{\mathrm{min}}$ from below is equivalent to bounding $\hat{\mu}_{\mathrm{max}}$ from above and  the uniform  bound is independent of choices of $\mathcal{F}$ in the set-theoretic family,  the proof is  completed by  [\cite{Nir1}, Theorem 4.27-(1)].
\end{proof}
With the above preparations, we have a relative version of [\cite{Lyj}, Lemma 3.5] which is useful for relating  GIT-(semi)stability with $\delta$-(semi)stability.
\begin{lemma}\label{bound2}
	Let $p: \mathcal{X}\to S$ be a family of projective Deligne-Mumford stacks   with a moduli scheme $\pi:\mathcal{X}\to X$ and a relative polarization  $(\mathcal{E}, \mathcal{O}_{X}(1))$.	
	Assume that $\mathrm{deg}\,\delta<\mathrm{deg}\,P$. Then there is an integer $m_{0}>0$, such that for any integer $m\geq m_{0}$ and any nontrivial pair $(\mathcal{F},\varphi)$ on $\mathcal{X}/S$ satisfying that $\mathcal{F}$ is a  pure $S$-flat coherent sheaf of dimension $d$ on  $\mathcal{X}$ with
	$P_{\mathcal{E}_{s}}(\mathcal{F}_{s})=P$ and $r=r(F_{\mathcal{E}_{s}}(\mathcal{F}_{s}))$,  $\forall\, s\in S$, the following properties are equivalent.\\
	$(i)$ The pair $(\mathcal{F},\varphi)$ on $\mathcal{X}/S$ is $\delta$-(semi)stable.\\
	$(ii)$ For any geometric point  $s$ of  $S$, $P(m)\leq h^0(X_{s}, F_{\mathcal{E}_{s}}(\mathcal{F}_{s})(m))$ and for any  subpair $(\mathcal{F}_{s}^\prime,\varphi_{s}^\prime)$ of $(\mathcal{F}_{s},\varphi_{s})$ with  $r(F_{\mathcal{E}_{s}}(\mathcal{F}^\prime_{s}))=r^\prime_{s}$ satisfying $0<r^\prime_{s}<r$,
	\ben
	h^0(X_{s}, F_{\mathcal{E}_{s}}(\mathcal{F}^\prime_{s})(m))+\epsilon(\varphi^\prime_{s})\delta(m) (\leq)\frac{r^\prime_{s}}{r}(P(m)+\epsilon(\varphi_{s})\delta(m)).
	\een
	$(iii)$ For any geometric point  $s$ of  $S$ and for any quotient pair $(\mathcal{G}_{s},\varphi^{\prime\prime}_{s})$ of $(\mathcal{F}_{s},\varphi_{s})$ with  $r(F_{\mathcal{E}_{s}}(\mathcal{G}_{s}))=r^{\prime\prime}_{s}$ satisfying $0<r^{\prime\prime}_{s}<r$,
	\ben
	\frac{r^{\prime\prime}_{s}}{r}(P(m)+\epsilon(\varphi_{s})\delta(m)) (\leq) h^0(X_{s}, F_{\mathcal{E}_{s}}(\mathcal{G}_{s})(m))+\epsilon(\varphi^{\prime\prime}_{s})\delta(m).
	\een
\end{lemma}
\begin{proof}
Notice that it follows from [\cite{Lyj}, Lemma 3.4] that for any geometric point  $s$ of  $S$ and  any  subpair $(\mathcal{F}_{s}^\prime,\varphi_{s}^\prime)$ of $(\mathcal{F}_{s},\varphi_{s})$ with  $r(F_{\mathcal{E}_{s}}(\mathcal{F}^\prime_{s}))=r^\prime_{s}$ satisfying $0<r^\prime_{s}<r$, we have
\ben
	\frac{h^0(X_{s},F_{\mathcal{E}_{s}}(\mathcal{F}^\prime_{s})(m))}{r_{s}^\prime}&\leq&\frac{r_{s}^\prime-1}{r_{s}^\prime}\left[\binom{\hat{\mu}_{\mathrm{max}}(F_{\mathcal{E}_{s}}(\mathcal{F}^\prime_{s}))+m+C_{s}}{d}\right]_{+}+\frac{1}{r_{s}^\prime}\left[\binom{\hat{\mu}(F_{\mathcal{E}_{s}}(\mathcal{F}^\prime_{s}))+m+C_{s}}{d}\right]_{+}\\
	&\leq&\frac{r-1}{r}\left[\binom{\hat{\mu}_{\mathrm{max}}(F_{\mathcal{E}_{s}}(\mathcal{F}^\prime_{s}))+m+C}{d}\right]_{+}+\frac{1}{r}\left[\binom{\hat{\mu}(F_{\mathcal{E}_{s}}(\mathcal{F}^\prime_{s}))+m+C}{d}\right]_{+}
\een
	where  $C_{s}:=\widetilde{m}\deg(\mathcal{O}_{X_{s}}(1))+r_{s}^{\prime 2}+f(r_{s}^\prime)+\frac{d-1}{2}$ and $C:=\widetilde{m}\deg(\mathcal{O}_{X_{\bar{s}}}(1))+r^2+f(r)+\frac{d-1}{2}$. Here,  $[x]_{+}:=\max\{0,x\}$,  $f(r)=-1+\sum_{i=1}^r\frac{1}{i}$, and $\bar{s}$, $\widetilde{m}$ are chosen as in the proof of  Proposition \ref{bound1}. Now the proof follows from the similar argument in [\cite{Lyj}, Lemma 3.5] by applying Proposition \ref{bound1}, Grothendieck Lemma \ref{gro} which is used for two set-theoretic families restricted on the associated closed subschemes of $S$ (where the flatness of $\mathcal{F}$ is preserved by base change) induced respectively by bounding the slope of $F_{\mathcal{E}_{s}}(\mathcal{F}_{s}^\prime)$ and $F_{\mathcal{E}_{s}}(\mathcal{G}_{s})$ (e.g., as in Case $(a)$ and Case $(\tilde{b})$ in the proof of [\cite{Lyj}, Lemma 3.5]), and Kleiman criterion for stacks in  [\cite{Nir1}, Theorem 4.12].
\end{proof}

Similarly, the relative version of  [\cite{Lyj}, Lemma 3.6] is presented as follows.
\begin{lemma}\label{bound3}
Let $p: \mathcal{X}\to S$ be a family of projective Deligne-Mumford stacks   with a moduli scheme $\pi:\mathcal{X}\to X$ and a relative polarization  $(\mathcal{E}, \mathcal{O}_{X}(1))$.
	Assume that $\mathrm{deg}\,\delta\geq\mathrm{deg}\,P$. Then there is an integer $m_{0}>0$, such that for any integer $m\geq m_{0}$ and any nontrivial pair $(\mathcal{F},\varphi)$ on $\mathcal{X}/S$ satisfying that $\mathcal{F}$ is a  pure $S$-flat coherent sheaf of dimension $d$ on  $\mathcal{X}$ with
		$P_{\mathcal{E}_{s}}(\mathcal{F}_{s})=P$ and $r=r(F_{\mathcal{E}_{s}}(\mathcal{F}_{s}))$,  $\forall\, s\in S$, the following properties are equivalent.\\
	$(i)$ The pair $(\mathcal{F},\varphi)$ on $\mathcal{X}/S$ is $\delta$-stable.\\
	$(ii)$ For any geometric point  $s$ of  $S$,  $P(m)\leq h^0(X_{s}, F_{\mathcal{E}_{s}}(\mathcal{F}_{s})(m))$ and for any  subpair $(\mathcal{F}_{s}^\prime,\varphi^\prime_{s})$ of $(\mathcal{F}_{s},\varphi_{s})$ with  $r(F_{\mathcal{E}_{s}}(\mathcal{F}^\prime_{s}))=r^\prime_{s}$ satisfying $0<r^\prime_{s}<r$,
	\ben
	\frac{h^0(X_{s}, F_{\mathcal{E}_{s}}(\mathcal{F}^\prime_{s})(m))}{2r^\prime_{s}-\epsilon(\varphi^\prime_{s})} <\frac{h^0(X_{s}, F_{\mathcal{E}_{s}}(\mathcal{F}_{s})(m))}{2r-\epsilon(\varphi_{s})}.
	\een
	$(iii)$ For any geometric point  $s$ of  $S$ and for any quotient pair $(\mathcal{G},\varphi^{\prime\prime})$ with  $r(F_{\mathcal{E}_{s}}(\mathcal{G}_{s}))=r^{\prime\prime}_{s}$ satisfying $0<r^{\prime\prime}_{s}<r$,
	\ben
	\frac{P(m)}{2r-\epsilon(\varphi_{s})} <\frac{h^0(X_{s}, F_{\mathcal{E}_{s}}(\mathcal{G}_{s})(m))}{2r^{\prime\prime}_{s}-\epsilon(\varphi^{\prime\prime}_{s})}.
	\een
\end{lemma}

\subsection{Construction of relative moduli spaces of semistable pairs}
We will give a moduli functor of $\delta$-(semi)stable pairs on $\mathcal{X}/S$, and then use GIT machinery to prove the existence of a relative moduli space for this functor. We make the convention that $\delta$-(semi)stable pairs  in the following are   nontrivial. Let $P$ be a given polynomial of degree $d$ such that $d\leq\dim\mathcal{X}_{s}$ for any geometric point  $s$ of  $S$. 
\begin{definition}\label{moduli-functor}
	Define a functor
	\ben
	\mathcal{M}^{ss}_{\mathcal{X}/S}(\mathcal{F}_{0},P,\delta): (\mathrm{Sch}/S)^\circ\to(\mathrm{Sets})
	\een
	as follows. If $T$ is an $S$-scheme of finite type, let $\mathcal{M}^{ss}_{\mathcal{X}/S}(\mathcal{F}_{0},P,\delta)(T)$ be the set of isomorphism classes of flat families of $\delta$-semistable pairs  $(\mathcal{F},\varphi)$ on $\mathcal{X}\times_{S}T$ with  
	modified Hilbert polynomial $P$ parametrized by  $T$, that is, such a flat family $(\mathcal{F},\varphi)$ satisfies that for each point $t\in T$, the pair $(\mathcal{F}_{t},\varphi|_{(\pi_{\mathcal{X}/S}^*\mathcal{F}_{0})_{t}})$ is a $\delta$-semistable pair on $\mathcal{X}\times_{S}\mathrm{Spec}(k(t))$ with  modified Hilbert polynomial $P_{\check{\pi}_{t}^*\mathcal{E}}(\mathcal{F}_{t})=P$ where $\pi_{\mathcal{X}/S}:\mathcal{X}\times_{S} T\to \mathcal{X}$ and  $\check{\pi}_{t}:\mathcal{X}\times_{S}\mathrm{Spec}(k(t))\to\mathcal{X}$ are the projections. And for every morphism of $S$-schemes $f: T^\prime\to T$, we obtain a map
	\ben
	\mathcal{M}^{ss}_{\mathcal{X}/S}(\mathcal{F}_{0},P,\delta)(f):\mathcal{M}^{ss}_{\mathcal{X}/S}(\mathcal{F}_{0},P,\delta)(T)\to\mathcal{M}^{ss}_{\mathcal{X}/S}(\mathcal{F}_{0},P,\delta)(T^\prime)
	\een
	via pulling back $\mathcal{F}$ and $\varphi$ by $\mathrm{id}_{\mathcal{X}}\times f: \mathcal{X}\times_{S} T^\prime\to \mathcal{X}\times_{S} T$. If we take families of $\delta$-stable pairs, we denote the corresponding subfunctor by $\mathcal{M}^s_{\mathcal{X}/S}(\mathcal{F}_{0},P,\delta)$. 
\end{definition}

As in [\cite{Lyj}, Section 4.1], there exists an integer $\widehat{m}>0$ such that for each integer $m\geq\widehat{m}$,
we have the following conditions:\\
$(i)$ $F_{\mathcal{E}}(\mathcal{F}_{0})$ is $m$-regular (by Kleiman criterion for stacks for the fixed $S$-flat coherent sheaf $\mathcal{F}_{0}$ on $\mathcal{X}$ as in Section 3.2)  on the fiber of $X\to S$ over any geometric point  $s$ of  $S$, and $\delta(m)>0$ unless $\delta=0$.\\
$(ii)$ For any $\delta$-semistable pair
$(\mathcal{F},\varphi)$ of type $P$, the $S$-flat sheaf $\mathcal{F}$ is $m$-regular (by Proposition \ref{bound1}).\\
$(iii)$  Equivalent properties in both Lemma \ref{bound2} and Lemma \ref{bound3} hold.

Take such an integer $m$ and  define \ben
&&\widetilde{\mathcal{Q}}:=\mathrm{Quot}_{\mathcal{X}/S}(V\otimes\mathcal{E}\otimes\pi^*\mathcal{O}_{X}(-m),P),\\
&&\mathbb{P}:=\mathbb{P}(\mathrm{Hom}(H^0(X,F_{\mathcal{E}}(\mathcal{F}_{0})(m)),V)),
\een
where $V:=k^{\oplus P(m)}$, and $H^0(X,F_{\mathcal{E}}(\mathcal{F}_{0})(m))$ is a trivial vector bundle over $S$  by $(i)$ and  Lemma \ref{global-hilbpoly}. 

We follow the similar argument in [\cite{Lyj}, Section 4.1] for our relative case, see also [\cite{HL3,OS03,Nir1}].
Denote by  $\mathcal{Z}\subset \mathbb{P}\times_{S}\widetilde{\mathcal{Q}}$ the closed subscheme consisting of all points $([a], [q])$ such that for any geometric point  $s$ of  $S$,  $q_{s}\circ (a_{s}\otimes\mathrm{id})$ factors through $\widetilde{\mathrm{ev}}_{s}$ in a commutative diagram as follows
\ben
\xymatrix{
	H^0(X_{s},F_{\mathcal{E}_{s}}(\mathcal{F}_{0,s})(m))\otimes\mathcal{E}_{s}\otimes\pi^*\mathcal{O}_{X_{s}}(-m)\ar[d]_{a_{s}\otimes\mathrm{id}} \ar[r]^-{\widetilde{\mathrm{ev}}_{s}} &   \mathcal{F}_{0,s}\ar[d]^{\varphi_{s}}\\
	V\otimes\mathcal{E}_{s}\otimes\pi^*\mathcal{O}_{X_{s}}(-m)\ar[r]_-{q_{s}} & \mathcal{F}_s
}
\een
where $a_{s}: H^0(X_{s},F_{\mathcal{E}_{s}}(\mathcal{F}_{0,s})(m))\to V$, $\mathrm{ev}_{s}:H^0(X_{s},F_{\mathcal{E}_{s}}(\mathcal{F}_{0,s})(m))\otimes\mathcal{O}_{X_{s}}(-m)\to F_{\mathcal{E}_{s}}(\mathcal{F}_{0,s})$, and $\widetilde{\mathrm{ev}}_{s}:=\theta_{\mathcal{E}_{s}}(\mathcal{F}_{0,s})\circ G_{\mathcal{E}_{s}}(\mathrm{ev}_{s})$. Actually, a point $([a],[q])\in\mathcal{Z}$ corresponds to a pair $(\mathcal{F},\varphi)$ induced by  $q_{s}\circ (a_{s}\otimes\mathrm{id})$ factoring through $\widetilde{\mathrm{ev}}_{s}$ for any geometric point $s$ of  $S$, that is, the composition $q\circ (a\otimes\mathrm{id})$ factors through $\widetilde{\mathrm{ev}}:H^0(X,F_{\mathcal{E}}(\mathcal{F}_{0})(m))\otimes\mathcal{E}\otimes\pi^*\mathcal{O}_{X}(-m)\to\mathcal{F}_{0}$ relative to $S$. Then $\mathcal{Z}$ is a projective $S$-scheme. Consider the subscheme in $\mathcal{Z}$ consisting of  all points $([a], [q])\in\mathcal{Z}$  such that $\mathcal{F}_{s}$ is pure and $q$ induces an isomorphism
$V\rightarrow H^0(X_{s},F_{\mathcal{E}_{s}}(\mathcal{F}_{s})(m))$ for any  $s$ of  $S$, and denote by $\mathcal{Z}^\prime$ the closure of  this open subset in $\mathcal{Z}$. Then the parameter space $\mathcal{Z}^\prime$ is also a projective $S$-scheme. Now we have a natural $\mathrm{GL}(V)$-action on $\mathbb{P}\times_{S}\widetilde{\mathcal{Q}}$, and $\mathcal{Z}^\prime$ is invariant under this action. This induces the  natural $\mathrm{SL}(V)$-linearizations for the following very ample line bundles relative to $S$
\ben
\mathcal{O}_{\mathcal{Z}^\prime}(n_{1},n_{2}):=p_{\mathbb{P}\times_{S}\widetilde{\mathcal{Q}},\mathbb{P}}^*\mathcal{O}_{\mathbb{P}}(n_{1})\otimes p_{\mathbb{P}\times_{S}\widetilde{\mathcal{Q}},\widetilde{\mathcal{Q}}}^*\mathcal{L}_{l}^{n_{2}}
\een
where $\mathcal{L}_{l}:=\det\left(p_{X\times_{S}\widetilde{\mathcal{Q}},\widetilde{\mathcal{Q}}*}\left(F_{p_{\mathcal{X}\times_{S}\widetilde{\mathcal{Q}},\mathcal{X}}^*\mathcal{E}}(\widetilde{\mathcal{F}})(l)\right)\right)$ is a class of very ample invertible sheaves on $\widetilde{\mathcal{Q}}$  for $l$ sufficiently large,  $\widetilde{\mathcal{F}}$ is the universal quotient sheaf of $\widetilde{\mathcal{Q}}$, and $n_{1}, n_{2}$ are any two positive integers.  Here and below, we use the convention that   $p_{\star,\bullet}$ denotes the natural projection from $\star$  to $\bullet$.

Next, as in [\cite{Lyj}, Section 4.2], we assume  two positive integers $n_{1}$ and  $n_{2}$ are chosen to satisfy
\[
\frac{n_{1}}{n_{2}}=
\left\{
\begin{aligned}
&\frac{P(l)\cdot\delta(m)-\delta(l)\cdot P(m)}{P(m)+\delta(m)}, \;\;\mbox{if} \;\; \mathrm{deg}\,\delta<\mathrm{deg}\,P; \\
& \frac{P(l)}{2r}, \;\;\mbox{if}\;\;\mathrm{deg}\,\delta\geq\mathrm{deg}\,P,
\end{aligned}
\right.\]
where $\frac{r}{(\deg P)!}$ is the coefficient of the term $m^{\mathrm{deg}\,P}$ in $P(m)$. 
Notice that $\mathcal{Z}^\prime$ is a projective $S$-scheme with a $\mathrm{SL}(V)$-action  and the group $\mathrm{SL}(V)$ preserves the map $\mathcal{Z}^\prime\to S$ with the trivial action on $S$. Due to [\cite{Ses}, Proposition 7] or [\cite{Simp}, Lemma 1.13],  GIT-(semi)stable points of the total space $\mathcal{Z}^\prime$ restricted on  any geometric point  $s$ of  $S$ are exactly those GIT-(semi)stable points of the fiber of $\mathcal{Z}^\prime\to S$ over $s$. Then one can obtain GIT-(semi)stable points of $\mathcal{Z}^\prime$ via the restriction on fibers  over all geometric points. Therefore,  it follows from  the similar argument in [\cite{Lyj}, Section 4.2]  that
\begin{theorem}\label{GIT-ss}
For $l$ sufficiently large,  a point $([a],[q])\in\mathcal{Z}^\prime$ is GIT-(semi)stable with respect to $\mathcal{O}_{\mathcal{Z}^\prime}(n_{1},n_{2})$ if and only if  the corresponding pair $(\mathcal{F},\varphi)$ is $\delta$-(semi)stable and the map  $V\to H^0(X_{s},F_{\mathcal{E}_{s}}(\mathcal{F}_{s})(m))$ induced by $q_{s}$ is an isomorphism for any geometric point  $s$ of  $S$.
\end{theorem}
\begin{remark}
For each geometric point  $s$ of  $S$,   the map $V\to H^0(X_{s},F_{\mathcal{E}_{s}}(\mathcal{F}_{s})(m))$ induced by $q_{s}$ in Theorem \ref{GIT-ss} is $H^0(\widetilde{q}_{s})$ where the map $\widetilde{q}_{s}$ is defined as the following composition
\ben V\otimes\mathcal{O}_{X_{s}}\xrightarrow{\varphi_{\mathcal{E}_{s}}(V\otimes\mathcal{O}_{X_{s}})}F_{\mathcal{E}_{s}}\circ G_{\mathcal{E}_{s}}(V\otimes\mathcal{O}_{X_{s}})\xrightarrow{F_{\mathcal{E}_{s}}(q_{s})(m)}F_{\mathcal{E}_{s}}(\mathcal{F}_{s})(m).
\een
\end{remark}		
	
By Theorem \ref{GIT-ss}, we define $\widetilde{\mathcal{R}}^{(s)s}\subseteq\mathcal{Z}^\prime\subset\mathcal{Z}\subset \mathbb{P}\times_{S}\widetilde{\mathcal{Q}}$ to be 
the subset consisting of GIT-(semi)stable points $([a],[q])$ corresponding to $\delta$-(semi)stable pairs $(\mathcal{F},\varphi)$ with an isomorphism  $V\to H^0(X_{s},F_{\mathcal{E}_{s}}(\mathcal{F}_{s})(m))$ induced by $q_{s}$  for any geometric point  $s$ of  $S$. Then $\widetilde{\mathcal{R}}^{ss}$ and $\widetilde{\mathcal{R}}^{s}$ are open $\mathrm{SL}(V)$-invariant subsets of $\mathcal{Z}^\prime$. Let $p_{\mathcal{X}\times_{S}\widetilde{\mathcal{Q}},\mathcal{X}}^*(V\otimes\mathcal{E}\otimes\pi^*\mathcal{O}_{X}(-m))\to\widetilde{\mathcal{F}}$ be the universal family parameterized by $\widetilde{\mathcal{Q}}$.  As in [\cite{Lyj}, Remark 4.1], one has  a flat family $(\check{\mathcal{F}}_{\mathcal{X}\times_{S}\mathcal{Z}},\check{\varphi}_{\mathcal{X}\times_{S}\mathcal{Z}})$ of pairs parametrized by $\mathcal{Z}$, where 
\ben
\check{\mathcal{F}}_{\mathcal{X}\times_{S}\mathcal{Z}}:=(\mathrm{id}_{\mathcal{X}}\times\widetilde{i})^*p_{\mathcal{X}\times_{S}\mathbb{{P}}\times_{S}\widetilde{\mathcal{Q}},\mathcal{X}\times_{S}\widetilde{\mathcal{Q}}}^*\widetilde{\mathcal{F}}\otimes(\mathrm{id}_{\mathcal{X}}\times\widetilde{i})^* p_{\mathcal{X}\times_{S}\mathbb{P}\times_{S}\widetilde{\mathcal{Q}},\mathbb{P}}^*\mathcal{O}_{\mathbb{P}}(1)
\een
and 
$
\check{\varphi}_{\mathcal{X}\times_{S}\mathcal{Z}}:p_{\mathcal{X}\times_{S}\mathcal{Z},\mathcal{X}}^*\mathcal{F}_{0}\to \check{\mathcal{F}}_{\mathcal{X}\times_{S}\mathcal{Z}}
$. Here, $\widetilde{i}:\mathcal{Z}\to\mathbb{P}\times_{S}\widetilde{\mathcal{Q}}$ be the inclusion. Due to the universal properties of $\mathbb{P}$ and $\widetilde{\mathcal{Q}}$, one has a universal family  $(\check{\mathcal{F}}_{\mathcal{X}\times_{S}\widetilde{\mathcal{R}}^{(s)s}},\check{\varphi}_{\mathcal{X}\times_{S}\widetilde{\mathcal{R}}^{(s)s}})$ of $\delta$-(semi)stable pairs with  
modified Hilbert polynomial $P$ parameterized by $\widetilde{\mathcal{R}}^{(s)s}$ where  the morphism $\check{\varphi}_{\mathcal{X}\times_{S}\widetilde{\mathcal{R}}^{(s)s}}$ is 
\ben
 p_{\mathcal{X}\times_{S}\widetilde{\mathcal{R}}^{(s)s},\mathcal{X}}^*\mathcal{F}_{0}\to \check{\mathcal{F}}_{\mathcal{X}\times_{S}\widetilde{\mathcal{R}}^{(s)s}}:=(\mathrm{id}_{\mathcal{X}}\times\widetilde{i}_{(s)s})^*p_{\mathcal{X}\times_{S}\mathbb{{P}}\times_{S}\widetilde{\mathcal{Q}},\mathcal{X}\times_{S}\widetilde{\mathcal{Q}}}^*\widetilde{\mathcal{F}}\otimes(\mathrm{id}_{\mathcal{X}}\times\widetilde{i}_{(s)s})^* p_{\mathcal{X}\times_{S}\mathbb{P}\times_{S}\widetilde{\mathcal{Q}},\mathbb{P}}^*\mathcal{O}_{\mathbb{P}}(1)
\een
and $\widetilde{i}_{(s)s}: \widetilde{\mathcal{R}}^{(s)s}\to\mathbb{P}\times_{S}\widetilde{\mathcal{Q}}$ is the inclusion. 

As in  [\cite{HL3}, Theorem 4.3.7],  we have
\begin{theorem}\label{corepresent}
Let $p: \mathcal{X}\to S$ be a family of projective Deligne-Mumford stacks   with a moduli scheme $\pi:\mathcal{X}\to X$ and a relative polarization  $(\mathcal{E}, \mathcal{O}_{X}(1))$.
Then there exists a projective $S$-scheme $\widetilde{M}^{ss}_{\mathcal{X}/S}(\mathcal{F}_{0},P,\delta)$ which universally corepresents the functor $\mathcal{M}^{ss}_{\mathcal{X}/S}(\mathcal{F}_{0},P,\delta)$. 
In particular, for any geometric point  $s$ of  $S$, one has $\widetilde{M}^{ss}_{\mathcal{X}/S}(\mathcal{F}_{0},P,\delta)_{s}\cong \widetilde{M}^{ss}_{\mathcal{X}_{s}/k}(\mathcal{F}_{0}|_{s},P,\delta)$. 
Moreover, there is  an open subscheme $\widetilde{M}^{s}_{\mathcal{X}/S}(\mathcal{F}_{0},P,\delta)\subset \widetilde{M}^{ss}_{\mathcal{X}/S}(\mathcal{F}_{0},P,\delta)$ that universally corepresents the subfunctor $\mathcal{M}^s_{\mathcal{X}/S}(\mathcal{F}_{0},P,\delta)$.
\end{theorem}
\begin{proof}

Firstly, by  [\cite{Ses}, Theorem 4 and  Remarks 8, 9]  (see also  [\cite{HL3}, Theorem 4.2.10]), there is a projective $S$-scheme $\widetilde{M}^{ss}:=\widetilde{M}^{ss}_{\mathcal{X}/S}(\mathcal{F}_{0},P,\delta)$ and a morphism $\widetilde{\Theta}: \widetilde{\mathcal{R}}^{ss}\to \widetilde{M}^{ss}$ such that $\widetilde{\Theta}$ is a universal good quotient for the $\mathrm{SL}(V)$-action. And there is an open subscheme $\widetilde{M}^s:=\widetilde{M}^s_{\mathcal{X}/S}(\mathcal{F}_{0},P,\delta)\subseteq \widetilde{M}^{ss}$ such that $\widetilde{\mathcal{R}}^s=\widetilde{\Theta}^{-1}(\widetilde{M}^s)$ and $\widetilde{\Theta}:\widetilde{\mathcal{R}}^s\to \widetilde{M}^s$ is a universal geometric quotient. 

Secondly, combined with Theorem \ref{GIT-ss}, it follows from the arguments in [\cite{Lyj}, Lemmas 4.11,  4.12 and Remark 4.13] that for any geometric point  $s$ of  $S$,  two points $([a_{1}],[q_{1}])$ and $([a_{2}],[q_{2}])$  in $\widetilde{\mathcal{R}}^{ss}$ restricted on the fiber over $s$ intersect if and only if their corresponding $\delta$-semistable pairs $(\mathcal{F}_{1},\varphi_{1})|_{s}$ and $(\mathcal{F}_{2},\varphi_{2})|_{s}$   are $S$-equivalent (see [\cite{Lyj}, Definition 2.33]).

Next, we follow the similar argument in [\cite{HL3}, Lemma 4.3.1] or [\cite{BS}, Theorem 4.12] to prove the corepresentability for the moduli functor $\mathcal{M}^{(s)s}_{\mathcal{X}/S}(\mathcal{F}_{0},P,\delta)$.  Assume that $[(\mathcal{F},\varphi)]$ is an element of $\mathcal{M}^{(s)s}_{\mathcal{X}/S}(\mathcal{F}_{0},P,\delta)(T)$, i.e., there exists a flat family of  $\delta$-(semi)stable pairs $(\mathcal{F},\varphi)$ on $\mathcal{X}\times_{S}T$ with  modified Hilbert polynomial $P$ parametrized by  $T$. Set 
\ben
\mathcal{W}_{\mathcal{F}}:=p_{X\times_{S}T,T*}\left(F_{p_{\mathcal{X}\times_{S}T,\mathcal{X}}^*\mathcal{E}}(\mathcal{F})\otimes p_{X\times_{S}T,X}^*\mathcal{O}_{X}(m)\right).
\een
By [\cite{Nir1}, Proposition 1.5], we have $F_{p_{\mathcal{X}\times_{S}T,\mathcal{X}}^*\mathcal{E}}(\mathcal{F})_{t}\cong F_{\check{\pi}_{t}^*\mathcal{E}}(\mathcal{F}_{t})$, which is $m$-regular for each point $t\in T$ by the assumption of $m$. Then $h^0(F_{p_{\mathcal{X}\times_{S}T,\mathcal{X}}^*\mathcal{E}}(\mathcal{F})_{t}(m))=P(m)=\dim V$. By [\cite{Har}, Corollary 12.9], $\mathcal{W}_{\mathcal{F}}$ is locally free of rank $P(m)$.
Let $Fb(\mathcal{F}):=\mathrm{Isom}(V\otimes \mathcal{O}_{T},\mathcal{W}_{\mathcal{F}})$ be the frame bundle with natural projection $\overline{q}:Fb(\mathcal{F})\to T$ which is a $\mathrm{GL}(V)$-torsor over $T$. Denote  the universal trivialization by $\varpi:V\otimes\mathcal{O}_{Fb(\mathcal{F})}\xrightarrow{\cong}\overline{q}^*\mathcal{W}_{\mathcal{F}}$. One can pull back the surjection 
$p_{X\times_{S}T,T}^*\mathcal{W}_{\mathcal{F}}\otimes p_{X\times_{S}T,X}^*\mathcal{O}_{X}(-m)\to F_{p_{\mathcal{X}\times_{S}T,\mathcal{X}}^*\mathcal{E}}(\mathcal{F})$ by the map $\mathrm{id}_{X}\times\bar{q}:X\times_{S}Fb(\mathcal{F})\to X\times_{S}T$, and then apply the functor $G_{p_{\mathcal{X}\times_{S}Fb(\mathcal{F})}^*\mathcal{E}}$ and the surjective map $\theta_{p_{\mathcal{X}\times_{S}Fb(\mathcal{F})}^*\mathcal{E}}((\mathrm{id}_{\mathcal{X}}\times \overline{q})^*\mathcal{F})$ to obtain the quotient
\ben
p_{\mathcal{X}\times_{S} Fb(\mathcal{F}),\mathcal{X}}^*(V\otimes \mathcal{E}\otimes\pi^*\mathcal{O}_{X}(-m))\to(\mathrm{id}_{\mathcal{X}}\times \overline{q})^*\mathcal{F}.
\een
This implies that
one has a morphism $Fb(\mathcal{F})\to\widetilde{\mathcal{Q}}$ due to the universal property of $\widetilde{\mathcal{Q}}$. On the other hand,  from the morphism $\varphi: \pi_{\mathcal{X}/S}^*\mathcal{F}_{0}\to\mathcal{F}$ (notice that $\pi_{\mathcal{X}/S}=p_{\mathcal{X}\times_{S}T,\mathcal{X}}$), we have  
\ben
p_{X\times_{S}T,T*}(F_{p_{\mathcal{X}\times_{S}T,\mathcal{X}}^*\mathcal{E}}(\varphi)\otimes p_{X\times_{S}T,X}^*\mathcal{O}_{X}(m)):H^0(F_{\mathcal{E}}(\mathcal{F}_{0})(m))\otimes\mathcal{O}_{T}\to \mathcal{W}_{\mathcal{F}}
\een
where we use the isomorphism $p_{X\times_{S}T,X}^*F_{\mathcal{E}}(\mathcal{F}_{0})\cong F_{\pi_{\mathcal{X}/S}^*\mathcal{E}}(\pi_{\mathcal{X}/S}^*\mathcal{F}_{0})$, and then pull back the above morphism by $\overline{q}$ and compose with $\varpi^{-1}$, we obtain 
\ben
H^0(F_{\mathcal{E}}(\mathcal{F}_{0})(m))\otimes\mathcal{O}_{Fb(\mathcal{F})}\to V\otimes\mathcal{O}_{Fb(\mathcal{F})}.
\een
Then we also have a morphism $Fb(\mathcal{F})\to\mathbb{P}$ by the universal property of $\mathbb{P}$. By the definition of $\mathcal{Z}^\prime$ and Theorem \ref{GIT-ss}, we have the morphism $Fb(\mathcal{F})\to\mathbb{P}\times_{S}\widetilde{\mathcal{Q}}$  factoring through
a morphism $\phi: Fb(\mathcal{F})\to\widetilde{\mathcal{R}}^{(s)s}$ which is obviously $\mathrm{GL}(V)$-equivariant. Set $\widetilde{\mathfrak{R}}^{(s)s}:=[\widetilde{\mathcal{R}}^{(s)s}/\mathrm{GL}(V)]$. As in [\cite{BS}, Theorem 4.12] (see alternatively [\cite{HL3}, Lemma 4.3.1]), we have a natural transformation 
\ben
\Upsilon:\mathcal{M}^{(s)s}_{\mathcal{X}/S}(\mathcal{F}_{0},P,\delta)\to[\widetilde{\mathfrak{R}}^{(s)s}]
\een
where $[\widetilde{\mathfrak{R}}^{(s)s}]$ is the contravariant functor such that for any $S$-scheme $T$ of finite type, $[\widetilde{\mathfrak{R}}^{(s)s}](T)$ is the set of isomorphism classes of objects of $\widetilde{\mathfrak{R}}^{(s)s}(T)$. In fact, one can show that two functors $\mathcal{M}^{(s)s}_{\mathcal{X}/S}(\mathcal{F}_{0},P,\delta)$ and $[\widetilde{\mathfrak{R}}^{(s)s}]$ are isomorphic by following the similar argument in [\cite{BS}, Theorem 4.12].
 By [\cite{Alper}, Theorem 13.6], we have the morphism $[\widetilde{\mathcal{R}}^{(s)s}/\mathrm{SL}(V)]\to \widetilde{M}^{(s)s}:=\widetilde{\mathcal{R}}^{(s)s}\verb|//|\mathrm{SL}(V)$, which  is a good moduli space  on the fiber over each geometric point  $s$ of  $S$. And we have the morphism $\widetilde{\mathfrak{R}}^{(s)s}\to\widetilde{M}^{(s)s}$ by the discussion following  [\cite{BS}, Proposition 4.11] (see also the proof of [\cite{Simp}, Theorem 1.21-(1)]), which  factors through $\widetilde{\mathfrak{R}}^{(s)s}\to[\widetilde{\mathfrak{R}}^{(s)s}]$. Composing $\Upsilon$ with the map $[\widetilde{\mathfrak{R}}^{(s)s}]\to\widetilde{M}^{(s)s}$, we have the natural transformation $\mathcal{M}^{(s)s}_{\mathcal{X}/S}(\mathcal{F}_{0},P,\delta)\to\mathrm{Mor}(-,\widetilde{M}^{(s)s})$ by the Yoneda Lemma. Suppose there is a natural transformation $\mathcal{M}^{(s)s}_{\mathcal{X}/S}(\mathcal{F}_{0},P,\delta)\to \mathrm{Mor}(-,\overline{M})$, where $\overline{M}$ is a $S$-scheme. Since the universal family  $(\check{\mathcal{F}}_{\mathcal{X}\times_{S}\widetilde{\mathcal{R}}^{(s)s}},\check{\varphi}_{\mathcal{X}\times_{S}\widetilde{\mathcal{R}}^{(s)s}})$ is an element in $\mathcal{M}^{ss}_{\mathcal{X}/S}(\mathcal{F}_{0},P,\delta)(\widetilde{\mathcal{R}}^{(s)s})$, then we have a morphism $\widetilde{\mathcal{R}}^{(s)s}\to \overline{M}$ which is clearly $\mathrm{SL}(V)$-equivariant with the trivial $\mathrm{SL}(V)$-action on $\overline{M}$. Since $\widetilde{\mathcal{R}}^{(s)s}\to \widetilde{M}^{(s)s}$ is a categorical quotient, we have the following commutative diagram
\ben
\xymatrix{
	\mathcal{M}^{(s)s}_{\mathcal{X}/S}(\mathcal{F}_{0},P,\delta)\ar[dr] \ar[r] &   \mathrm{Mor}(-,\widetilde{M}^{(s)s})\ar[d]\\
	&\mathrm{Mor}(-,\overline{M})
}
\een
where the morphim $\widetilde{M}^{(s)s}\to \overline{M}$ is unique.
By [\cite{HL3}, Definition 2.2.1], this implies that $\widetilde{M}^{(s)s}$ corepresents the functor $\mathcal{M}^{(s)s}_{\mathcal{X}/S}(\mathcal{F}_{0},P,\delta)$.
Therefore, the proof follows from the above results.
\end{proof}
\begin{remark}
The moduli space $\widetilde{M}^{ss}_{\mathcal{X}_{s}/k}(\mathcal{F}_{0}|_{s},P,\delta)$ in Theorem \ref{corepresent} is exactly corresponding to the moduli space of $\delta$-semistable pairs in [\cite{Lyj}, Theorem 4.24].
\end{remark}	
Now we have
\begin{theorem}\label{fine-coarse}
The projective $S$-scheme $\widetilde{M}^{ss}_{\mathcal{X}/S}(\mathcal{F}_{0},P,\delta)$ is a moduli space for the moduli functor $\mathcal{M}^{ss}_{\mathcal{X}/S}(\mathcal{F}_{0},P,\delta)$, and the quasi-projective $S$-scheme $\widetilde{M}^{s}_{\mathcal{X}/S}(\mathcal{F}_{0},P,\delta)$ is a fine moduli space  for the moduli functor $\mathcal{M}^s_{\mathcal{X}/S}(\mathcal{F}_{0},P,\delta)$.
\end{theorem}	
\begin{proof}
It remains to show the  fineness of the moduli space $\widetilde{M}^{s}_{\mathcal{X}/S}(\mathcal{F}_{0},P,\delta)$. By the  argument in [\cite{Lyj}, Lemmas 4.21, 4.22 and Remark 4.23] for each geometric point of $S$,  one can show the universal family
$(\check{\mathcal{F}}_{\mathcal{X}\times_{S}\widetilde{\mathcal{R}}^{s}},\check{\varphi}_{\mathcal{X}\times_{S}\widetilde{\mathcal{R}}^{s}})$ of $\delta$-stable pairs on $\mathcal{X}/S$ is $\mathrm{PGL}(V)$-linearized. It follows from  [\cite{Lyj}, Lemma 4.21] that the stabilizer group of a closed point of $\widetilde{\mathcal{R}}^s$ in $\mathrm{PGL}(V)$ is trivial due to [\cite{Lyj}, Lemma 2.28] for every geometric point. By [\cite{HL3}, Theorem 4.2.12 and Corollary 4.2.13], the map $\widetilde{\mathcal{R}}^s\to \widetilde{M}^{s}_{\mathcal{X}/S}(\mathcal{F}_{0},P,\delta)$ is a principal $\mathrm{PGL}(V)$-bundle (or a $\mathrm{PGL}(V)$-torsor). Then $(\check{\mathcal{F}}_{\mathcal{X}\times_{S}\widetilde{\mathcal{R}}^{s}},\check{\varphi}_{\mathcal{X}\times_{S}\widetilde{\mathcal{R}}^{s}})$ can be descended to be a universal family parameterized by $\widetilde{M}^{s}_{\mathcal{X}/S}(\mathcal{F}_{0},P,\delta)$, which completes the proof.
\end{proof}

\section{Perfect relative obstruction theory and virtual fundamental class}
In this section, we will consider a special kind of relative moduli spaces constructed in Section 3, for which we construct a perfect relative obstruction theory and define the virtual fundamental class. Let $p: \mathcal{X}\to S$ be a family of projective Deligne-Mumford stacks of relative dimension three   with a moduli scheme $\pi:\mathcal{X}\to X$ and a relative polarization  $(\mathcal{E}, \mathcal{O}_{X}(1))$. We assume that $k=\mathbb{C}$, $\mathcal{F}_{0}=\mathcal{O}_{\mathcal{X}}$ and polynomials $\delta$, $P$   satisfy $\deg\,\delta\geq\deg \,P=1$.  
By Theorem \ref{fine-coarse}, we have a relative moduli space  $\widetilde{M}^{s}_{\mathcal{X}/S}(\mathcal{O}_{\mathcal{X}},P,\delta)$, which
is equal to $\widetilde{M}^{ss}_{\mathcal{X}/S}(\mathcal{O}_{\mathcal{X}},P,\delta)$ due to [\cite{Lyj}, Lemma 2.24] and hence also a projective $S$-scheme. By [\cite{Lyj}, Lemma 2.26 and Remark 2.27], we simply call $\widetilde{M}^{s}_{\mathcal{X}/S}(\mathcal{O}_{\mathcal{X}},P,\delta)$ the moduli space  of orbifold Pandharipande-Thomas stable pairs on $\mathcal{X}/S$.
Notice that $\widetilde{M}^{s}_{\mathcal{X}/S}(\mathcal{O}_{\mathcal{X}},P,\delta)$  is in the same chamber of the stability parameter $\delta$ whenever $\deg\,\delta\geq\deg\,P$, see the proof of [\cite{Lyj}, Theorem 4.26] and [\cite{Lyj}, Remark 4.27] for the fiber over any geometric point $s$ of $S$. To simplify the notation for later use, by fixing a choice of $\delta$ implicitly, we set $\mathrm{PT}_{\mathcal{X}/S}^{(P)}:=\widetilde{M}^{s}_{\mathcal{X}/S}(\mathcal{O}_{\mathcal{X}},P,\delta)$, where  we use the superscript $(P)$ to denote the modified Hilbert polynomial and distinguish from the superscript $P$ which will indicate Hilbert homomorphism defined in  Section 5.3. While for the absolute case, with the same assumption on $\delta, P$ as above, we set
$\mathrm{PT}_{\mathcal{X}/\mathbb{C}}^{(P)}:=M^{s}_{\mathcal{X}/\mathbb{C}}(\mathcal{O}_{\mathcal{X}},P,\delta)$  (see [\cite{Lyj}, Theorem 4.24] for this moduli space), which is the moduli space of orbifold PT stable pairs on $3$-dimensional projective Deligne-Mumford stack $\mathcal{X}$  and is a projective scheme over $\mathbb{C}$. When $S=\mathrm{Spec}\,\mathbb{C}$, we have $\mathrm{PT}_{\mathcal{X}/S}^{(P)}=\mathrm{PT}_{\mathcal{X}/\mathbb{C}}^{(P)}$.

Assume that all fibers of $p$ are smooth. As the similar argument in the proof of [\cite{Lyj}, Theorem 5.17] for the absolute case,  we shall give a perfect relative obstruction theory with fixed determinant for $\widetilde{\mathbf{M}}^s:=\mathrm{PT}_{\mathcal{X}/S}^{(P)}$ over $S$.   Since $\widetilde{\mathbf{M}}^s$ is a fine moduli space  by Theorem \ref{fine-coarse},  there is a universal  complex concentrated in degree 0 and 1 as follows
\ben
\mathbb{I}^\bullet=\{\mathcal{O}_{\mathcal{X}\times_{S} \widetilde{\mathbf{M}}^s}\to\mathbb{F}\}\in\mathrm{D}^b(\mathcal{X}\times_{S} \widetilde{\mathbf{M}}^s)
\een
where $\mathbb{F}$ is flat over $\widetilde{\mathbf{M}}^s$. 
Let $\tilde{\pi}_{\widetilde{\mathbf{M}}^s}:\mathcal{X}\times_{S} \widetilde{\mathbf{M}}^s\to \widetilde{\mathbf{M}}^s$ and $\tilde{\pi}_{\mathcal{X}}:\mathcal{X}\times_{S} \widetilde{\mathbf{M}}^s\to\mathcal{X}$ be the projections. 
Since $\widetilde{\mathbf{M}}^s$ is a projective $S$-scheme, as in [\cite{PT1}, Section 2.3] or [\cite{HT09}, Section 4.2],
we have $R\mathcal{H}om(\mathbb{I}^{\bullet},\mathbb{I}^{\bullet})\cong\mathcal{O}_{\mathcal{X}\times_{S} \widetilde{\mathbf{M}}^s}\oplus R\mathcal{H}om(\mathbb{I}^{\bullet},\mathbb{I}^{\bullet})_{0}$,  where  $R\mathcal{H}om(\mathbb{I}^{\bullet},\mathbb{I}^{\bullet})_{0}$ denotes its traceless part (see [\cite{Tho}, Section 3] for the trace map).
Let $\mathbb{L}_{\mathcal{X}\times_{S} \widetilde{\mathbf{M}}^s/S}$ be the truncated cotangent complex, i.e., the truncation $\tau^{\geq-1}L^\bullet_{\mathcal{X}\times_{S} \widetilde{\mathbf{M}}^s/S}$ of Illusie's cotangent complex $L^\bullet_{\mathcal{X}\times_{S} \widetilde{\mathbf{M}}^s/S}$. Then we have the truncated Atiyah class of $\mathbb{I}^{\bullet}$
\ben
\mathrm{At}(\mathbb{I}^{\bullet})\in\mathrm{Ext}^1(\mathbb{I}^{\bullet},\mathbb{I}^{\bullet}\otimes\mathbb{L}_{\mathcal{X}\times_{S} \widetilde{\mathbf{M}}^s/S})
\een
By composing the map $\mathbb{I}^{\bullet}\to\mathbb{I}^{\bullet}\otimes\mathbb{L}_{\mathcal{X}\times_{S} \widetilde{\mathbf{M}}^s/S}[1]$ with the projection $\mathbb{L}_{\mathcal{X}\times_{S} \widetilde{\mathbf{M}}^s/S}\to \tilde{\pi}_{\widetilde{\mathbf{M}}^s}^*\mathbb{L}_{\widetilde{\mathbf{M}}^s/S}$ and then restricting to  $R\mathcal{H}om(\mathbb{I}^{\bullet},\mathbb{I}^{\bullet})_{0}$, we have a class in
\ben
\mathrm{Ext}^1(R\mathcal{H}om(\mathbb{I}^{\bullet},\mathbb{I}^{\bullet})_{0},\tilde{\pi}_{\widetilde{\mathbf{M}}^s}^*\mathbb{L}_{\widetilde{\mathbf{M}}^s/S}).
\een 
By tensoring the map $R\mathcal{H}om(\mathbb{I}^{\bullet},\mathbb{I}^{\bullet})_{0}\to\tilde{\pi}_{\widetilde{\mathbf{M}}^s}^*\mathbb{L}_{\widetilde{\mathbf{M}}^s/S}[1]$ with $\tilde{\pi}_{\mathcal{X}}^*\omega_{\mathcal{X}/S}$, and using Serre duality in [\cite{Nir1}, Theorem 2.22 and Corollary 2.10], we have the following map
\ben
\widetilde{\Phi}: \widetilde{\mathbf{E}}^\bullet:=R\tilde{\pi}_{\widetilde{\mathbf{M}}^s*}(R\mathcal{H}om(\mathbb{I}^{\bullet},\mathbb{I}^{\bullet})_{0}\otimes\tilde{\pi}_{\mathcal{X}}^*\omega_{\mathcal{X}/S})[2]\to\mathbb{L}_{\widetilde{\mathbf{M}}^s/S}.
\een
where $\omega_{\mathcal{X}/S}$ is the relative dualizing sheaf of $\mathcal{X}\to S$.

\begin{theorem}\label{perf-ob1}
	In the sense of [\cite{BF}], the map $\widetilde{\Phi}$ is a perfect relative obstruction theory for $\widetilde{\mathbf{M}}^s$ over $S$. And there exists a virtual fundamental class $[\widetilde{\mathbf{M}}^s]^{\mathrm{vir}}\in A_{\mathrm{vdim}+\dim S}(\widetilde{\mathbf{M}}^s)$ of virtual dimension $\mathrm{vdim}=\mathrm{rk}(\widetilde{\mathbf{E}}^\bullet)$.
\end{theorem}
\begin{proof}
Combined with the relative version of [\cite{BF}, Theorem 4.5], the argument in the proof of [\cite{Lyj}, Theorem 5.17] can be extended  to the relative case to obtain the first statement. The existence of the virtual fundamental class follows from Behrend-Fantechi [\cite{BF}] and Li-Tian [\cite{LT}].
\end{proof}

By [\cite{BF}, Proposition 7.2], we have the following property of pullback for a perfect relative obstruction theory as  in [\cite{Zhou1}, Proposition 6.5].
\begin{proposition}
Suppose that $S$ is smooth of constant dimension. Then for each closed point $\bar{s}\in S$, the pull-back $i_{\bar{s}}^*\widetilde{\mathbf{E}}^\bullet$ is a perfect obstruction theory for the fiber  $\widetilde{\mathbf{M}}^s_{\bar{s}}$ and $[\widetilde{\mathbf{M}}^s_{\bar{s}}]^{\mathrm{vir}}=i_{\bar{s}}^![\widetilde{\mathbf{M}}^s]^{\mathrm{vir}}$. 
\end{proposition}

\section{ Moduli  of stable orbifold Pandharipande-Thomas pairs}

In Section 4, we have defined an orbifold Pandharipande-Thomas stable pair $(\mathcal{F},\varphi)$ on $\mathcal{X}/S$, where $[(\mathcal{F},\varphi)]$ is one point in $\mathrm{PT}_{\mathcal{X}/S}^{(P)}$. As we will  introduce another notion of stability in Definition \ref{stable} which is different from the stability of a pair $(\mathcal{F},\varphi)$  in Definition \ref{semi-sub}, we make the convention that for a projective Deligne-Mumford stack $\mathcal{Z}$ over $\mathbb{C}$, we call a pair $(\mathcal{H},\psi)$ on $\mathcal{Z}$ or $\psi:\mathcal{O}_{\mathcal{Z}}\to\mathcal{H}$ an orbifold PT pair (it is called a coherent system in [\cite{LW}]) if $\mathcal{H}$ is a pure sheaf of dimension one and $\mathrm{coker}\,\psi$ has finite $0$-dimensional support. Here, we do not impose the condition of  fixed modified Hilbert polynomial on an orbifold PT pair. And we call a representative $(\mathcal{F},\varphi)$ of an object in $\mathrm{PT}_{\mathcal{X}/S}^{(P)}(T)$ an $T$-flat family of orbifold PT pairs on $\mathcal{X}\times_{S}T$ (with fixed modified Hilbert polynomial $P$) as in Definition \ref{fam-PT}.

In the following, we will assume that $(\mathcal{Y},\mathcal{D})$ is a smooth pair and $\mathcal{Y}$
is also a projective Deligne-Mumford stack with a coarse moduli scheme $Y$ and  a  polarization $(\mathcal{E}_{\mathcal{Y}},\mathcal{O}_{Y}(1))$, and $\pi:\mathcal{X}\to\mathbb{A}^1$ is a simple degeneration and also a family of projective Deligne-Mumford stacks with a coarse moduli scheme $X$ and a relative polarization $(\mathcal{E}_{\mathcal{X}},\mathcal{O}_{X}(1))$. Then $\mathcal{X}_{0}[k]$ and $\mathcal{Y}[k]$ are projective Deligne-Mumford stacks by the projective assumption, the projectivity of the bubble component $\Delta$  in [\cite{Zhou1}, Proposition 4.13] and the uniqueness of a coarse moduli space (as a categorical quotient) via gluing (see [\cite{KM}, Lemma 3.2] or [\cite{Con}, Section 2]) for any $k\geq0$. Due to  properties  $(ii)$ in both Proposition \ref{property1} and Proposition \ref{property2},  one can actually show that two standard families $\pi_{k}:\mathcal{X}(k)\to\mathbb{A}^{k+1}$  and $\pi_{k}: \mathcal{Y}(k)\to\mathbb{A}^k$ are  families of projective Deligne-Mumford stacks for any $k\geq0$ as follows. Using Proposition \ref{property1} and Proposition \ref{property2} again,  $\mathcal{X}(k)$ and $\mathcal{Y}(k)$ are  separated Deligne-Mumford stacks of finite type,  and hence  have  their coarse moduli spaces respectively. Furthermore, $\mathcal{X}(k)$ and $\mathcal{Y}(k)$  have projective coarse moduli schemes over $\mathbb{A}^{k+1}$ and $\mathbb{A}^k$ respectively by their groupoid construction and [\cite{Rydh}, Theorem 5.3 and Proposition 4.7].  Using [\cite{OS03}, Section 5]  or [\cite{Nir1}, Theorem 2.16-(2)], $\mathcal{X}(k)$ (resp. $\mathcal{Y}(k)$) possesses a generating sheaf $p^*\mathcal{E}_{\mathcal{X}}$ (resp. $p^*\mathcal{E}_{\mathcal{Y}}$) via the natural map $p:\mathcal{X}(k)\to\mathcal{X}$ (resp. $p:\mathcal{Y}(k)\to\mathcal{Y}$). This completes the proof of projectivity of the above two families. Next, we also assume that $\dim\mathcal{X}_{t}=\dim\mathcal{Y}=3$ for any $t\in\mathbb{A}^1$.

In this section, we will introduce the notion of stable orbifold PT pairs, and define moduli stacks of stable orbifold PT pairs on stacks of expanded degenerations and pairs with some fixed topological data which are  the desired good degenerations, and then prove some properties of these moduli stacks.

\subsection{Stable orbifold PT pairs}
We follow [\cite{LW,Zhou1}] to define the stability condition for an orbifold PT pair.
We consider the following orbifold PT pair on $\mathcal{X}_{0}[k]$
\ben
\varphi:\mathcal{O}_{\mathcal{X}_{0}[k]}\to \mathcal{F}
\een
 (resp. $\varphi:\mathcal{O}_{\mathcal{Y}[k]}\to \mathcal{F}$ on $\mathcal{Y}[k]$) where  $\mathcal{F}$ is pure of dimension one  and $\varphi$ has finite 0-dimensional cokernel.
Given  two orbifold PT pairs $\varphi_{1}:\mathcal{O}_{\mathcal{X}_{0}[k]}\to \mathcal{F}_{1}$  and $\varphi_{2}:\mathcal{O}_{\mathcal{X}_{0}[k]}\to \mathcal{F}_{2}$, we define an equivalence between them  by a couple $(\sigma,\psi)$ such that the following diagram is commutative:
\ben
\xymatrixcolsep{4pc}\xymatrix{
	\mathcal{O}_{\mathcal{X}_{0}[k]}\ar[d]_{\mathrm{Id}} \ar[r]^{\varphi_{1}} &  \mathcal{F}_{1} \ar[d]^{\psi}  \\
	\mathcal{O}_{\mathcal{X}_{0}[k]}\cong\sigma^*\mathcal{O}_{\mathcal{X}_{0}[k]}\ar[r]^-{\sigma^*\varphi_2}&\sigma^{*}\mathcal{F}_{2} &
}
\een
where $\sigma:\mathcal{X}_{0}[k]\to \mathcal{X}_{0}[k]$ is an isomorphism induced by the $(\mathbb{C}^{*})^k$-action on $\mathcal{X}_{0}[k]$ (see Proposition \ref{property1} for the action) and $\psi:\mathcal{F}_{1}\to\sigma^{*}\mathcal{F}_{2}$ is an isomorphism. One can similarly  define an equivalence  between $\varphi_{1}:\mathcal{O}_{\mathcal{Y}[k]}\to \mathcal{F}_{1}$ and $\varphi_{2}:\mathcal{O}_{\mathcal{Y}[k]}\to \mathcal{F}_{2}$. 
Let $\mathrm{Aut}(\varphi)$ be the group of autoequivalence of a fixed orbifold PT pair $\varphi:\mathcal{O}_{\mathcal{X}_{0}[k]}\to \mathcal{F}$ (resp. $\varphi:\mathcal{O}_{\mathcal{Y}[k]}\to \mathcal{F}$). It is easy to show as in [\cite{LW}, Section 4.1] or [\cite{Zhou1}, Section 3.4] that $\mathrm{Aut}(\varphi)$
is a subgroup of $(\mathbb{C}^{*})^k$.

As in [\cite{LW}, Definition 4.8], we have
\begin{definition}\label{stable}
	An orbifold PT pair $\varphi:\mathcal{O}_{\mathcal{X}_{0}[k]}\to \mathcal{F}$ (resp. $\varphi:\mathcal{O}_{\mathcal{Y}[k]}\to \mathcal{F}$) is called admissible if both $\mathrm{coker}\,\varphi$ and $\mathcal{F}$ are admissible. We call  $\varphi$  stable if it is admissible and $\mathrm{Aut}(\varphi)$ is finite.
\end{definition}
\begin{remark}\label{adm-cri}
For an orbifold PT pair $\varphi:\mathcal{O}_{\mathcal{X}_{0}[k]}\to \mathcal{F}$ (resp. $\varphi:\mathcal{O}_{\mathcal{Y}[k]}\to \mathcal{F}$), it is shown in [\cite{Zhou1}, Proposition 3.23 and Example 3.24] that  $\mathrm{coker}\,\varphi$ is admissible if and only if any point in  $\mathrm{Supp}(\mathrm{coker}\,\varphi)$ is away from the divisor $\mathcal{D}_{i}$, $0\leq i\leq k$,
and the pure 1-dimensional sheaf $\mathcal{F}$ is admissible if and  only if no irreducible components of $\mathrm{Supp}(\mathcal{F})$ lie entirely in any divisor $\mathcal{D}_{i}$, $0\leq i\leq k$. 
\end{remark}	
We make the convention that any orbifold PT pair on the fiber of $\mathcal{X}(k)\to\mathbb{A}^{k+1}$  over $t=(t_{0},\cdots,t_{k})\in\mathbb{A}^{k+1}$ with $t_{0}\cdots t_{k}\ne 0$  which is isomorphic to the smooth fiber $\mathcal{X}_{c}$ ($c\ne0$) in the original family $\pi:\mathcal{X}\to\mathbb{A}^1$ by Proposition \ref{property1},
is admissible and stable. 
For the families of expanded degenerations and pairs, we have the following definition.
\begin{definition}\label{fam-PT}
Let $\mathtt{X}_{S}\to S$ (resp. $\mathtt{Y}_{S}\to S$) be a family of expanded degenerations (resp. pairs). Suppose $\varphi: \mathcal{O}_{\mathtt{X}_{S}}\to\mathcal{F}$ (resp. $ \mathcal{O}_{\mathtt{Y}_{S}}\to\mathcal{F}$) is  an $S$-flat family of  orbifold PT pairs on $\mathtt{X}_{S}$ (resp. $\mathtt{Y}_{S}$), that is, $\mathcal{F}$ is an $S$-flat family of pure 1-dimensional sheaves and $\mathrm{coker}\,\varphi$ has relative dimension at most zero. 
The family $\varphi$ is called an $S$-flat family of stable orbifold PT pairs  if for every point $s\in S$, the orbifold PT pair $\varphi_{s}: \mathcal{O}_{\mathtt{X}_{S,s}}\to\mathcal{F}_{s}$ (resp. $\mathcal{O}_{\mathtt{Y}_{S,s}}\to\mathcal{F}_{s}$) is stable on the fiber $\mathtt{X}_{S,s}$ (resp. $\mathtt{Y}_{S,s}$).
\end{definition}

By  the discussions at the begining of Section 4 and this section, we have the moduli spaces $\mathrm{PT}_{\mathcal{X}(k)/\mathbb{A}^{k+1}}^{(P)}$ and $\mathrm{PT}_{\mathcal{Y}(k)/\mathbb{A}^{k}}^{(P)}$, which are projective schemes over $\mathbb{A}^{k+1}$ and $\mathbb{A}^k$ respectively. We begin with the degeneration case.
Consider the moduli space of orbifold PT pairs on $\mathcal{X}(k)/\mathbb{A}^{k+1}$, denoted by $\mathrm{PT}_{\mathcal{X}(k)/\mathbb{A}^{k+1}}$, which is a disjoint union of 
$\mathrm{PT}_{\mathcal{X}(k)/\mathbb{A}^{k+1}}^{(P)}$ (it is possible to be empty) where  $P$ is taken over all degree one polynomials. Then $\mathrm{PT}_{\mathcal{X}(k)/\mathbb{A}^{k+1}}$ is  a scheme locally of finite type  over $\mathbb{A}^{k+1}$, and each  $\mathrm{PT}_{\mathcal{X}(k)/\mathbb{A}^{k+1}}^{(P)}\subset\mathrm{PT}_{\mathcal{X}(k)/\mathbb{A}^{k+1}}$ is a projective $\mathbb{A}^{k+1}$-scheme.  Similarly, one can define $\mathrm{PT}_{\mathcal{Z}/\mathbb{C}}$ as a disjoint union of $\mathrm{PT}_{\mathcal{Z}/\mathbb{C}}^{(P)}$ (see Section 4) over all polynomials $P$ of degree one for a $3$-dimensional projective Deligne-Mumford stack $\mathcal{Z}$  over $\mathbb{C}$. And $\mathrm{PT}_{\mathcal{Z}/\mathbb{C}}$ is locally of finite type over $\mathbb{C}$ with each $\mathrm{PT}_{\mathcal{Z}/\mathbb{C}}^{(P)}\subset\mathrm{PT}_{\mathcal{Z}/\mathbb{C}}$ being a projective $\mathbb{C}$-scheme. We have a $(\mathbb{C}^*)^k$-action on $\mathrm{PT}_{\mathcal{X}(k)/\mathbb{A}^{k+1}}$ induced from the $(\mathbb{C}^*)^k$-action on $\mathcal{X}(k)$ and $\mathbb{A}^{k+1}$. As in [\cite{Zhou1}, Section 3.4], given an object $(\xi,\varphi)\in\mathrm{PT}_{\mathcal{X}(k)/\mathbb{A}^{k+1}}(S)$ where $\xi:S\to \mathbb{A}^{k+1}$ is  a  $\mathbb{A}^1$-map  and $\varphi$ is an $S$-flat family of orbifold $\mathrm{PT}$ pairs on $\mathtt{X}_{S}$, one can define the stabilizer as the following fiber product:
\ben
\xymatrixcolsep{4pc}\xymatrix{
	\mathrm{Aut}_{(\mathbb{C}^*)^k,S}(\xi,\varphi)\ar[d] \ar[r] &  S \ar[d]^{\Delta\circ(\xi,\varphi)}  \\
	(\mathbb{C}^*)^k\times\mathrm{PT}_{\mathcal{X}(k)/\mathbb{A}^{k+1}}\ar[r]& \mathrm{PT}_{\mathcal{X}(k)/\mathbb{A}^{k+1}}\times\mathrm{PT}_{\mathcal{X}(k)/\mathbb{A}^{k+1}} &
}
\een
Then $\mathrm{Aut}_{(\mathbb{C}^*)^k,S}(\xi,\varphi)\hookrightarrow(\mathbb{C}^*)^k\times S$ is a subgroup scheme over $S$, and hence  quasi-compact and separated over $S$. Similarly, one can also define $\mathrm{PT}_{\mathcal{Y}(k)/\mathbb{A}^{k}}$, which is locally of finite type over  $\mathbb{A}^k$ with each  $\mathrm{PT}_{\mathcal{Y}(k)/\mathbb{A}^{k}}^{(P)}\subset\mathrm{PT}_{\mathcal{Y}(k)/\mathbb{A}^{k}}$ being a projective $\mathbb{A}^k$-scheme, and the stabilizer of an object in $\mathrm{PT}_{\mathcal{Y}(k)/\mathbb{A}^{k}}(S)$ is a quasi-compact and separated subgroup scheme  of $(\mathbb{C}^*)^k\times S$ over $S$.
Combined with Proposition \ref{adm-open}, as in [\cite{Zhou1}, Proposition 3.20], we have
\begin{proposition}\label{open}
Let $\mathtt{X}_{S}\to S$ $($resp. $\mathtt{Y}_{S}\to S$$)$ be a family of expanded degenerations (resp. pairs), and let $\varphi: \mathcal{O}_{\mathtt{X}_{S}}\to\mathcal{F}$ $($resp. $ \mathcal{O}_{\mathtt{Y}_{S}}\to\mathcal{F}$$)$ be an $S$-flat family of orbifold PT pairs on $\mathtt{X}_{S}$ (resp. $\mathtt{Y}_{S}$). Then
$\{s\in S\,|\,\varphi_{s}:\mathcal{O}_{\mathtt{X}_{S,s}}\to\mathcal{F}_{s} \mbox{ is stable}\}$
$($resp. $\{s\in S\,|\,\varphi_{s}:\mathcal{O}_{\mathtt{Y}_{S,s}}\to\mathcal{F}_{s} \mbox{ is stable}\}$$)$
is  open in S.
\end{proposition}

\subsection{Moduli of stable orbifold PT pairs}

We study stable orbifold PT pairs on stacks of expanded degenerations $\mathfrak{C}$ and pairs $\mathfrak{A}$ together with their universal families $\mathfrak{X}$ and $\mathfrak{Y}$ respectively as in  [\cite{Zhou1}, Section 4.1]. We begin with the following Cartesian diagram for the discrete symmetries
\ben
\begin{tikzpicture}
\node (a) at (0,0) {$R_{\mathrm{discrete},\mathcal{X}(k)}$};
\node (b) at (3,0) {$\mathcal{X}(k)$};
\node (c) at (0,-1.5) {$R_{\mathrm{discrete},\mathbb{A}^{k+1}}$};
\node (d) at (3,-1.5) {$\mathbb{A}^{k+1}$};
\path[->,font=\scriptsize,>=angle 90]
([yshift= 2pt]a.east) edge  ([yshift= 2pt]b.west)
([yshift= -2pt]a.east) edge  ([yshift= -2pt]b.west)
(a) edge (c)
(b) edge (d)
([yshift= 2pt]c.east) edge  ([yshift= 2pt]d.west)
([yshift= -2pt]c.east) edge  ([yshift= -2pt]d.west);
\end{tikzpicture}
\een
which is equivalent to the obvious relation $R_{\mathrm{discrete},\mathcal{X}(k)}=\pi_{k}^*R_{\mathrm{discrete},\mathbb{A}^{k+1}}$ by Proposition \ref{property1} where $\pi_{k}:\mathcal{X}(k)\to\mathbb{A}^{k+1}$.
This diagram induces an $\acute{e}$tale equivalence relation
\ben
\begin{tikzpicture}
\node (a) at (0,0) {$R_{\mathrm{discrete},\mathrm{PT}_{\mathcal{X}(k)/\mathbb{A}^{k+1}}}:=\mathrm{PT}_{R_{\mathrm{discrete},\mathcal{X}(k)}/R_{\mathrm{discrete},\mathbb{A}^{k+1}}}$};
\node (b) at (6,0) {$\mathrm{PT}_{\mathcal{X}(k)/\mathbb{A}^{k+1}}$};
\path[->,font=\scriptsize,>=angle 90]
([yshift= 2pt]a.east) edge  ([yshift= 2pt]b.west)
([yshift= -2pt]a.east) edge  ([yshift= -2pt]b.west);
\end{tikzpicture}
\een
which factors through $S_{k+1}\times\mathrm{PT}_{\mathcal{X}(k)/\mathbb{A}^{k+1}}$ since one can view the discrete relation on $\mathrm{PT}_{\mathcal{X}(k)/\mathbb{A}^{k+1}}$ as a subrelation of $S_{k+1}$-action. Here, the $S_{k+1}$-action on $\mathrm{PT}_{\mathcal{X}(k)/\mathbb{A}^{k+1}}$ is induced from $S_{k+1}$-action on  $\mathbb{A}^{k+1}$ and $\mathcal{X}(k)$  simultaneously by permutation of coordinates. Combined with the $(\mathbb{C}^*)^k$-action
\ben
\begin{tikzpicture}
\node (a) at (0,0) {$(\mathbb{C}^*)^k\times\mathrm{PT}_{\mathcal{X}(k)/\mathbb{A}^{k+1}}$};
\node (b) at (4,0) {$\mathrm{PT}_{\mathcal{X}(k)/\mathbb{A}^{k+1}}$};
\path[->,font=\scriptsize,>=angle 90]
([yshift= 2pt]a.east) edge  ([yshift= 2pt]b.west)
([yshift= -2pt]a.east) edge  ([yshift= -2pt]b.west);
\end{tikzpicture}
\een
we have the following smooth equivalence relation on $\mathrm{PT}_{\mathcal{X}(k)/\mathbb{A}^{k+1}}$:
\ben
\begin{tikzpicture}
\node (a) at (0,0) {$R_{\sim,\mathrm{PT}_{\mathcal{X}(k)/\mathbb{A}^{k+1}}}:=(\mathbb{C}^*)^k\times R_{\mathrm{discrete},\mathrm{PT}_{\mathcal{X}(k)/\mathbb{A}^{k+1}}}$};
\node (b) at (6.7,0) {$(\mathbb{C}^*)^{k+1}\rtimes S_{k+1}\times\mathrm{PT}_{\mathcal{X}(k)/\mathbb{A}^{k+1}}$};
\node (c) at (10.8,0) {$\mathrm{PT}_{\mathcal{X}(k)/\mathbb{A}^{k+1}}\;.$};
\path[right hook->] (a) edge (b);
\path[->,font=\scriptsize,>=angle 90]
([yshift= 2pt]b.east) edge  ([yshift= 2pt]c.west)
([yshift= -2pt]b.east) edge  ([yshift= -2pt]c.west);
\end{tikzpicture}
\een
where $(\mathbb{C}^*)^k$-action on $\mathrm{PT}_{\mathcal{X}(k)/\mathbb{A}^{k+1}}$ is viewed as some $(\mathbb{C}^*)^{k+1}$-action via the natural map $(\mathbb{C}^*)^k\hookrightarrow(\mathbb{C}^*)^{k+1}$ as in Section 2.2.
Also, we have the following Cartesian diagram with $|I|=l+1\leq k+1$:
\ben
	\xymatrixcolsep{4pc}\xymatrix{
		\mathcal{X}(l)\ar@{^{(}->}[r] \ar[d]^{\pi_{l}}  &    \mathcal{X}(k)|_{U_{I}}  \ar[d]\ar@{^{(}->}[r] & \mathcal{X}(k) \ar[d]^{\pi_{k}} \\
	\mathbb{A}^{l+1} \ar@{^{(}->}[r]^{\tau_{I}}& U_{I} \ar@{^{(}->}[r]^{\widehat{\tau}_{I}}& \mathbb{A}^{k+1}
	}
\een
which induces the closed embedding $\mathrm{PT}_{\mathcal{X}(l)/\mathbb{A}^{l+1}}\hookrightarrow \mathrm{PT}_{\mathcal{X}(k)/\mathbb{A}^{k+1}}$ by taking  objects $(\xi,\varphi)\in \mathrm{PT}_{\mathcal{X}(k)/\mathbb{A}^{k+1}}(S)$ with $\mathbb{A}^1$-map $\xi:S\to \mathbb{A}^{k+1}$ factoring through $\tau_{I}:\mathbb{A}^{l+1}\hookrightarrow\mathbb{A}^{k+1}$  and the open embedding
\ben
(\mathbb{C}^*)^{k-l}\times\mathrm{PT}_{\mathcal{X}(l)/\mathbb{A}^{l+1}}\cong \mathrm{PT}_{\mathcal{X}(k)|_{U_{I}}/U_{I}} \hookrightarrow\mathrm{PT}_{\mathcal{X}(k)/\mathbb{A}^{k+1}}.
\een
Then we have the following diagram
\ben
\begin{tikzpicture}
\node (a) at (0,0) {$R_{\sim,\mathrm{PT}_{\mathcal{X}(l)/\mathbb{A}^{l+1}}}$};
\node (b) at (4,0) {$(\mathbb{C}^*)^{k-l}\times R_{\sim,\mathrm{PT}_{\mathcal{X}(l)/\mathbb{A}^{l+1}}}$};
\node (c) at (8,0) {$R_{\sim,\mathrm{PT}_{\mathcal{X}(k)/\mathbb{A}^{k+1}}}$};
\node (d) at (0,-1.5) {$\mathrm{PT}_{\mathcal{X}(l)/\mathbb{A}^{l+1}}$};
\node (e) at (4,-1.5) {$(\mathbb{C}^*)^{k-l}\times\mathrm{PT}_{\mathcal{X}(l)/\mathbb{A}^{l+1}}$};
\node (f) at (8,-1.5) {$\mathrm{PT}_{\mathcal{X}(k)/\mathbb{A}^{k+1}}$};
\path[->,font=\scriptsize,>=angle 90]
([xshift= 2pt]a.south) edge  ([xshift= 2pt]d.north)
([xshift= -2pt]a.south) edge  ([xshift= -2pt]d.north)
([xshift= 2pt]b.south) edge  ([xshift= 2pt]e.north)
([xshift= -2pt]b.south) edge  ([xshift= -2pt]e.north)
([xshift= 2pt]c.south) edge  ([xshift= 2pt]f.north)
([xshift= -2pt]c.south) edge  ([xshift= -2pt]f.north);
\path[right hook->]
(a) edge (b)
(b) edge (c)
(d) edge (e)
(e) edge (f);
\end{tikzpicture}
\een 
By the Yoneda Lemma, one can also take the projective $\mathbb{A}^{k+1}$-scheme  $\mathrm{PT}_{\mathcal{X}(k)/\mathbb{A}^{k+1}}^{(P)}$ as its functor of points.
Let  $\mathrm{PT}_{\mathcal{X}(k)/\mathbb{A}^{k+1}}^{\mathrm{st},(P)}\subset \mathrm{PT}_{\mathcal{X}(k)/\mathbb{A}^{k+1}}^{(P)}$ be the subfunctor parameterizing stable orbifold PT pairs. Then it is an open subfunctor  by Proposition \ref{open} and represented by a quasi-projective $\mathbb{A}^{k+1}$-scheme. One can define  $\mathrm{PT}_{\mathcal{X}(k)/\mathbb{A}^{k+1}}^{\mathrm{st}}$ as a disjoint union of quasi-projective $\mathbb{A}^{k+1}$-schemes $\mathrm{PT}_{\mathcal{X}(k)/\mathbb{A}^{k+1}}^{\mathrm{st},(P)}$ for all possible degree one polyomial $P$, and then it is locally of finite type over $\mathbb{A}^{k+1}$ and invariant under the equivalence relation $R_{\sim,\mathrm{PT}_{\mathcal{X}(k)/\mathbb{A}^{k+1}}}$. We define the moduli stack of stable orbifold PT pairs by
\ben
\mathfrak{PT}_{\mathfrak{X}/\mathfrak{C}}:=\lim\limits_{\longrightarrow}\left[\mathrm{PT}_{\mathcal{X}(k)/\mathbb{A}^{k+1}}^{\mathrm{st}}\bigg/R_{\sim,\mathrm{PT}_{\mathcal{X}(k)/\mathbb{A}^{k+1}}}\right].
\een
Similarly, one can define the following moduli stack in the relative case 
\ben
\mathfrak{PT}_{\mathfrak{Y}/\mathfrak{A}}:=\lim\limits_{\longrightarrow}\left[\mathrm{PT}_{\mathcal{Y}(k)/\mathbb{A}^{k}}^{\mathrm{st}}\bigg/R_{\sim,\mathrm{PT}_{\mathcal{Y}(k)/\mathbb{A}^{k}}}\right].
\een

The stack $\mathfrak{PT}_{\mathfrak{X}/\mathfrak{C}}$  has the following categorical interpretation as in [\cite{Zhou1}, Section 4.1]. For any $\mathbb{A}^1$-scheme $S$, we define $\mathfrak{PT}_{\mathfrak{X}/\mathfrak{C}}(S)$ to be the set of objects $(\overline{\xi},\overline{\varphi})$, where each object $(\overline{\xi},\overline{\varphi})$ is represented by some object $(\xi,\varphi)\in [\mathrm{PT}_{\mathcal{X}(k)/\mathbb{A}^{k+1}}^{\mathrm{st}}/R_{\sim,\mathrm{PT}_{\mathcal{X}(k)/\mathbb{A}^{k+1}}}](S)$ for some $k$ such that by passing to a surjective $\acute{e}$tale  covering, one has the $\mathbb{A}^1$-map $\xi: S_{\xi}=\coprod S_{i}\to\mathbb{A}^{k+1}$ and an $S_{\xi}$-flat family of stable orbifold PT pairs $\varphi:\mathcal{O}_{\mathtt{X}_{S_{\xi}}}\to\mathcal{F}$ on the family of expanded degenerations $\mathtt{X}_{S_{\xi}}$. If $(\xi^\prime,\varphi^\prime)$ is another representative  of $(\overline{\xi},\overline{\varphi})$ where $\xi^\prime: S_{\xi^\prime}\to\mathbb{A}^{k^\prime+1}$, by restricting to the refinement $S_{\xi\xi^\prime}:=S_{\xi}\times_{S}S_{\xi^\prime}$, we have two induced families $\mathtt{X}_{S_{\xi\xi^\prime}}:=\xi^*\mathcal{X}(k)|_{S_{\xi\xi^\prime}}$ and $\mathtt{X}^\prime_{S_{\xi\xi^\prime}}:=\xi^{\prime*}\mathcal{X}(k^\prime)|_{S_{\xi\xi^\prime}}$. By embedding $\mathcal{X}(k)$ and $\mathcal{X}(k^\prime)$ into a common  ambient space $\mathcal{X}(k^{\prime\prime})$ for some $k^{\prime\prime}\geq k,k^\prime$ and pulling back the smooth equivalence relation $R_{\sim,\mathcal{X}(k^{\prime\prime})}$,  we have some isomorphism $\sigma: \mathtt{X}\xrightarrow{\cong}\mathtt{X}^\prime$ which implies that $\mathtt{X}_{S_{\xi\xi^\prime}}$ and $\mathtt{X}^\prime_{S_{\xi\xi^\prime}}$ are isomorphic. Then we have an isomorphism $\varphi\cong\sigma^*\varphi^\prime$
when passing to $S_{\xi\xi^\prime}$ due to the compatibility between different representatives. Given a morphism of $\mathbb{A}^1$-scheme $f: T\to S$, the map $\mathfrak{PT}_{\mathfrak{X}/\mathfrak{C}}(f): \mathfrak{PT}_{\mathfrak{X}/\mathfrak{C}}(S)\to \mathfrak{PT}_{\mathfrak{X}/\mathfrak{C}}(T)$ is defined by pull back as follows.
Choosing a representative $(\xi,\varphi_{S_{\xi}})$ for an object $(\overline{\xi},\overline{\varphi})\in\mathfrak{PT}_{\mathfrak{X}/\mathfrak{C}}(S)$ as above, the 1-arrow maps $(\xi,\varphi_{S_{\xi}})$ to $(\eta,\varphi_{T_{\xi}})$ which represents an object in $\mathfrak{PT}_{\mathfrak{X}/\mathfrak{C}}(T)$ where $f_{\xi}:T_{\xi}:=T\times_{S}S_{\xi}\to S_{\xi}$, $\eta:=\xi\circ f_{\xi}$ and the $T_{\xi}$-flat family of stable orbifold PT pairs $\varphi_{T_{\xi}}$ on $\mathtt{X}_{T_{\xi}}=f_{\xi}^*\mathtt{X}_{S_{\xi}}$ is the pull back of the $S_{\xi}$-flat family $\varphi_{S_{\xi}}:\mathcal{O}_{\mathtt{X}_{S_{\xi}}}\to\mathcal{F}$.
One can interpret the stack $\mathfrak{PT}_{\mathfrak{Y}/\mathfrak{A}}$ in the similar way. 

By using the similar argument in the proof of [\cite{Zhou1}, Theorem 4.1], together with the fact that  $\mathrm{PT}_{\mathcal{X}(k)/\mathbb{A}^{k+1}}$ and  $\mathrm{PT}^{\mathrm{st}}_{\mathcal{X}(k)/\mathbb{A}^{k+1}}$  (resp. $\mathrm{PT}_{\mathcal{Y}(k)/\mathbb{A}^{k}}$ and $\mathrm{PT}^{\mathrm{st}}_{\mathcal{Y}(k)/\mathbb{A}^{k}}$) are locally of finite type in the degeneration case (resp. the relative case), we have
\begin{theorem}\label{stack1}
$(i)$ The stack	$\mathfrak{PT}_{\mathfrak{X}/\mathfrak{C}}$ is a Deligne-Mumford stack, locally of finite type over $\mathbb{A}^1$;

$(ii)$ The stack $\mathfrak{PT}_{\mathfrak{Y}/\mathfrak{A}}$ is a Deligne-Mumford stack, locally of finite type over $\mathbb{C}$.
\end{theorem}

\subsection{Moduli of stable orbifold PT pairs with fixed Hilbert homomorphism}
We recall the definition of  Hilbert homomorphism  in [\cite{Zhou1}, Section 4.3], then consider moduli stacks defined in the previous subsection with this fixed topological data. 

For a simple degeneration $\pi:\mathcal{X}\to\mathbb{A}^1$, by assumption it is also a family of projective Deligne-Mumford stacks with a moduli scheme $X$ and a relative polarization $(\mathcal{E}_{\mathcal{X}},\mathcal{O}_{X}(1))$. For $0\neq c\in\mathbb{A}^1$, we define the Hilbert  homomorphism (see Definition \ref{h-h}) with respect to $\mathcal{E}_{\mathcal{X}_{c}}$ by a group homomorphism $P_{\mathcal{G}_{c}}^{\mathcal{E}_{\mathcal{X}_{c}}}: K^0(\mathcal{X}_{c})\to\mathbb{Z}$: 
\ben
[V]\mapsto\chi(\mathcal{X}_{c}, V\otimes_{\mathcal{O}_{\mathcal{X}_{c}}}\mathcal{G}_{c}\otimes_{\mathcal{O}_{\mathcal{X}_{c}}}\mathcal{E}_{\mathcal{X}_{c}}^\vee)
\een
with $V$ being a vector bundle on $\mathcal{X}_{c}$ and extended to $K^0(\mathcal{X}_{c})$ additively, where $\mathcal{G}_{c}$ is a coherent sheaf on $\mathcal{X}_{c}$ and $\mathcal{E}_{\mathcal{X}_{c}}$ is a generating sheaf for $\mathcal{X}_{c}$ (see [\cite{Lyj}, Remark 2.8]). In the case for the central fiber $\mathcal{X}_{0}$, 
we define a group homomorphism $P_{\mathcal{G}_{0}}^{\mathcal{E}_{\mathcal{X}_{0}}}:K^0(\mathcal{X}_{0})\to\mathbb{Z}$ (it is different from Definition \ref{h-h}) as
\ben
[V]\mapsto\chi(\mathcal{X}_{0}[l], p^*V\otimes_{\mathcal{O}_{\mathcal{X}_{0}[l]}}\mathcal{G}_{0}\otimes_{\mathcal{O}_{\mathcal{X}_{0}[l]}}p^*\mathcal{E}_{\mathcal{X}_{0}}^\vee)
\een
with a vector bundle $V$  on $\mathcal{X}_{0}$ and extended to $K^0(\mathcal{X}_{0})$ additively, where  $\mathcal{G}_{0}$ is a coherent sheaf on $\mathcal{X}_{0}[l]$ and $p:\mathcal{X}_{0}[l]\to\mathcal{X}_{0}$ is the contraction map. Here, $\mathcal{E}_{\mathcal{X}_{0}}$ is a generating sheaf for $\mathcal{X}_{0}$, and hence $p^*\mathcal{E}_{\mathcal{X}_{0}}$ is also a generating sheaf for $\mathcal{X}_{0}[l]$ (see [\cite{Lyj}, Theorem 2.5]). We still call $P_{\mathcal{G}_{0}}^{\mathcal{E}_{\mathcal{X}_{0}}}$ 
the Hilbert homomorphism with respect to $\mathcal{E}_{\mathcal{X}_{0}}$.
Now we can define the Hilbert homomorphism $K^0(\mathcal{X})\to\mathbb{Z}$ with respect to $\mathcal{E}_{\mathcal{X}}$ by composing the restriction $K^0(\mathcal{X})\to K^0(\mathcal{X}_{c})$ ($c\neq0$) and $K^0(\mathcal{X})\to K^0(\mathcal{X}_{0})$ with  Hilbert homomorphisms $P_{\mathcal{G}_{c}}^{\mathcal{E}_{\mathcal{X}_{c}}}$ and $P_{\mathcal{G}_{0}}^{\mathcal{E}_{\mathcal{X}_{0}}}$ respectively.

It is clear that $P_{\mathcal{G}_{c}}^{\mathcal{E}_{\mathcal{X}_{c}}}(H_{c}^{\otimes m})=P_{\mathcal{E}_{\mathcal{X}_{c}}}(\mathcal{G}_{c})(m)$ is the modified Hilbert polynomial of $\mathcal{G}_{c}$ on the smooth fiber $\mathcal{X}_{c}$ ($c\neq0$) with the moduli scheme $\pi_{c}:\mathcal{X}_{c}\to X_{c}$ where $H_{c}:=\pi_{c}^*\mathcal{O}_{X_{c}}(1)$. But $P_{\mathcal{G}_{0}}^{\mathcal{E}_{\mathcal{X}_{0}}}(H_{0}^{\otimes m})$ may not be a modified Hilbert polynomial of $\mathcal{G}_{0}$ on $\mathcal{X}_{0}[l]$ since $p^*H_{0}$ is possibly not the pull back of some ample line bundle on the moduli scheme of $\mathcal{X}_{0}[l]$ where $H_{0}=\pi_{0}^*\mathcal{O}_{X_{0}}(1)$ with $\pi_{0}:\mathcal{X}_{0}\to X_{0}$.
With the Hilbert homomorphism $K^0(\mathcal{X})\to\mathbb{Z}$ defined above, we simply call $P_{\mathcal{G}_{t}}^{\mathcal{E}_{\mathcal{X}_{t}}}(H_{t}^{\otimes m})$ $($$t\in\mathbb{A}^1$$)$ the generalized Hilbert polynomial.

Let $P: K^0(\mathcal{X})\to\mathbb{Z}$ be a fixed Hilbert homomorphism with respect to $\mathcal{E}_{\mathcal{X}}$. Define 
$\mathrm{PT}_{\mathcal{X}(k)/\mathbb{A}^{k+1}}^{P}\subset \mathrm{PT}_{\mathcal{X}(k)/\mathbb{A}^{k+1}}$ to be the subfunctor parameterizing orbifold PT pairs with fixed Hilbert homomorphism $P$. Then $\mathrm{PT}_{\mathcal{X}(k)/\mathbb{A}^{k+1}}^{P}$ is an open and closed subfunctor, and is represented by (or is viewed by the Yoneda Lemma as) an  open and closed subscheme of $\mathrm{PT}_{\mathcal{X}(k)/\mathbb{A}^{k+1}}$. Similarly, one has the subfunctor $\mathrm{PT}_{\mathcal{X}(k)/\mathbb{A}^{k+1}}^{\mathrm{st},P}\subset \mathrm{PT}_{\mathcal{X}(k)/\mathbb{A}^{k+1}}^{\mathrm{st}}$, and is (represented by) an open and closed subscheme of $\mathrm{PT}_{\mathcal{X}(k)/\mathbb{A}^{k+1}}^{\mathrm{st}}$.

For a family of expanded degenerations $\pi:\mathtt{X}_{S}\to S$ over $(\xi, S)$ and an $S$-flat coherent sheaf $\mathcal{F}$ on $\mathtt{X}_{S}$, one has the Hilbert homomorphism $P_{\mathcal{F}_{s}}^{\mathcal{E}_{\mathcal{X}_{\mathbf{m}(\xi(s))}}}: K^0(\mathcal{X}_{\mathbf{m}(\xi(s))})\to\mathbb{Z}$
on the fiber over any point $s\in S$ where $\xi: S\to \mathbb{A}^{k+1}$  is a $\mathbb{A}^1$-map  and  $\mathbf{m}:\mathbb{A}^{k+1}\to\mathbb{A}^1$ is the multiplication map. When $\mathbf{m}(\xi(s))=0$, $\mathcal{F}_{s}$ is defined on the fiber over $s$ which is isomorphic to $\mathcal{X}_{0}[l]$, where $(l+1)$ is equal to the number of zeros in the coordinate of $\xi(s)$ by Proposition \ref{property1}. Since $S$ is connected, the similar argument in the proof of [\cite{OS03}, Lemma 4.3] shows that there is a group homomorphism $P:K^0(\mathcal{X})\to\mathbb{Z}$ such that for any $s\in S$ we have $P_{\mathcal{F}_{s}}^{\mathcal{E}_{\mathcal{X}_{\mathbf{m}(\xi(s))}}}=P$.
With the categorical interpretation of the stack $\mathfrak{PT}_{\mathfrak{X}/\mathfrak{C}}$, for a fixed group homomorphism $P:K^{0}(\mathcal{X})\to\mathbb{Z}$, we define $\mathfrak{PT}_{\mathfrak{X}/\mathfrak{C}}^{P}$ to be the subfunctor of $\mathfrak{PT}_{\mathfrak{X}/\mathfrak{C}}$ parameterizing stable orbifold PT pairs with fixed Hilbert homomorphism $P$. Then $\mathfrak{PT}_{\mathfrak{X}/\mathfrak{C}}^{P}$ is an open and closed subfunctor of $\mathfrak{PT}_{\mathfrak{X}/\mathfrak{C}}$ (it is  empty unless the associated generalized Hilbert polynomial is of degree one). By 
Theorem \ref{stack1},  $\mathfrak{PT}_{\mathfrak{X}/\mathfrak{C}}^{P}$ is  a Deligne-Mumford stack locally of finite type over $\mathbb{A}^1$.

Similarly, in the relative case,  the  Hilbert homomorphism with respect to $\mathcal{E}_{\mathcal{Y}}$ defined in [\cite{Zhou1}, Section 4.3] as a group homomorphism $P_{\mathcal{G}}^{\mathcal{E}_{\mathcal{Y}}}: K^0(\mathcal{Y})\to\mathbb{Z}$:
\ben
[V]\mapsto\chi(\mathcal{Y}[l],p^*V\otimes_{\mathcal{O}_{\mathcal{Y}[l]}}\mathcal{G}\otimes_{\mathcal{O}_{\mathcal{Y}[l]}} p^*\mathcal{E}_{\mathcal{Y}}^\vee)
\een
where $\mathcal{G}$ is a coherent sheaf on $\mathcal{Y}[l]$ and $p: \mathcal{Y}[l]\to\mathcal{Y}$ is the contraction map. Again, we call $P_{\mathcal{G}}^{\mathcal{E}_{\mathcal{Y}}}(H^{\otimes m})$ the generalized Hilbert polynomial  where $H:=\pi^*\mathcal{O}_{Y}(1)$.
One can similarly define the subfunctor $\mathfrak{PT}_{\mathfrak{Y}/\mathfrak{A}}^{P}$ of $\mathfrak{PT}_{\mathfrak{Y}/\mathfrak{A}}$, which is  also a Deligne-Mumford stack locally of finite type over $\mathbb{C}$.
\begin{remark}\label{absolute}
If $\mathcal{Z}$ is a $3$-dimensional smooth projective Deligne-Mumford stack over $\mathbb{C}$, we have defined the moduli space $\mathrm{PT}_{\mathcal{Z}/\mathbb{C}}$ in Section 5.1. Similarly, one can define $\mathrm{PT}_{\mathcal{Z}/\mathbb{C}}^P\subset\mathrm{PT}_{\mathcal{Z}/\mathbb{C}}$ as a subfunctor parameterizing orbifold PT pairs with fixed Hilbert homomorphism $P$ (see Definition \ref{h-h}) as above. If the associated generalized Hilbert polynomial with repect to the Hilbert homomorphism $P$ in $\mathrm{PT}_{\mathcal{Z}/\mathbb{C}}^P$ is exactly the modified Hilbert polynomial (denoted also by) $P$ in the superscript of $\mathrm{PT}_{\mathcal{Z}/\mathbb{C}}^{(P)}$, then by Definitions \ref{m-h-p} and \ref{h-h}, $\mathrm{PT}_{\mathcal{Z}/\mathbb{C}}^P$ is equal to $\mathrm{PT}_{\mathcal{Z}/\mathbb{C}}^{(P)}$, which is a projective scheme over $\mathbb{C}$.
\end{remark}

\subsection{Properties of the moduli of stable orbifold PT pairs}
In this subsection, we will investigate the further properties of moduli stacks $\mathfrak{PT}_{\mathfrak{X}/\mathfrak{C}}^{P}$ and  $\mathfrak{PT}_{\mathfrak{Y}/\mathfrak{A}}^{P}$ including the boundedness, separatedness, and properness. These three properties has been proved for  moduli stacks of stable quotient and of stable PT pairs (or stable coherent systems) in [\cite{LW}, Section 5] for the variety case, where the former has been generalized to the case of Deligne-Mumford stacks in [\cite{Zhou1}, Section 5]. We will show that the latter can be also generalized to the stacky case by combining their arguments.

\subsubsection{Boundedness}
Due to the categorical interpretation of the stack $\mathfrak{PT}_{\mathfrak{X}/\mathfrak{C}}^P$  (resp. $\mathfrak{PT}_{\mathfrak{Y}/\mathfrak{A}}^{P}$) as in Section 5.2, the representatives of its objects we will consider are  $S_{\xi}$-flat families of stable orbifold PT pairs $\varphi: \mathcal{O}_{\mathtt{X}_{S_{\xi}}}\to\mathcal{F}$ on $\mathtt{X}_{S_{\xi}}$ (resp. $\varphi: \mathcal{O}_{\mathtt{Y}_{S_{\xi}}}\to\mathcal{F}$ on $\mathtt{Y}_{S_{\xi}}$) with fixed Hilbert homomorphism $P: K^0(\mathcal{X})\to\mathbb{Z}$ (resp.  $P: K^0(\mathcal{Y})\to\mathbb{Z}$). Since $\mathcal{F}$ is an $S_{\xi}$-flat family of pure 1-dimensional sheaves on $\mathtt{X}_{S_{\xi}}$ (resp.  $\mathtt{Y}_{S_{\xi}}$),  
the associated generalized Hilbert polynomial should be a polynomial of degree one, i.e., $P(H^{\otimes m})=am+b$  with $a,b\in\mathbb{Z}$ and $a>0$ where $H=\pi^*\mathcal{O}_{X}(1)$ (resp. $H=\pi^*\mathcal{O}_{Y}(1)$). The boundedness here means that two stacks $\mathfrak{PT}_{\mathfrak{X}/\mathfrak{C}}^{P}$ and $\mathfrak{PT}_{\mathfrak{Y}/\mathfrak{A}}^{P}$ are of finite type. By [\cite{DM}, p. 100] and Theorem \ref{stack1}, we only need to show that these two stacks are quasi-compact.

\begin{proposition}\label{boundedness}
	The stack	$\mathfrak{PT}_{\mathfrak{X}/\mathfrak{C}}^{P}$ $($resp. $\mathfrak{PT}_{\mathfrak{Y}/\mathfrak{A}}^{P}$$)$ is of finite type over $\mathbb{A}^1$ $($resp. over $\mathbb{C}$$)$.
\end{proposition}
\begin{proof}
We follow the argument for boundedness of Quot-stacks in [\cite{Zhou1}, Section 5.1].
The stack $\mathfrak{PT}_{\mathfrak{X}/\mathfrak{C}}^{P}$ has an $\acute{e}$tale covering $\coprod_{k\geq0}[\mathrm{PT}_{\mathcal{X}(k)/\mathbb{A}^{k+1}}^{\mathrm{st},P}\big/R_{\sim,\mathrm{PT}_{\mathcal{X}(k)/\mathbb{A}^{k+1}}}]$, where $\mathrm{PT}_{\mathcal{X}(k)/\mathbb{A}^{k+1}}^{\mathrm{st},P}$	is moduli space of stable orbifold PT pairs with fixed Hilbert homomorphism $P$. For each $k\geq0$, quasi-compactness of $\mathrm{PT}_{\mathcal{X}(k)/\mathbb{A}^{k+1}}^{\mathrm{st},P}$ can be shown by the pushforward property in [\cite{Zhou1}, Corollary 5.3] and quasi-compactness of  $\mathrm{PT}_{\mathcal{X}/\mathbb{A}^{1}}^{(P^\prime)}$ for some modified Hilbert polynomial $P^\prime$ associated with Hilbert homomorphism $P$. The similar argument also holds  for  $\mathfrak{PT}_{\mathfrak{Y}/\mathfrak{A}}^{P}$. Then the proof is completed  since [\cite{Zhou1}, Proposition 5.5] also holds for stable orbifold PT pairs in both  the degeneration and relative cases.
\end{proof}	
\subsubsection{Separatedness}
We apply the valuative criterion in  [\cite{LMB}, Proposition 7.8] to prove the separatedness of Deligne-Mumford stacks 	$\mathfrak{PT}_{\mathfrak{X}/\mathfrak{C}}^{P}$ and  $\mathfrak{PT}_{\mathfrak{Y}/\mathfrak{A}}^{P}$. As in [\cite{Zhou1}, Section 5.2], let $S=\mathrm{Spec}\,R$ and let $\eta=\mathrm{Spec}\,K$ and $\eta_{0}=\mathrm{Spec}\,k$ be the generic point and closed point respectively, where $R$ can be chosen to be a complete discrete valuation ring ($u\in R$ is a uniformizer) with the fractional field $K$ and  an algebraically closed residue field $k$ due to Proposition \ref{boundedness}. And one can simply assume that $S_{\xi}=S$ instead of using the $\acute{e}$tale covering $S_{\xi}:=\coprod S_{i}\to S$ by [\cite{Zhou1}, Remark 5.10]. 
Now, let $(\overline{\xi}, \overline{\varphi})$ and $(\overline{\xi^\prime}, \overline{\varphi^\prime})$ be two objects in $\mathfrak{PT}_{\mathfrak{X}/\mathfrak{C}}^{P}(S)$ $($or $\mathfrak{PT}_{\mathfrak{Y}/\mathfrak{A}}^{P}(S)$$)$  such that the restrictions of their representatives $(\xi, \varphi)$ and $(\xi^\prime, \varphi^\prime)$  to $\eta$ are isomorphic. It remains to show that this isomorphism over $\eta$ can be extended to $S$, i.e.,  $(\overline{\xi}, \overline{\varphi})$ and $(\overline{\xi^\prime}, \overline{\varphi^\prime})$ are isomorphic.
\begin{proposition}\label{separatedness}
	The stack	$\mathfrak{PT}_{\mathfrak{X}/\mathfrak{C}}^{P}$ $($resp. $\mathfrak{PT}_{\mathfrak{Y}/\mathfrak{A}}^{P}$$)$ is separated over $\mathbb{A}^1$ $($resp. over $\mathbb{C}$$)$.
\end{proposition}
\begin{proof}
We follow the similar argument in the proof of [\cite{Zhou1}, Proposition 5.9]
by replacing  quotient sheaves with  $S$-flat families of orbifold PT pairs (notice that pulling back by the contraction map here preserves the purity of one-dimensional sheaf).   The proof is completed by using separatedness of moduli spaces of orbifold PT pairs $\mathrm{PT}_{\mathcal{X}(k)/\mathbb{A}^{k+1}}$   and $\mathrm{PT}_{\mathcal{Y}(k)/\mathbb{A}^{k}}$ together with applying the blowup construction in [\cite{Zhou1}, Proposition 2.11] to both the degeneration and relative cases.
\end{proof}

\subsubsection{Properness}
We apply the definition of properness in [\cite{LMB}, Definition 7.11] for Deligne-Mumford stacks. We begin with the degeneration case.
With the notation $S$, $\eta$, $\eta_{0}$ in Section 5.4.2, for any object $(\overline{\xi}_{\eta}, \overline{\varphi}_{\eta})\in\mathfrak{PT}_{\mathfrak{X}/\mathfrak{C}}^{P}(\eta)$, which is represented by $\xi_{\eta}:\eta\to\mathbb{A}^{k+1}$ and a stable orbifold PT pair $\varphi_{\eta}: \mathcal{O}_{\mathtt{X}_{\eta}}\to\mathcal{F}_{\eta}$ on $\mathtt{X}_{\eta}:=\mathtt{X}_{S,\eta}$, it remains to show that after a finite base change of $S$, one can extend $(\overline{\xi}_{\eta}, \overline{\varphi}_{\eta})$ to $\eta_{0}$ by the valuative criterion of properness (see also [\cite{LMB}, Theorem 7.3] for the valuative criterion for proving a separated Deligne-Mumford stack of finite type to be universally closed).

To prove the properness of $\mathfrak{PT}_{\mathfrak{X}/\mathfrak{C}}^{P}$ (resp. $\mathfrak{PT}_{\mathfrak{Y}/\mathfrak{A}}^{P}$), we adopt the strategy described in [\cite{LW}, Section 5.4] for the variety case and  carry it out with the technique and argument in [\cite{Zhou1}, Section 5.4] which are  generalizing  the  1-dimensional case of those developed in [\cite{LW}, Section 5.1 and 5.2] to the stacky case. More precisely, in the degeneration case, in order to obtain $\widetilde{\varphi}:\mathcal{O}_{\mathtt{X}_{\widetilde{S}}}\to\widetilde{F}$ such that $\widetilde{\varphi}\in\mathfrak{PT}_{\mathfrak{X}/\mathfrak{C}}^{P}(\widetilde{S})$ extends $\varphi_{\eta}$, the strategy is taking several possible modifications of an initial (trivial) extension of $\varphi_{\eta}$ to obtain first the admissibility of  $\mathrm{coker}\,\widetilde{\varphi}$ and then of $\widetilde{\mathcal{F}}$ accompanied with the finiteness of $\mathrm{Aut}(\widetilde{\varphi})$     along the fiber over the closed point of $\widetilde{S}$ via a totally finite base change $\widetilde{S}\to S$.

For the degeneration case, we need the following lemmas.
\begin{lemma}\label{admissble1}
	Let  the family $\Delta\times S\to S$ be a component of the family of expanded degenerations $\mathtt{X}_{S}$ over $(S, \xi)$ where  $\xi: S\to\mathbb{A}^2$ maps $S$ constantly to $0\in\mathbb{A}^2$. Suppose $\varphi$ is an $S$-flat family of  orbifold PT pairs on $\Delta\times S$ such that $\varphi|_{\eta}$ is stable but $\mathrm{coker}(\varphi|_{\eta_{0}})$ is not admissible. Then there exists another $S$-flat family of orbifold PT pairs $\widetilde{\varphi}$ such that $\widetilde{\varphi}|_{\eta}$ is related to $\varphi|_{\eta}$ through the $\mathbb{C}^*$-action on $\mathcal{X}(1)\to\mathbb{A}^2$ and $\mathrm{coker}(\widetilde{\varphi}|_{\eta_{0}})$ is normal to $\mathcal{D}_{\pm}\subset\Delta$.
\end{lemma}
\begin{proof}
Since $\mathrm{coker}(\varphi|_{\eta_{0}})$ is not admissible, $\mathrm{coker}(\varphi|_{\eta_{0}})$ is not normal to $\mathcal{D}_{-}$, or $\mathcal{D}_{+}$, or both of them. We will deal with the case where $\mathrm{coker}(\varphi|_{\eta_{0}})$ is not normal to $\mathcal{D}_{-}$, the remaining cases are the same. As normality has the local property in the $\acute{e}$tale topology, we consider an affine $\acute{e}$tale neighborhood of  a point in $\mathcal{D}_{-}\times \eta_{0}$. As in the proof of  [\cite{Zhou1}, Lemma 5.16], the local model of $\Delta\times S$ can be chosen to be $U=\mathrm{Spec}\,A:=\mathrm{Spec}\,R[y,\mathbf{z}]$ where $(\mathcal{D}_{-}\times\eta_{0})|_{U}$ is defined by the ideal $(y,u)$, the vector $\mathbf{z}$ represents the coordinates in $\mathcal{D}$ (containing $\mathcal{D}_{-}$) and $y=0$ is the local equation defining $\mathcal{D}$. Suppose the coherent sheaf $\mathrm{coker}\,\varphi$ on $U$ is represented by $A$-module $M$. Since $\varphi:\mathcal{O}_{\Delta\times S}\to\mathcal{F}$ is an $S$-flat family of orbifold PT pairs and the map $(\mathrm{Im}\,\varphi)|_{s}\to\mathcal{F}|_{s}$ restricted on $U$ is injective for each $s\in S$, then the coherent sheaf $\mathrm{coker}\,\varphi$ on $U$  is flat over $S$ by [\cite{Mat}, (20.E)].
That is, the $A$-module $M$ is flat over $R$. 
Since $\dim\mathrm{coker}\,\varphi=0$, 
by [\cite{Eis}, Theorem 2.14 and Corollary 2.17] or [\cite{Mat}, (12.B)], the $A$-module $M$ is of finite length. 
Then there exists an integer $l>0$ such that we have the surjection $A^{\oplus l}\to M\to 0$. Denote by $K$ the kernel of this surjection, and we have the following short exact sequence
\ben
0\to K\to A^{\oplus l}\to M\to 0.
\een
Now, the proof is completed by applying the similar argument in [\cite{Zhou1}, Lemma 5.19] where we obtain the normality of  $\mathrm{coker}\,\widetilde{\varphi}$ (viewed as some quotient sheaf due to the above short exact sequence) to $\mathcal{D}_{-}\times \eta_{0}$  via using the properness of moduli space $\mathrm{PT}_{\Delta/\mathbb{C}}$ (see Section 5.1 for the definition) to  extend $\widetilde{\varphi}|_{\eta}$ (it is still an orbifold PT pair) over $\eta_{0}$ and taking the flat limit of $\mathrm{coker}(\widetilde{\varphi}|_{\eta_{0}})$, where $\widetilde{\varphi}|_{\eta}$ is related to $\varphi|_{\eta}$ through some $\mathbb{C}^*$-action on $\Delta$ (hence on $\mathcal{X}(1)\to\mathbb{A}^2$). 
\end{proof}

\begin{remark}
With the hypothesis in Lemma \ref{admissble1}, one may  follow the similar argument in the proof of [\cite{Li1}, Lemma 3.9] to find a new family $\widetilde{\varphi}$ (with some $\mathbb{C}^*$-action on $\varphi$) such that $\mathrm{Supp}(\mathrm{coker}(\widetilde{\varphi}|_{\widetilde{\eta}_{0}}))$  is away from the divisors $\mathcal{D}_{\pm}$ (it is equivalent to the admissibility of $\mathrm{coker}(\widetilde{\varphi}|_{\widetilde{\eta}_{0}})$ by Remark \ref{adm-cri}) possibly after a finite base change $\widetilde{S}\to S$ where $\widetilde{\eta}_{0}$ is the closed point of $\widetilde{S}$.
\end{remark}

	
\begin{lemma}\label{admissible3}
	Let  the family $\Delta\times S\to S$ be a component of the family of expanded degenerations $\mathtt{X}_{S}$ over $(S, \xi)$ where  $\xi: S\to\mathbb{A}^2$ maps $S$ constantly to $0\in\mathbb{A}^2$. Suppose $\varphi$ is an $S$-flat family of  orbifold PT pairs on $\Delta\times S$ such that $\varphi|_{\eta}$ is stable. Then there exists an $S$-flat family of orbifold PT pairs $\widetilde{\varphi}$ such that $\widetilde{\varphi}|_{\eta}$ is related to $\varphi|_{\eta}$ through the $\mathbb{C}^*$-action on $\mathcal{X}(1)\to\mathbb{A}^2$ and $\mathrm{Aut}(\widetilde{\varphi}|_{\eta_{0}})$ is finite.
\end{lemma}
\begin{proof}
By Lemma \ref{admissble1},	 there exists an $S$-flat family of orbifold PT pairs $\widehat{\varphi}$ such that $\widehat{\varphi}|_{\eta}$ is related to $\varphi|_{\eta}$ through the $\mathbb{C}^*$-action on $\mathcal{X}(1)\to\mathbb{A}^2$ and $\mathrm{coker}(\widehat{\varphi}|_{\eta_{0}})$ is normal to $\mathcal{D}_{\pm}$.
If $\mathrm{Supp}(\mathrm{coker}(\widehat{\varphi}|_{\eta_{0}}))\cap(\Delta\times\eta_{0})\neq\emptyset$, then $\mathrm{Aut}(\widehat{\varphi}|_{\eta_{0}})$ is finite and we are done by setting $\widetilde{\varphi}:=\widehat{\varphi}$.
Otherwise, $\mathrm{Supp}(\mathrm{coker}(\widehat{\varphi}|_{\eta_{0}}))\cap(\Delta\times\eta_{0})=\emptyset$,  then $\mathrm{coker}(\widehat{\varphi}|_{\eta_{0}})=0$. Assume that $\mathrm{Aut}(\widehat{\varphi}|_{\eta_{0}})$ is not finite, then $\widehat{\varphi}|_{\eta_{0}}$ is $\mathbb{C}^*$-equivariant.   Combined with the fact that $\acute{e}$tale locally $\mathrm{coker}\,\widehat{\varphi}$ is flat over $S$ as shown in Lemma \ref{admissble1} and $\mathrm{coker}(\widehat{\varphi}|_{\eta})$ is admissible due to $\mathbb{C}^*$-action, $\mathrm{coker}(\widehat{\varphi}|_{\eta})$ is also zero $\acute{e}$tale locally by using Lemma \ref{global-hilbpoly}.  
Then $\widehat{\varphi}$ is also an $S$-flat quotient sheaf on $\Delta\times S$. Now, the proof is completed by the similar argument in the proof of [\cite{Zhou1}, Lemma 5.20]  where we use the properness of moduli space $\mathrm{PT}_{\Delta/\mathbb{C}}$ (or $\mathrm{PT}_{\Delta/\mathbb{C}}^P$).
\end{proof}

\begin{lemma}\label{admissible4}
Let $\varphi:\mathcal{O}_{\mathtt{X}_{S}}\to\mathcal{F}$ be an $S$-flat family of orbifold PT pairs on $\mathtt{X}_{S}=\xi^*\mathcal{X}(k)$ such that $\mathtt{X}_{S}|_{\eta}$ is smooth over $\eta$
where $\xi:S\to\mathbb{A}^{k+1}$. Assume $\mathrm{coker}(\varphi|_{\eta_{0}})$ is admissible. Then there are a finite base change $\widetilde{S}\to S$ and $(\widetilde{\xi}, \widetilde{\varphi})$, where $\widetilde{\xi}: \widetilde{S}\to\mathbb{A}^{\widetilde{k}+1}$ and  $\widetilde{\varphi}$ is an $\widetilde{S}$-flat family of orbifold PT pairs   on $\mathtt{X}_{\widetilde{S}}=\widetilde{\xi}^*\mathcal{X}(\widetilde{k})$, such that
$(\widetilde{\xi},\widetilde{\varphi})\cong(\xi, \varphi)\times_{\eta}\widetilde{\eta}$, the $\widetilde{S}$-flat family $\widetilde{\varphi}$ is admissible and $[\mathrm{Aut}(\widetilde{\varphi}|_{\widetilde{\eta}_{0}}): \mathrm{Aut}(\varphi|_{\eta_{0}})]$ is finite, where $\widetilde{\eta}$ and $\widetilde{\eta_{0}}$ are the generic and closed points of $\widetilde{S}$.
\end{lemma}

\begin{proof}
Without loss of generality, we assume that $\mathtt{X}_{S}|_{\eta_{0}}=\mathcal{X}_{0}[k]$. The smoothness of  $\mathtt{X}_{S}|_{\eta}$ over $\eta$   implies that  $\varphi|_{\eta}$ is admissible and stable by the convention before Definiton \ref{fam-PT}.
If $\mathcal{F}|_{\eta_{0}}$ is admissible, then we are done. Otherwise, we suppose that $\mathcal{F}|_{\eta_{0}}$ is not normal to some $\mathcal{D}_{l}$ ($0\leq l\leq k$). Since $\mathrm{coker}(\varphi|_{\eta_{0}})$ is admissible, then $\mathrm{Supp}(\mathrm{coker}(\varphi|_{\eta_{0}}))$ is away from all the singular divisors of $\mathcal{X}_{0}[k]$. Hence, $\acute{e}$tale locally near any singular divisor of $\mathcal{X}_{0}[k]$, $\varphi$ is also a quotient. Now,   the proof is completed by following the similar argument in  [\cite{Zhou1}, Section 5.4, Step 2]  for the degeneration case to reduce all $\mathrm{Err}_{l}(\mathcal{F}|_{\eta_{0}})$ which are not zero and hence $\mathrm{Err}(\mathcal{F}|_{\eta_{0}})$ to zero by modifying the family $\varphi$ where we use the properness of $\mathrm{PT}_{\mathcal{X}(k+1)/\mathbb{A}^{k+2}}^P$ instead of the Quot-space for obtaining the extension (see also the original argument in [\cite{LW}, Section 5.1]). Here, the admissibility of the cokernel on the corresponding closed fiber is preserved after the finite base change $\widetilde{S}\to S$.
\end{proof}

For the properness of $\mathfrak{PT}_{\mathfrak{X}/\mathfrak{C}}^{P}$, we still need the following result in the relative case.
\begin{lemma}\label{admissible5}
	Let $\varphi:\mathcal{O}_{\mathtt{Y}_{S}}\to\mathcal{F}$ be an $S$-flat family of orbifold PT pairs on $\mathtt{Y}_{S}=\xi^*\mathcal{Y}(k)$ such that $\varphi|_{\eta}$ is admissible and $\mathrm{coker}(\varphi|_{\eta_{0}})$ is admissible where $\xi:S\to\mathbb{A}^k$. Then there exists $(\widetilde{\xi}, \widetilde{\varphi})$ where $\widetilde{\xi}:S\to\mathbb{A}^{\widetilde{k}}$ and $\widetilde{\varphi}:\mathcal{O}_{\widetilde{\mathtt{Y}}_{S}}\to\widetilde{\mathcal{F}}$ is an $S$-flat family of orbifold PT pairs on $\widetilde{\mathtt{Y}}_{S}=\widetilde{\xi}^*\mathcal{Y}(\widetilde{k})$, such that $(\widetilde{\xi}, \widetilde{\varphi})|_{\eta}\cong(\xi, \varphi)|_{\eta}$, and    sheaves $\mathrm{coker}\,(\widetilde{\varphi}|_{\eta_{0}})$ and  $\widetilde{\mathcal{F}}|_{\eta_{0}}$ are normal to the distinguished divisor $\widetilde{\mathcal{D}}_{S}\subset\widetilde{\mathtt{Y}}_{S}$.
\end{lemma}
\begin{proof}
Assume $\mathtt{Y}_{S}|_{\eta_{0}}=\mathcal{Y}[k]$ with its distinguished divisor $\mathcal{D}[k]$. It follows from the hypothesis that $\varphi$ is also a quotient on an affine $\acute{e}$tale neighborhood of $\mathcal{D}[k]\times\eta_{0}$. The proof follows exactly from the similar argument in [\cite{Zhou1}, Lemma 5.16]  where we use the properness of $\mathrm{PT}_{\mathcal{Y}(k+1)/\mathbb{A}^{k+1}}$ for extensions since the modified family induced by the successive blowups preserves the admissibility of the cokernel and the purity of one-dimensional sheaves.	
\end{proof}
\begin{remark}\label{finite-aut}
In the proof of Lemma \ref{admissible5}, suppose $\mathcal{F}|_{\eta_{0}}$ is not normal to $\mathcal{D}[k]$ at the begining, then the modified family induced by some blowups will not  produce any new infinitely autoequivalences.
\end{remark}
Similarly  we have the following results for the relative case.
\begin{lemma}\label{admissible6}
	Let  the family $\Delta\times S\to S$ be a component of the family of expanded degenerations $\mathtt{Y}_{S}$ over $(S, \xi)$ where  $\xi: S\to\mathbb{A}^1$ maps $S$ constantly to $0\in\mathbb{A}^1$. Suppose $\varphi$ is an $S$-flat family of  orbifold PT pairs on $\Delta\times S$ such that $\varphi|_{\eta}$ is stable but $\mathrm{coker}(\varphi|_{\eta_{0}})$ is not admissible. Then there exists another $S$-flat family of orbifold PT pairs $\widetilde{\varphi}$ such that $\widetilde{\varphi}|_{\eta}$ is related to $\varphi|_{\eta}$ through the $\mathbb{C}^*$-action on $\mathcal{Y}(1)\to\mathbb{A}^1$ and $\mathrm{coker}(\widetilde{\varphi}|_{\eta_{0}})$ is normal to $\mathcal{D}_{\pm}$.
\end{lemma}


\begin{lemma}\label{admissible8}
	Let  the family $\Delta\times S\to S$ be a component of the family of expanded degenerations $\mathtt{Y}_{S}$ over $(S, \xi)$ where  $\xi: S\to\mathbb{A}^1$ maps $S$ constantly to $0\in\mathbb{A}^1$. Suppose $\varphi$ is a $S$-flat family of  orbifold PT pairs on $\Delta\times S$ such that $\varphi|_{\eta}$ is stable. Then there exists another $S$-flat family of orbifold PT pairs $\widetilde{\varphi}$ such that $\widetilde{\varphi}|_{\eta}$ is related to $\varphi|_{\eta}$ through the $\mathbb{C}^*$-action on $\mathcal{Y}(1)\to\mathbb{A}^1$ and $\mathrm{Aut}(\widetilde{\varphi}|_{\eta_{0}})$ is finite.
\end{lemma}

\begin{lemma}\label{admissible9}
	Let $\varphi:\mathcal{O}_{\mathtt{Y}_{S}}\to\mathcal{F}$ be an $S$-flat family of orbifold PT pairs on $\mathtt{Y}_{S}=\xi^*\mathcal{Y}(k)$ such that $\mathtt{Y}_{S}|_{\eta}$ is smooth over $\eta$ and $\varphi|_{\eta}$ is admissible
	where $\xi:S\to\mathbb{A}^k$. Assume that $\mathrm{coker}(\varphi|_{\eta_{0}})$ is admissible.  Then there are a finite base change $\widetilde{S}\to S$ and $(\widetilde{\xi}, \widetilde{\varphi})$, where $\widetilde{\xi}: \widetilde{S}\to\mathbb{A}^{\widetilde{k}}$ and $\widetilde{\varphi}$ is an $\widetilde{S}$-flat family of orbifold PT pairs on $\mathtt{Y}_{\widetilde{S}}=\widetilde{\xi}^*\mathcal{Y}(\widetilde{k})$, such that
	$(\widetilde{\xi},\widetilde{\varphi})\cong(\xi, \varphi)\times_{\eta}\widetilde{\eta}$, the $\widetilde{S}$-flat family $\widetilde{\varphi}$ is admissible and $[\mathrm{Aut}(\widetilde{\varphi}|_{\widetilde{\eta}_{0}}): \mathrm{Aut}(\varphi|_{\eta_{0}})]$ is finite, where $\widetilde{\eta}$ and $\widetilde{\eta_{0}}$ are the generic and closed points of $\widetilde{S}$.
\end{lemma}

Now,  we have
\begin{proposition}\label{properness}
	The stack	$\mathfrak{PT}_{\mathfrak{X}/\mathfrak{C}}^{P}$  $($resp. $\mathfrak{PT}_{\mathfrak{Y}/\mathfrak{A}}^{P}$$)$ is proper over $\mathbb{A}^1$  $($resp. over $\mathbb{C}$$)$.
\end{proposition}
\begin{proof}
As in the discussion before   [\cite{Zhou1}, Section 5.4, Step 1], since one can embed $\mathbb{A}^{k+1}$ into a large base $\mathbb{A}^{\widetilde{k}+1}$ ($\widetilde{k}>k$) and use the smooth equivalence relation $R_{\sim,\mathbb{A}^{\widetilde{k}+1}}$ for choosing another representative of an object in 	$\mathfrak{PT}_{\mathfrak{X}/\mathfrak{C}}^{P}(\eta)$, we only need to consider $\xi_{\eta}$ in the following form 
\ben
\xi_{\eta}: \eta\to\mathbb{A}^{k+1},\;\;\;\mathbb{C}[t_{0},\cdots,t_{k}]\to K,\;\;\;(t_{0},\cdots,t_{k})\to(0,\cdots,0).
\een	
Then we have a trivial extension $\xi: S\to\mathbb{A}^{k+1}$ mapping $S$ constantly to $0\in\mathbb{A}^{k+1}$ which implies that $\mathtt{X}_{S}\cong \mathcal{X}_{0}[k]\times S=(\mathcal{Y}_{-}\times S)\cup_{\mathcal{D}\times S}(\Delta_{1}\times S)\cup_{\mathcal{D}\times S}\cdots\cup_{\mathcal{D}\times S}(\Delta_{k}\times S)\cup_{\mathcal{D}\times S}(\mathcal{Y}_{+}\times S)$.
By the properness of moduli space $\mathrm{PT}_{\mathcal{X}_{0}[k]/\mathbb{C}}$, the stable orbifold PT pair $\varphi_{\eta}$ extends to an $S$-flat family of orbifold PT pairs $\varphi$ on $\mathtt{X}_{S}$. If $\varphi|_{\eta_{0}}$ is stable, we are done. Otherwise, by using Lemma \ref{admissble1} and Lemma \ref{admissible3} (if $\varphi$ restricted on $\Delta_{i}\times  S$ is an $S$-flat family of orbifold PT pairs which is not a quotient) or using [\cite{Zhou1}, Lemma 5.20] (if $\varphi$ restricted on $\Delta_{i}\times  S$ is a quotient) for each $\Delta_{i}\times S$ together with applying the smooth equivalence relation $R_{\sim,\mathrm{PT}_{\mathcal{X}(k)/\mathbb{A}^{k+1}}}$, one can assume that for $1\leq i\leq k$,  the restriction of $\mathrm{coker}(\varphi|_{\eta_{0}})$ on  every smooth pair $(\Delta_{i}, \mathcal{D}_{i-1}\cup\mathcal{D}_{i})$  is admissible or zero and $\mathrm{Aut}(\varphi|_{\eta_{0}})$ is finite.

Next, if $\mathrm{coker}(\varphi|_{\eta_{0}})$ is also admissible on 
two smooth pairs $(\mathcal{Y}_{-}, \mathcal{D}_{-})$ or $(\mathcal{Y}_{+}, \mathcal{D}_{+})$, by Lemma \ref{admissible5}, Remark \ref{finite-aut}
and Lemma \ref{admissible9}, after a finite base change $\widetilde{S}_{\pm}\to S$, we can find an $\widetilde{S}_{\pm}$-flat family of stable orbifold PT pairs $\widetilde{\varphi}_{\pm}$
on $(\mathtt{Y}_{\pm})_{\widetilde{S}_{\pm}}=\widetilde{\xi}_{\pm}^*\mathcal{Y}_{\pm}(\widetilde{k}_{\pm})$ extending $\varphi|_{\mathcal{Y}_{\pm}\times\eta}\times_{\eta}\widetilde{\eta}_{\pm}$ where $\widetilde{\xi}_{\pm}: \widetilde{S}_{\pm}\to\mathbb{A}^{\widetilde{k}_{\pm}}$
and $\widetilde{\eta}_{\pm}$ is a generic point of $\widetilde{S}_{\pm}$. Otherwise,  we suppose that $\mathrm{coker}(\varphi|_{\eta_{0}})$ is not admissible on  $(\mathcal{Y}_{-}, \mathcal{D}_{-})$ and $(\mathcal{Y}_{+}, \mathcal{D}_{+})$. Viewing $\mathrm{coker}\,\varphi$ as a quotient sheaf on the local model of $\mathcal{Y}_{-}\times S$ as in Lemma \ref{admissble1}, one can obtain a modified  $S$-flat family $\varphi^\prime$ on $(\mathtt{Y}_{-})_{S}^\prime=\xi^{\prime*}\mathcal{Y}_{-}(k^\prime)$
as in [\cite{Zhou1}, Lemma 5.16] such that $(\xi^\prime,\varphi^\prime)|_{\eta}\cong(\xi,\varphi)|_{\mathcal{Y}_{-}\times\eta}$ and  $\mathrm{coker}(\varphi^\prime|_{\eta_{0}})$ is normal to the distinguished divisor $\mathcal{D}[k^\prime]\times\eta_{0}$ where $\xi^\prime: S\to\mathbb{A}^{k^\prime}$. As in Remark \ref{finite-aut}, the autoequivalence group is also finite.  With this result, furtherly applying the similar argument in the proof of [\cite{Zhou1}, Section 5.4, Step 2] (as in the proof of Lemma \ref{admissible4}) to deal with the admissibility of $\mathrm{coker}\,\varphi^\prime$ (reducing $\mathrm{Err}(\mathrm{coker}\,\varphi^\prime)$ to  zero), one can find a finite base change $\widehat{S}\to S$ and $(\widehat{\xi},\widehat{\varphi})$ such that $(\widehat{\xi},\widehat{\varphi})|_{\widehat{\eta}}\cong(\xi^\prime,\varphi^\prime)|_{\eta}\times_{\eta}\widehat{\eta}$, the sheaf $\mathrm{coker}(\widehat{\varphi}|_{\widehat{\eta}_{0}})$ is admissible and $[\mathrm{Aut}(\widehat{\varphi}|_{\widehat{\eta}_{0}}):\mathrm{Aut}(\varphi^\prime|_{\eta_{0}})]$ is finite where $\widehat{\xi}:\widehat{S}\to \mathbb{A}^{\widehat{k}}$,  the modified family $\widehat{\varphi}$ is an $\widehat{S}$-flat family of  orbifold PT pairs on $(\mathtt{Y}_{-})_{\widehat{S}}$, and $\widehat{\eta}$ and $\widehat{\eta}_{0}$ are the generic and closed points of $\widehat{S}$. Now, by Lemma \ref{admissible5}, Remark \ref{finite-aut}
and Lemma \ref{admissible9}, after a finite base change $\widetilde{S}_{-}\to\widehat{S}$ (or finally $\widetilde{S}_{-}\to S$), there exists an  $\widetilde{S}_{-}$-flat family of stable orbifold PT pairs $\widetilde{\varphi}_{-}$ on $(\mathtt{Y}_{-})_{\widetilde{S}_{-}}=\widetilde{\xi}_{-}^*\mathcal{Y}_{-}(\widetilde{k}_{-})$ extending $\widehat{\varphi}|_{\widehat{\eta}}\times_{\widehat{\eta}}\widetilde{\eta}_{-}\cong\varphi^\prime|_{\eta}\times_{\eta}\widetilde{\eta}_{-}\cong\varphi|_{\mathcal{Y}_{-}\times\eta}\times_{\eta}\widetilde{\eta}_{-}$ where $\widetilde{\xi}_{-}: \widetilde{S}_{-}\to\mathbb{A}^{\widetilde{k}_{-}}$ and $\widetilde{\eta}_{-}$ is a generic point of $\widetilde{S}_{-}$. Similarly, after a finite base change $\widetilde{S}_{+}\to S$, we have 
an  $\widetilde{S}_{+}$-flat family of stable orbifold PT pairs $\widetilde{\varphi}_{+}$
on $(\mathtt{Y}_{+})_{\widetilde{S}_{+}}=\widetilde{\xi}_{+}^*\mathcal{Y}_{+}(\widetilde{k}_{+})$ extending $\varphi|_{\mathcal{Y}_{+}\times\eta}\times_{\eta}\widetilde{\eta}_{+}$ where $\widetilde{\xi}_{+}: \widetilde{S}_{+}\to\mathbb{A}^{\widetilde{k}_{+}}$
and $\widetilde{\eta}_{+}$ is a generic point of $\widetilde{S}_{+}$.

 For each remaining smooth pair $(\Delta_{i}, \mathcal{D}_{i-1}\cup\mathcal{D}_{i})$  ($1\leq i\leq k$), again by Lemma \ref{admissible5}, Remark \ref{finite-aut} and Lemma \ref{admissible9}, after a finite base change $\widetilde{S}_{i}\to S$, there exists an $\widetilde{S}_{i}$-flat family of stable orbifold PT pairs $\widetilde{\varphi}_{i}$ on $(\mathtt{Y}_{i})_{\widetilde{S}_{i}}$ extending $\varphi|_{\Delta_{i}\times\eta}\times_{\eta}\widetilde{\eta}_{i}$ where $(\mathtt{Y}_{i})_{\widetilde{S}_{i}}$ is the family of expanded pairs corresponding to $(\Delta_{i}, \mathcal{D}_{i-1}\cup\mathcal{D}_{i})$ and $\widetilde{\eta}_{i}$ is a generic point of $\widetilde{S}_{i}$.

Denote by $\widetilde{S}$ the fiber product of $\widetilde{S}_{\pm}$ and $\widetilde{S}_{i}$ ($1\leq i\leq k$). For a common finite base change $\widetilde{S}\to S$, one can glue $\widetilde{\varphi}_{-}$, $\widetilde{\varphi}_{1},\cdots,\widetilde{\varphi}_{k}$, $\widetilde{\varphi}_{+}$ along their common divisors by the admissibility and separatedness of moduli spaces of stable pairs and moduli spaces of stable quotients for the unique flat limit as in [\cite{LW}, Section 5.2]. Then the obtained family $\widetilde{\varphi}$ is an  $\widetilde{S}$-flat family of stable orbifold PT pairs extending $\varphi|_{\eta}\times_{\eta}\widetilde{\eta}$, i.e., $\widetilde{\varphi}\in\mathfrak{PT}_{\mathfrak{X}/\mathfrak{C}}^{P}(\widetilde{S})$. This completes the proof of  properness of the stack	$\mathfrak{PT}_{\mathfrak{X}/\mathfrak{C}}^{P}$.

The properness of the stack $\mathfrak{PT}_{\mathfrak{Y}/\mathfrak{A}}^{P}$ can be proved in a similar way with Remark \ref{finite-aut} and Lemmas \ref{admissible5}, \ref{admissible6}, \ref{admissible8} and \ref{admissible9}.
\end{proof}

\section{Degeneration  formula}
In this section, we present degeneration formulas for the orbifold Pandharipande-Thomas theory in both cycle and numerical versions following [\cite{LW,MPT,Zhou1}]. 
\subsection{Decomposition of central fibers} 
We will first recall some results on the decomposition of central fibers in [\cite{Zhou1}, Section 7.1] which is a analogue of those in [\cite{LW}, Section 2.4, 2.5 and 5.7], and then give a description of decomposition of the central fiber  of  $\mathfrak{PT}_{\mathfrak{X}/\mathfrak{C}}^{P}$ over $\mathbb{A}^1$.

Let $\pi:\mathcal{X}\to\mathbb{A}^1$ be a locally simple degeneration. The (continuous) weight assignment on (a family of) expanded degenerations is defined as follows.
\begin{definition}(see [\cite{LW}, Definitions 2.14 and  2.15] and [\cite{Zhou1}, Definition 7.2])\label{weight-assignment}
Let $\Lambda:=\mathrm{Hom}(K^0(\mathcal{X}),\mathbb{Z})$. A weight assignment on $\mathcal{X}_{0}[k]$ is a function
\ben
\omega:  \{\Delta_{0},\cdots,\Delta_{k+1},\mathcal{D}_{0},\cdots,\mathcal{D}_{k}\}\to\Lambda
\een
such that $\omega(\Delta_{i})\neq0$ for every $0\leq i\leq k+1$ where $\Delta_{0}=\mathcal{Y}_{-}$ and $\Delta_{k+1}=\mathcal{Y}_{+}$. For any $0\leq a\leq b\leq k+1$, we define the weight of the subchain $\Delta_{[a,b]}:=\Delta_{a}\cup\cdots\cup\Delta_{b}$ to be 
\ben
\omega(\Delta_{[a,b]}):=\sum_{i=a}^{b}\omega(\Delta_{i})-\sum_{i=a}^{b-1}\omega(\mathcal{D}_{i}).
\een
The total weight of $\omega$ is $\omega(\mathcal{X}_{0}[k])$.
For any family of expanded degenerations $\pi:\mathtt{X}_{S}\to S$, a continuous weight assignment on $\mathcal{X}(k)$ is an assignment of a weight function on each fiber that is continuous over $S$. More precisely, if for any curve $C\subset S$ with $s_{0},s\in C$, the general fiber $\mathtt{X}_{S,s}\cong\mathcal{X}_{0}[m]$ specializes to the fiber $\mathtt{X}_{S,s_{0}}\cong\mathcal{X}_{0}[n]$ $(m\leq n)$
where $\Delta_{i}\subset\mathtt{X}_{S,s}$ specializes to $\Delta_{[a_{i},b_{i}]}\subset\mathtt{X}_{S,s_{0}}$ with $0\leq a_{i}\leq b_{i}\leq n+1$, then we have $\omega_{s}(\Delta_{i})=\omega_{s_{0}}(\Delta_{[a_{i},b_{i}]})$ for $0\leq i\leq m+1$.
\end{definition}

We recall some stacks with  weight assignments  in [\cite{Zhou1}, Section 7.1] (see also [\cite{LW}, Section 2.4 and 2.5]) as follows.
Assume $H_{i}:=\{(t_{0},\cdots,t_{k})\in\mathbb{A}^{k+1}|t_{i}=0\}$ is the hyperplane in $\mathbb{A}^{k+1}$, we define 
\ben
&&\mathfrak{X}_{0}^\dag:=\lim\limits_{\longrightarrow}\bigg[\coprod_{i=0}^k\mathcal{X}(k)|_{H_{i}}\bigg/R_{\sim,\coprod_{i=0}^{k} \mathcal{X}(k)|_{H_{i}}}\bigg],\\
&&\mathfrak{C}_{0}^\dag:=\lim\limits_{\longrightarrow}\bigg[\coprod_{i=0}^kH_{i}\bigg/R_{\sim,\coprod_{i=0}^k H_{i}}\bigg],
\een
where $R_{\sim,\coprod_{i=0}^k \mathcal{X}(k)|_{H_{i}}}$ and $R_{\sim,\coprod_{i=0}^k H_{i}}$ are the equivalence relations induced by restricting relations $R_{\sim,\mathcal{X}(k)}$ and $R_{\sim,\mathbb{A}^{k+1}}$ to $\coprod_{i=0}^k \mathcal{X}(k)|_{H_{i}}$ and $\coprod_{i=0}^k H_{i}$ respectively, and the inductive systems are compatible with the restrictions. Fix a $P\in\mathrm{Hom}(K^0(\mathcal{X}),\mathbb{Z})$, assume $\mathcal{X}(k)^P$  is the disjoint union of all copies of  $\mathcal{X}(k)$ which are indexed by all possible continuous weight assignments with the same total weight $P$. Let $(\mathbb{A}^{k+1})^P$ be defined similarly as  $\mathcal{X}(k)^P$ with the corresponding weight assignments induced from $\pi_{k}:\mathcal{X}(k)\to\mathbb{A}^{k+1}$. Denote by $R_{\sim,\mathcal{X}(k)}^P$ the equivalence relation generated by relations $R_{\sim,\mathcal{X}(k)}$ on every copy of $\mathcal{X}(k)$ and relations induced by identifying different copies of $\mathcal{X}(k)$ respecting weight assignments. One can define $R_{\sim,\mathbb{A}^{k+1}}^P$ similarly. Define the following stacks
\ben
\mathfrak{X}^P:=\lim\limits_{\longrightarrow}\bigg[\mathcal{X}(k)^P\bigg/R_{\sim,\mathcal{X}(k)}^P\bigg],\,\,\,\,\,\,\mathfrak{C}^P:=\lim\limits_{\longrightarrow}\bigg[(\mathbb{A}^{k+1})^P\bigg/R_{\sim,\mathbb{A}^{k+1}}^P\bigg].
\een
One can  define $\mathfrak{X}_{0}^{\dagger,P}$ and $\mathfrak{C}_{0}^{\dagger,P}$ in the same manner  with respect to $\mathfrak{X}_{0}^{\dagger}$ and $\mathfrak{C}_{0}^{\dagger}$.
Define 
\ben
\Lambda_{P}^{\mathrm{spl}}:=\{\varrho=(\varrho_{-},\varrho_{+},\varrho_{0})|\varrho_{\pm},\varrho_{0}\in\Lambda, \varrho_{-}+\varrho_{+}-\varrho_{0}=P\}.
\een
For $\varrho\in\Lambda_{P}^\mathrm{spl}$, define $\mathfrak{X}_{0}^{\dagger,\varrho}\subset\mathfrak{X}_{0}^{\dagger,P}$ to be the open and closed substacks parameterizing those families with Hilbert homomorphisms $\varrho_{-}$, $\varrho_{0}$, $\varrho_{+}$ as weight assignments on the fiber decomposition $\mathtt{Y}_{S,-}\cup_{\mathcal{D}_{S}}\mathtt{Y}_{S,+}$ from Proposition \ref{decom}. One can similarly define $\mathfrak{C}_{0}^{\dagger,\varrho}\subset\mathfrak{C}_{0}^{\dagger,P}$.
 Similarly, for a given smooth pair $(\mathcal{Y},\mathcal{D})$, one can define the (continuous) weight assignment on (a family of) expanded pairs as in Definition \ref{weight-assignment} with $\Lambda$ chosen to be $\mathrm{Hom}(K^0(\mathcal{Y}),\mathbb{Z})$. For a fixed $P\in\mathrm{Hom}(K^0(\mathcal{Y}),\mathbb{Z})$,  one can define $\mathfrak{Y}^P$ and $\mathfrak{A}^P$  as in the degeneration case.
Suppose $P, P^\prime\in\mathrm{Hom}(K^0(\mathcal{Y}),\mathbb{Z})$, define $\mathfrak{Y}^{P,P^\prime}\subset\mathfrak{Y}^P$ to be the substacks parameterizing those families with the weight assignment $P$ on each fiber and the associated $P^\prime$ on its  distinguished divisor. The substack $\mathfrak{A}^{P,P^\prime}\subset\mathfrak{A}^P$ can be defined similarly.
Due to the central fiber decomposition in Proposition \ref{decom}, we have 
\begin{proposition}(see [\cite{LW}, Proposition 2.20] and [\cite{Zhou1}, Proposition 7.4])\label{split-iso} 
	For each $\varrho\in\Lambda_{P}^\mathrm{spl}$, we have isomorphisms in  the following commutative diagram
\ben
\xymatrixcolsep{3pc}\xymatrix{
(\mathfrak{Y}_{-}^{\varrho_{-},\varrho_{0}}\times\mathfrak{A}^{\varrho_{+},\varrho_{0}})\cup_{\mathfrak{A}^{\varrho_{-},\varrho_{0}}\times\mathcal{D}\times\mathfrak{A}^{\varrho_{+},\varrho_{0}}}(\mathfrak{A}^{\varrho_{-},\varrho_{0}}\times(\mathfrak{Y}_{+}^0)^{\varrho_{+},\varrho_{0}})\ar[d] \ar[r]^-{\cong} &  \mathfrak{X}_{0}^{\dagger,\varrho} \ar[d] \\
	\mathfrak{A}^{\varrho_{-},\varrho_{0}}\times\mathfrak{A}^{\varrho_{+},\varrho_{0}}\ar[r]^-{\cong} & \mathfrak{C}_{0}^{\dagger,\varrho} &
}
\een

\end{proposition}

It is shown as in [\cite{Li2}, Section 3.1] that
\begin{proposition}(see [\cite{LW}, Proposition 2.19] and [\cite{Zhou1}, Proposition 7.5])\label{linebd-split} There are  canonical line bundles on $\mathfrak{C}^P$ with sections $(L_{\varrho}, s_{\varrho})$ indexed by $\varrho\in\Lambda_{P}^{\mathrm{spl}}$ such that\\
$(i)$ we have
\ben
\bigotimes_{\varrho\in\Lambda_{P}^\mathrm{spl}}L_{\varrho}\cong\mathcal{O}_{\mathfrak{C}^P},\;\;\;\;\prod_{\varrho\in\Lambda_{P}^\mathrm{spl}}s_{\varrho}=\pi^*t,
\een
where $\pi: \mathfrak{C}^P\to\mathbb{A}^1=\mathrm{Spec}\,\mathbb{C}[t]$ is the canonical map;\\
$(ii)$ $\mathfrak{C}^{\dagger,\varrho}_{0}$ is the closed substack in $\mathfrak{C}^P$ defined by $(s_{\varrho}=0)$.
\end{proposition}

Now let $\pi:\mathcal{X}\to\mathbb{A}^1$ be also a family of projective Deligne-Mumford stacks of relative dimension three
and assume that the generalized Hilbert polynomial with respect to Hilbert homomorphism $P\in \mathrm{Hom}(K^0(\mathcal{X}),\mathbb{Z})$ is of degree one.
The stack $\mathfrak{PT}_{\mathfrak{X}/\mathfrak{C}}^P$ is defined in Section 5.3 parameterizing stable orbifold PT pairs with fixed Hilbert homomorphism $P$. As in [\cite{MPT}, Section 3.8], for any $\mathbb{A}^1$-scheme $S$, we impose the finite resolution assumption on $\mathfrak{PT}_{\mathfrak{X}/\mathfrak{C}}^P$, i.e., for any object $(\overline{\xi},\overline{\varphi})$ in 
$\mathfrak{PT}_{\mathfrak{X}/\mathfrak{C}}^P(S)$ represented by $(\xi,\varphi)$ with the $\mathbb{A}^1$-map $\xi:S_{\xi}=\coprod S_{i}\to\mathbb{A}^{k+1}$, the associated $S_{\xi}$-flat family of stable orbifold PT pairs $\varphi:\mathcal{O}_{\mathtt{X}_{S_{\xi}}}\to\mathcal{F}$ on $\mathtt{X}_{S_{\xi}}=\xi^*\mathcal{X}(k)$
satisfies that $\mathcal{F}$ has finite locally  free resolution on each fiber where $S_{\xi}:=\coprod S_{i}\to S$ is an $\acute{e}$tale covering of $S$.  Let $(\mathcal{Y},\mathcal{D})$ be a smooth pair with a $3$-dimensional projective Deligne-Mumford stack $\mathcal{Y}$, we also impose this finite resolution  assumption on $\mathfrak{PT}_{\mathfrak{Y}/\mathfrak{A}}^{P}$.  These assumptions are needed such that  orbifold Chern character of some universal sheaves  in Section 6.2 is  well-defined. As shown in Section 5.4, the stacks $\mathfrak{PT}_{\mathfrak{X}/\mathfrak{C}}^P$ and $\mathfrak{PT}_{\mathfrak{Y}/\mathfrak{A}}^{P}$ are proper Deligne-Mumford stacks of finite type. Due to [\cite{Zhou1}, Corollary 3.9] for admissible sheaves, we next take Hilbert homomorphisms defined in Section 5.3 as weight assignments. Then we have the following morphism
\ben
\pi_{P}: \mathfrak{PT}_{\mathfrak{X}/\mathfrak{C}}^P\to\mathfrak{C}^P
\een
which is the natural map to the  base with the associated weight assignments from Hilbert homomorphisms for stable orbifold PT pairs. For $\varrho\in\Lambda_{P}^{\mathrm{spl}}$, define
\ben
\mathfrak{PT}_{\mathfrak{X}_{0}^\dagger/\mathfrak{C}_{0}^\dagger}^\varrho:=\mathfrak{PT}_{\mathfrak{X}/\mathfrak{C}}^P\times_{\mathfrak{C}^P}\mathfrak{C}_{0}^{\dag,\varrho}
\een
which is the open and closed substack of the proper Deligne-Mumford stack $\mathfrak{PT}_{\mathfrak{X}_{0}^\dagger/\mathfrak{C}_{0}^\dagger}^P:=\mathfrak{PT}_{\mathfrak{X}/\mathfrak{C}}^P\times_{\mathfrak{C}^P}\mathfrak{C}_{0}^{\dag,P}$ and  parameterizes those families of stable orbifold PT pairs with Hilbert homomorphisms $\varrho_{-}$, $\varrho_{0}$, $\varrho_{+}$ as weight assignments on the fiber decomposition. Similarly, suppose $P, P^\prime\in\mathrm{Hom}(K^0(\mathcal{Y}),\mathbb{Z})$ with $P^\prime\preceq P$  (their associated generalized Hilbert polynomials satisfying the relation $\preceq$ defined at the end of [\cite{Zhou1}, Section 5.3]), define $\mathfrak{PT}_{\mathfrak{Y}/\mathfrak{A}}^{P,P^\prime}\subset\mathfrak{PT}_{\mathfrak{Y}/\mathfrak{A}}^{P}$ to be the substack parameterizing those families of stable orbifold PT pairs with Hilbert homomorphism $P$ on each fiber and $P^\prime$ on its corresponding distinguished divisor. Then $\mathfrak{PT}_{\mathfrak{Y}/\mathfrak{A}}^{P,P^\prime}$ is again a proper Deligne-Mumford stack of finite type. For $\varrho\in\Lambda_{P}^{\mathrm{spl}}$,
both $\mathfrak{PT}_{\mathfrak{Y}_{-}/\mathfrak{A}}^{\varrho_{-},\varrho_{0}}$ and $\mathfrak{PT}_{\mathfrak{Y}_{+}/\mathfrak{A}}^{\varrho_{+},\varrho_{0}}$ have 
 evaluation morphisms to Hilbert stack $\mathrm{Hilb}_{\mathcal{D}}^{\varrho_{0}}$ by admissibility.
Then we have the following diagram
\be\label{diagram}
\xymatrixcolsep{3pc}\xymatrix{
	\mathfrak{PT}_{\mathfrak{Y}_{-}/\mathfrak{A}}^{\varrho_{-},\varrho_{0}}\times_{\mathrm{Hilb}_{\mathcal{D}}^{\varrho_{0}}}\mathfrak{PT}_{\mathfrak{Y}_{+}/\mathfrak{A}}^{\varrho_{+},\varrho_{0}}\ar[d] \ar[r]^-{\Phi_{\varrho}} &  \mathfrak{PT}_{\mathfrak{X}_{0}^\dagger/\mathfrak{C}_{0}^\dagger}^\varrho \ar[d] \\
	\mathfrak{A}^{\varrho_{-},\varrho_{0}}\times\mathfrak{A}^{\varrho_{+},\varrho_{0}}\ar[r]^-{\cong} & \mathfrak{C}_{0}^{\dagger,\varrho} &
}
\ee

By Propositions \ref{split-iso} and \ref{linebd-split}, gluing by admissibility as in [\cite{Zhou1}, Propositions 7.6 and 7.7], we have

\begin{proposition}([\cite{LW}, Theorem 5.28])\label{split-data}
Let $(L_{\varrho}, s_{\varrho})$ be the line bundle and section on $\mathfrak{C}^P$ indexed by $\varrho\in\Lambda_{P}^\mathrm{spl}$. Then \\
$(i)$ we have $\otimes_{\varrho\in\Lambda_{P}^{\mathrm{spl}}}\,\pi_{P}^*L_{\varrho}\cong\mathcal{O}_{\mathfrak{PT}^P_{\mathfrak{X}/\mathfrak{C}}}$ and $\prod_{\varrho\in\Lambda_{P}^{\mathrm{spl}}}\,\pi_{P}^*s_{\varrho}=\pi_{P}^*\pi^*t$, where $\pi: \mathfrak{C}^P\to\mathbb{A}^1=\mathrm{Spec}\,\mathbb{C}[t]$ is the canonical map;\\
$(ii)$ the stack $\mathfrak{PT}_{\mathfrak{X}_{0}^\dag/\mathfrak{C}_{0}^\dag}^{\varrho}$ is the closed substack in $\mathfrak{PT}^P_{\mathfrak{X}/\mathfrak{C}}$ defined by $(\pi_{P}^*s_{\varrho}=0)$;\\
$(iii)$ the morphism $\Phi_{\varrho}$ is an isomorphism.
\end{proposition}

\subsection{Absolute and relative orbifold Pandharipande-Thomas theory}
 We will give the definition of  the absolute and relative orbifold Pandharipande-Thomas invariants following the similar definition in Donaldson-Thomas case [\cite{Zhou1}, Sections 6 and 8] (see also [\cite{MNOP2,LW}]) together with the similar definition of descendents in 
 [\cite{PT2}]. We briefly recall   the absolute case in [\cite{Lyj}, Section 5.3]. Let $\mathcal{Z}$ be a $3$-dimensional smooth projective Deligne-Mumford stack over $\mathbb{C}$ with a moduli scheme $\pi:\mathcal{Z}\to Z$ and a polarization $(\mathcal{E}_{\mathcal{Z}},\mathcal{O}_{Z}(1))$. Let $K(\mathcal{Z}):=K^0(\mathrm{Coh}(\mathcal{Z}))_{\mathbb{Q}}=K^0(D^b(\mathcal{Z}))_{\mathbb{Q}}$ be the Grothendieck group of $\mathcal{Z}$ over $\mathbb{Q}$. 
Define the Euler pairing $\chi: K(\mathcal{Z})\times K(\mathcal{Z})\to\mathbb{Z}$  by
 \ben
 \chi(\mathcal{F} ,\mathcal{G})=\sum_{i}(-1)^i\dim\mathrm{Hom}(\mathcal{F},\mathcal{G}[i]),
 \een
 where $\mathcal{F},\mathcal{G}\in K(\mathcal{Z})$.
As in [\cite{Zhou1}, Section 8.1], the pairing $\chi$ is  nondegenerate by the projectivity of $\mathcal{Z}$ and hence one can view $P\in\mathrm{Hom}(K(\mathcal{Z}),\mathbb{Q})$ as an element in $K(\mathcal{Z})$. Let $F_{\bullet}K(\mathcal{Z})$ be the natural topological filtration of $K(\mathcal{Z})$. Let $P\in F_{1}K(\mathcal{Z})$.
By Remark \ref{absolute} and [\cite{Lyj}, Theorem 5.21] (where $\beta$  is corresponding to $P$), there exists a virtual fundamental class $[\mathrm{PT}_{\mathcal{Z}/\mathbb{C}}^P]^{\mathrm{vir}}\in A_{*}(\mathrm{PT}_{\mathcal{Z}/\mathbb{C}}^P)$.
 
We briefly recollect some definitions in [\cite{Tse}, Section 2 and Appendix A] as follows. Let $I\mathcal{Z}$ be the inertia stack and $\pi_{0}:I\mathcal{Z}\to\mathcal{Z}$ be the natural projection. We have a decomposition of $I\mathcal{Z}$ as a disjoint union $I\mathcal{Z}=\coprod_{i\in\mathcal{I}}\mathcal{Z}_{i}$ for some index set $\mathcal{I}$, and a canonical involution $\iota: I\mathcal{Z}\to I\mathcal{Z}$. Given a vector bundle $W$ on $I\mathcal{Z}$, there exists a decomposition $W=\bigoplus_{\vartheta}W^{(\vartheta)}$ where $W^{(\vartheta)}$ is an eigenbundle with the eigenvalue $\vartheta$. And we have a map $\hat{\rho}: K(I\mathcal{Z})\to K(I\mathcal{Z})_{\mathbb{C}}$ defined by $\hat{\rho}:=\sum_{\vartheta}\vartheta W^{(\vartheta)}$. Define $\widetilde{\mathrm{ch}}: K(\mathcal{Z})\to A^*(I\mathcal{Z})_{\mathbb{C}}$ by $\widetilde{\mathrm{ch}}(V):=\mathrm{ch}(\hat{\rho}(\pi_{0}^*V))$ where $\mathrm{ch}$ is the usual Chern character. By [\cite{AGV}, Section 7.3], we have the orbifold (Chen-Ruan) cohomology of $\mathcal{Z}$ as $A^*_{\mathrm{orb}}(\mathcal{Z}):=\bigoplus_{i} A^{*-\mathrm{age}_{i}}(\mathcal{Z}_{i})$ where $\mathrm{age}_{i}$ is the degree shift number. The orbifold Chern character $\widetilde{\mathrm{ch}}^{\mathrm{orb}}: K(\mathcal{Z})\to A^*_{\mathrm{orb}}(\mathcal{Z})$ is defined by $\widetilde{\mathrm{ch}}_{k}^{\mathrm{orb}}|_{\mathcal{Z}_{i}}:=\widetilde{\mathrm{ch}}_{k-\mathrm{age}_{i}}|_{\mathcal{Z}_{i}}$. 
Since $\mathrm{PT}_{\mathcal{Z}/\mathbb{C}}^P$ is a fine moduli space (see [\cite{Lyj}, Theorem 5.20]), there is a universal complex $\hat{\mathbb{I}}^\bullet:=\{\mathcal{O}_{\mathcal{Z}\times \mathrm{PT}_{\mathcal{Z}/\mathbb{C}}^P}\to\mathbb{F}\}$ where $\mathbb{F}$ is the universal sheaf on $\mathcal{Z}\times \mathrm{PT}_{\mathcal{Z}/\mathbb{C}}^P$. Since $\mathbb{F}$ has a locally free resolution of finite length by [\cite{Lyj}, Lemma 5.9], then $\widetilde{\mathrm{ch}}^{\mathrm{orb}}(\mathbb{F})\in A^*(I\mathcal{Z}\times\mathrm{PT}_{\mathcal{Z}/\mathbb{C}}^P)$.
For any $\gamma\in A^l_{\mathrm{orb}}(\mathcal{Z})$, define the operator $\widetilde{\mathrm{ch}}_{k+2}^{\mathrm{orb}}(\gamma): A_{*}(\mathrm{PT}_{\mathcal{Z}/\mathbb{C}}^P)\to A_{*-k+1-l}(\mathrm{PT}_{\mathcal{Z}/\mathbb{C}}^P)$ to be
\ben
\widetilde{\mathrm{ch}}_{k+2}^{\mathrm{orb}}(\gamma)(\xi):=\pi_{2*}\left(\widetilde{\mathrm{ch}}^{\mathrm{orb}}_{k+2}(\mathbb{F})\cdot \iota^*\pi_{1}^*\gamma\cap\pi_{2}^*\xi\right)
\een
where the map $\pi_{i}$ is the natural projections from $I\mathcal{Z}\times\mathrm{PT}_{\mathcal{Z}/\mathbb{C}}^P$ to the $i$-th factor for $i=1,2$.
Now we have the following definition of  absolute orbifold Pandharipande-Thomas invariants.
\begin{definition}([\cite{Lyj}, Definition 5.23])\label{absolute-def}
 Given $\gamma_{i}\in A_{\mathrm{orb}}^*(\mathcal{Z})$, $1\leq i\leq r$, define the Pandharipande-Thomas invariants with descendents as
 \ben
 \bigg\langle\prod_{i=1}^r\tau_{k_{i}}(\gamma_{i})\bigg\rangle_{\mathcal{Z}}^{P}&:=&\int_{[\mathrm{PT}_{\mathcal{Z}/\mathbb{C}}^P]^{\mathrm{vir}}}\prod_{i=1}^r\widetilde{\mathrm{ch}}_{k_{i}+2}^{\mathrm{orb}}(\gamma_{i})
 =\deg\left[\prod_{i=1}^r\widetilde{\mathrm{ch}}_{k_{i}+2}^{\mathrm{orb}}(\gamma_{i})\cdot [\mathrm{PT}_{\mathcal{Z}/\mathbb{C}}^P]^{\mathrm{vir}}\right]_{0}.
 \een
\end{definition}

Next, we will consider the perfect relative obstruction theories and the virtual fundamental classes  for moduli stacks $\mathfrak{PT}^P_{\mathfrak{X}/\mathfrak{C}}$, $\mathfrak{PT}_{\mathfrak{X}_{0}^\dagger/\mathfrak{C}_{0}^\dagger}^\varrho$, and 
$\mathfrak{PT}_{\mathfrak{Y}_{\pm}/\mathfrak{A}}^{\varrho_{\pm},\varrho_{0}}$  which are needed for defining the relative orbifold Pandharipande-Thomas invariants and deriving degeneration formulas in both cycle and numerical versions later where  $\varrho=(\varrho_{-},\varrho_{+},\varrho_{0})\in\Lambda_{P}^{\mathrm{spl}}$.

Let $\pi_{\mathfrak{PT}^P_{\mathfrak{X}/\mathfrak{C}}}: \mathfrak{X}^P\times_{\mathfrak{C}^P}\mathfrak{PT}^P_{\mathfrak{X}/\mathfrak{C}}\to\mathfrak{PT}^P_{\mathfrak{X}/\mathfrak{C}}$ and $\pi_{\mathfrak{X}^P}:\mathfrak{X}^P\times_{\mathfrak{C}^P}\mathfrak{PT}^P_{\mathfrak{X}/\mathfrak{C}}\to\mathfrak{X}^P$ be the natural projections. As in [\cite{MPT}, Section 3.8], due to the finite autoequivalence groups by the definition of $\mathfrak{PT}^P_{\mathfrak{X}/\mathfrak{C}}$, we have a universal complex $\bar{\mathbb{I}}^\bullet$ over $\mathfrak{X}^P\times_{\mathfrak{C}^P}\mathfrak{PT}^P_{\mathfrak{X}/\mathfrak{C}}$.  Now we have the following truncated Atiyah class
\ben
\mathrm{At}(\bar{\mathbb{I}}^\bullet)\in\mathrm{Ext}^1(\bar{\mathbb{I}}^\bullet,\bar{\mathbb{I}}^\bullet\otimes\mathbb{L}_{\mathfrak{X}^P\times_{\mathfrak{C}^P}\mathfrak{PT}^P_{\mathfrak{X}/\mathfrak{C}}/\mathfrak{C}^P})
\een 
where $\mathbb{L}_{\mathfrak{X}^P\times_{\mathfrak{C}^P}\mathfrak{PT}^P_{\mathfrak{X}/\mathfrak{C}}/\mathfrak{C}^P}$ is the truncated  complex $\tau^{\geq-1}L^\bullet_{\mathfrak{X}^P\times_{\mathfrak{C}^P}\mathfrak{PT}^P_{\mathfrak{X}/\mathfrak{C}}/\mathfrak{C}^P}$ with $L^\bullet_{\mathfrak{X}^P\times_{\mathfrak{C}^P}\mathfrak{PT}^P_{\mathfrak{X}/\mathfrak{C}}/\mathfrak{C}^P}$ being the cotangent complex of the map $\mathfrak{X}^P\times_{\mathfrak{C}^P}\mathfrak{PT}^P_{\mathfrak{X}/\mathfrak{C}}\to\mathfrak{C}^P$  (is of Deligne-Mumford type by Definition \ref{degen-stacks}). One can refer to [\cite{LMB}, Chapitre 17] and [\cite{Ols1}, Section 8] for the definition and properties of the cotangent complex. Using the standard process as in Section 4, we have the following map
\ben
\Phi: \mathbf{E}^\bullet:=R\pi_{\mathfrak{PT}^P_{\mathfrak{X}/\mathfrak{C}}*}(R\mathcal{H}om(\bar{\mathbb{I}}^\bullet,\bar{\mathbb{I}}^\bullet)_{0}\otimes\pi_{\mathfrak{X}^P}^*\omega_{\mathfrak{X}^P/\mathfrak{C}^P})[2]\to\mathbb{L}_{\mathfrak{PT}^P_{\mathfrak{X}/\mathfrak{C}}/\mathfrak{C}^P}, 
\een
where we use Serre duality for simple normal crossing families as in [\cite{Zhou1}, Section 7.2].
\begin{lemma} \label{perfect}
	Suppose $\mathcal{X}\to\mathbb{A}^1$ is a simple degeneration of relative dimension three and is a family of projective Deligne-Mumford stacks. Let $\varphi:\mathcal{O}_{\mathcal{X}_{0}[k]}\to\mathcal{F}$ be an admissible orbifold PT pair and $\bar{\mathbf{I}}^{\bullet}:=\{\mathcal{O}_{\mathcal{X}_{0}[k]}\xrightarrow{\varphi}\mathcal{F}\}$ be a complex concentrated in degree 0 and 1 for some $k\geq0$. Then we have
	\ben
	\mathrm{Ext}^{i}(\bar{\mathbf{I}}^\bullet,\bar{\mathbf{I}}^\bullet)_{0}=0\;\;\;\mbox{if $i\neq1,2.$}
	\een
\end{lemma}
\begin{proof}
	We will deal with the case for $k=0$, for the other cases one can follow the same argument by induction. Now we have the following short exact sequence on $\mathcal{X}_{0}=\mathcal{Y}_{-}\cup_{\mathcal{D}}\mathcal{Y}_{+}$:
	\be\label{SES1}
	0\to\mathcal{O}_{\mathcal{X}_{0}}\to\mathcal{O}_{\mathcal{Y}_{-}}\oplus\mathcal{O}_{\mathcal{Y}_{+}}\to\mathcal{O}_{\mathcal{D}}\to0.
	\ee
	Since $\varphi:\mathcal{O}_{\mathcal{X}_{0}}\to\mathcal{F}$ is an admissible orbifold PT pair, tensoring \eqref{SES1} with
	$\bar{\mathbf{I}}^\bullet=\{\mathcal{O}_{\mathcal{X}_{0}}\to\mathcal{F}\}$ we have the following exact triple of complexes
	\be\label{SES-complexes}
	0\to\bar{\mathbf{I}}^\bullet\to\bar{\mathbf{I}}^\bullet_{-}\oplus\bar{\mathbf{I}}^\bullet_{+}\to\bar{\mathbf{I}}^\bullet_{\mathcal{D}}\to0
	\ee
	where $\bar{\mathbf{I}}^\bullet_{\pm}$ and $\bar{\mathbf{I}}^\bullet_{\mathcal{D}}$ denote the restriction of $\bar{\mathbf{I}}^\bullet$ on $\mathcal{Y}_{\pm}$ and $\mathcal{D}$ respectively. By using [\cite{Zhou1}, Corollary 3.9 and Proposition 3.23] for the  admissibility of sheaves and using [\cite{Nir1}, Remark 3.3-(2)] (see also [\cite{Lieb}, Remark 2.2.6.15]) for the purity of sheaves, $\varphi_{\pm}:\mathcal{O}_{\mathcal{Y}_{\pm}}\to\mathcal{F}|_{\mathcal{Y}_{\pm}}$ are also admissible orbifold PT pairs and $\varphi_{\mathcal{D}}:\mathcal{O}_{\mathcal{D}}\to\mathcal{F}|_{\mathcal{D}}$ is a quotient. By applying $R\mathrm{Hom}^\bullet(\bar{\mathbf{I}}^\bullet,\cdot)$ to \eqref{SES-complexes}, we have 
	\be\label{LES1}
	\cdots\to\mathrm{Ext}^{-1}(\bar{\mathbf{I}}_{\mathcal{D}}^\bullet,\bar{\mathbf{I}}^\bullet_{\mathcal{D}})\to\mathrm{Hom}(\bar{\mathbf{I}}^\bullet,\bar{\mathbf{I}}^\bullet)\to\mathrm{Hom}(\bar{\mathbf{I}}^\bullet_{-},\bar{\mathbf{I}}^\bullet_{-})\oplus\mathrm{Hom}(\bar{\mathbf{I}}^\bullet_{+},\bar{\mathbf{I}}^\bullet_{+})\to\mathrm{Hom}(\bar{\mathbf{I}}_{\mathcal{D}}^\bullet,\bar{\mathbf{I}}^\bullet_{\mathcal{D}})
	\ee
	and 
	\be\label{LES2}
	\cdots\to\mathrm{Ext}^{2}(\bar{\mathbf{I}}_{\mathcal{D}}^\bullet,\bar{\mathbf{I}}^\bullet_{\mathcal{D}})\to\mathrm{Ext}^3(\bar{\mathbf{I}}^\bullet,\bar{\mathbf{I}}^\bullet)\to\mathrm{Ext}^3(\bar{\mathbf{I}}^\bullet_{-},\bar{\mathbf{I}}^\bullet_{-})\oplus\mathrm{Ext}^3(\bar{\mathbf{I}}^\bullet_{+},\bar{\mathbf{I}}^\bullet_{+})\to\mathrm{Ext}^3(\bar{\mathbf{I}}_{\mathcal{D}}^\bullet,\bar{\mathbf{I}}^\bullet_{\mathcal{D}}).
	\ee
	Since $\mathcal{F}$ is a pure 1-dimensional admissible sheaf and $\mathrm{coker}\,\varphi$ is normal to $\mathcal{D}$, the complex $\bar{\mathbf{I}}^\bullet_{\mathcal{D}}$ is quasi-isomorphic to an ideal sheaf of some points in $\mathcal{D}$. Then we have $\mathrm{Ext}^{-1}(\bar{\mathbf{I}}_{\mathcal{D}}^\bullet,\bar{\mathbf{I}}^\bullet_{\mathcal{D}})=0$ and $\mathrm{Ext}^{2}(\bar{\mathbf{I}}_{\mathcal{D}}^\bullet,\bar{\mathbf{I}}^\bullet_{\mathcal{D}})_{0}=0$, where the vanishing of the latter is proved by using the similar argument in [\cite{Lyj}, Lemmas 5.15 and 5.16] to show $\mathrm{Hom}(\bar{\mathbf{I}}_{\mathcal{D}}^\bullet,\bar{\mathbf{I}}^\bullet_{\mathcal{D}})\cong\mathcal{O}_{\mathcal{D}}$ and then tensoring with $\omega_{\mathcal{D}}$ to apply Serre duality in [\cite{Lyj}, Lemma 5.11] since $\mathcal{D}$ is a smooth projective Deligne-Mumford stack of dimension 2 by assumption. By [\cite{Lyj}, Lemma 5.16], we have $\mathrm{Hom}(\bar{\mathbf{I}}^\bullet_{\pm},\bar{\mathbf{I}}^\bullet_{\pm})_{0}=\mathrm{Ext}^3(\bar{\mathbf{I}}^\bullet_{\pm},\bar{\mathbf{I}}^\bullet_{\pm})_{0}=0$ which implies $\mathrm{Hom}(\bar{\mathbf{I}}^\bullet,\bar{\mathbf{I}}^\bullet)_{0}=\mathrm{Ext}^3(\bar{\mathbf{I}}^\bullet,\bar{\mathbf{I}}^\bullet)_{0}=0$ by taking the traceless part of \eqref{LES1} and \eqref{LES2} respectively. Applying the similar argument for the cases in higher degrees and lower degrees, the proof is completed. 
\end{proof} 

With the universal approach to deformation-obstruction theory of complexes developed in [\cite{HT09}], as in [\cite{MPT}, Section 3.8], one can follow the similar argument in [\cite{HT09}, Theorem 4.1 and Lemma 4.2] together with using Lemma \ref{perfect}  to obtain

\begin{theorem}
In the sense of [\cite{BF}], the map $\Phi: \mathbf{E}^\bullet\to\mathbb{L}_{\mathfrak{PT}^P_{\mathfrak{X}/\mathfrak{C}}/\mathfrak{C}^P}$ is a perfect relative obstruction theory for 
$\mathfrak{PT}^P_{\mathfrak{X}/\mathfrak{C}}$ over $\mathfrak{C}^P$.
\end{theorem}
\begin{corollary}
There exists a virtual fundamental class $[\mathfrak{PT}^P_{\mathfrak{X}/\mathfrak{C}}]^{\mathrm{vir}}\in
A_{*}(\mathfrak{PT}^P_{\mathfrak{X}/\mathfrak{C}})$.
\end{corollary}

For $\varrho=(\varrho_{-},\varrho_{+},\varrho_{0})\in\Lambda_{P}^{\mathrm{spl}}$, let 
\ben
&&\pi_{\mathfrak{PT}_{\mathfrak{X}_{0}^\dagger/\mathfrak{C}_{0}^\dagger}^\varrho}: \mathfrak{X}_{0}^{\dagger,\varrho}\times_{\mathfrak{C}_{0}^{\dagger,\varrho}}\mathfrak{PT}_{\mathfrak{X}_{0}^\dagger/\mathfrak{C}_{0}^\dagger}^\varrho\to\mathfrak{PT}_{\mathfrak{X}_{0}^\dagger/\mathfrak{C}_{0}^\dagger}^\varrho,\\
&&\pi_{\mathfrak{X}_{0}^{\dagger,\varrho}}: \mathfrak{X}_{0}^{\dagger,\varrho}\times_{\mathfrak{C}_{0}^{\dagger,\varrho}}\mathfrak{PT}_{\mathfrak{X}_{0}^\dagger/\mathfrak{C}_{0}^\dagger}^\varrho\to\mathfrak{X}_{0}^{\dagger,\varrho},\\
&&\pi_{\mathfrak{PT}_{\mathfrak{Y}_{\pm}/\mathfrak{A}}^{\varrho_{\pm},\varrho_{0}}}:\mathfrak{Y}_{\pm}^{\varrho_{\pm},\varrho_{0}}\times_{\mathfrak{A}^{\varrho_{\pm},\varrho_{0}}}\mathfrak{PT}_{\mathfrak{Y}_{\pm}/\mathfrak{A}}^{\varrho_{\pm},\varrho_{0}}\to\mathfrak{PT}_{\mathfrak{Y}_{\pm}/\mathfrak{A}}^{\varrho_{\pm},\varrho_{0}},\\
&&\pi_{\mathfrak{Y}_{\pm}^{\varrho_{\pm},\varrho_{0}}}: \mathfrak{Y}_{\pm}^{\varrho_{\pm},\varrho_{0}}\times_{\mathfrak{A}^{\varrho_{\pm},\varrho_{0}}}\mathfrak{PT}_{\mathfrak{Y}_{\pm}/\mathfrak{A}}^{\varrho_{\pm},\varrho_{0}}\to\mathfrak{Y}_{\pm}^{\varrho_{\pm},\varrho_{0}},
\een 
be the natural projections.
Let $\bar{\mathbb{I}}^\bullet_{\varrho}:=\{\mathcal{O}_{\mathfrak{X}_{0}^{\dagger,\varrho}\times_{\mathfrak{C}_{0}^{\dagger,\varrho}}\mathfrak{PT}_{\mathfrak{X}_{0}^\dagger/\mathfrak{C}_{0}^\dagger}^\varrho}\to\mathbb{F}_{\varrho}\}$ and $\bar{\mathbb{I}}^\bullet_{\pm}=\{\mathcal{O}_{\mathfrak{Y}_{\pm}^{\varrho_{\pm},\varrho_{0}}\times_{\mathfrak{A}^{\varrho_{\pm},\varrho_{0}}}\mathfrak{PT}_{\mathfrak{Y}_{\pm}/\mathfrak{A}}^{\varrho_{\pm},\varrho_{0}}}\to\mathbb{F}_{\pm}\}$ be the universal complexes over  $\mathfrak{X}_{0}^{\dagger,\varrho}\times_{\mathfrak{C}_{0}^{\dagger,\varrho}}\mathfrak{PT}_{\mathfrak{X}_{0}^\dagger/\mathfrak{C}_{0}^\dagger}^\varrho$ and  $\mathfrak{Y}_{\pm}^{\varrho_{\pm},\varrho_{0}}\times_{\mathfrak{A}^{\varrho_{\pm},\varrho_{0}}}\mathfrak{PT}_{\mathfrak{Y}_{\pm}/\mathfrak{A}}^{\varrho_{\pm},\varrho_{0}}$ respectively. Similarly,  we have the  perfect relative obstruction theories for $\mathfrak{PT}_{\mathfrak{X}_{0}^\dagger/\mathfrak{C}_{0}^\dagger}^\varrho$ and $\mathfrak{PT}_{\mathfrak{Y}_{\pm}/\mathfrak{A}}^{\varrho_{\pm},\varrho_{0}}$:
\ben
&&\hat{\Phi}: \hat{\mathbf{E}}^\bullet:=R\pi_{\mathfrak{PT}_{\mathfrak{X}_{0}^\dagger/\mathfrak{C}_{0}^\dagger}^\varrho*}(R\mathcal{H}om(\bar{\mathbb{I}}_{\varrho}^\bullet,\bar{\mathbb{I}}_{\varrho}^\bullet)_{0}\otimes\pi_{\mathfrak{X}_{0}^{\dagger,\varrho}}^*\omega_{\mathfrak{X}_{0}^{\dagger,\varrho}/\mathfrak{C}_{0}^{\dagger,\varrho}})[2]\to\mathbb{L}_{\mathfrak{PT}_{\mathfrak{X}_{0}^\dagger/\mathfrak{C}_{0}^\dagger}^\varrho/\mathfrak{C}_{0}^{\dagger,\varrho}},\\
&&\Phi_{\pm}:\mathbf{E}^\bullet_{\pm}:=R\pi_{\mathfrak{PT}_{\mathfrak{Y}_{\pm}/\mathfrak{A}}^{\varrho_{\pm},\varrho_{0}}*}(R\mathcal{H}om(\bar{\mathbb{I}}^\bullet_{\pm},\bar{\mathbb{I}}^\bullet_{\pm})_{0}\otimes\pi_{\mathfrak{Y}_{\pm}^{\varrho_{\pm},\varrho_{0}}}^*\omega_{\mathfrak{Y}_{\pm}^{\varrho_{\pm},\varrho_{0}}/\mathfrak{A}^{\varrho_{\pm},\varrho_{0}}})[2]\to\mathbb{L}_{\mathfrak{PT}_{\mathfrak{Y}_{\pm}/\mathfrak{A}}^{\varrho_{\pm},\varrho_{0}}/\mathfrak{A}^{\varrho_{\pm},\varrho_{0}}},
\een
and the virtual fundamental classes
$[\mathfrak{PT}_{\mathfrak{X}_{0}^\dagger/\mathfrak{C}_{0}^\dagger}^\varrho]^\mathrm{vir}\in A_{*}(\mathfrak{PT}_{\mathfrak{X}_{0}^\dagger/\mathfrak{C}_{0}^\dagger}^\varrho)$ and  $[\mathfrak{PT}_{\mathfrak{Y}_{\pm}/\mathfrak{A}}^{\varrho_{\pm},\varrho_{0}}]^\mathrm{vir}\in A_{*}(\mathfrak{PT}_{\mathfrak{Y}_{\pm}/\mathfrak{A}}^{\varrho_{\pm},\varrho_{0}})$.
Since we have the following Cartesian diagram
\ben
\xymatrixcolsep{3pc}\xymatrix{
\mathfrak{PT}_{\mathfrak{X}_{0}^\dagger/\mathfrak{C}_{0}^\dagger}^\varrho\ar[d] \ar[r] &  \mathfrak{PT}_{\mathfrak{X}/\mathfrak{C}}^P \ar[d] \\
	 \mathfrak{C}_{0}^{\dagger,\varrho}\ar[r]^-{\zeta_{\varrho}}  & \mathfrak{C}^{P} &
}
\een
by [\cite{BF}, Proposition 7.2] and Proposition \ref{linebd-split}-(ii), the perfect relative obstruction theory  for $\mathfrak{PT}_{\mathfrak{X}_{0}^\dagger/\mathfrak{C}_{0}^\dagger}^\varrho$ over $\mathfrak{C}_{0}^{\dagger,\varrho}$ can be also defined by pulling back of $\Phi: \mathbf{E}^\bullet\to\mathbb{L}_{\mathfrak{PT}^P_{\mathfrak{X}/\mathfrak{C}}/\mathfrak{C}^P}$, and  $\zeta_{\varrho}^![\mathfrak{PT}_{\mathfrak{X}/\mathfrak{C}}^P]^\mathrm{vir}=[\mathfrak{PT}_{\mathfrak{X}_{0}^\dagger/\mathfrak{C}_{0}^\dagger}^\varrho]^\mathrm{vir}$.

Now, we define the relative orbifold Pandharipande-Thomas invariants as in [\cite{Zhou1},  Section 8.1]. Let $\mathcal{Y}$ be a $3$-dimensional smooth projective Deligne-Mumford stack with a smooth divisor $\mathcal{D}\subset\mathcal{Y}$. 
Let $K(\mathcal{Y}):=K^0(\mathrm{Coh}(\mathcal{Y}))_{\mathbb{Q}}$ be Grothendieck group of $\mathcal{Y}$ over $\mathbb{Q}$. As before one can view $P\in\mathrm{Hom}(K(\mathcal{Y}),\mathbb{Q})$ as an element in $K(\mathcal{Y})$.
Assume $P\in F_{1}K(\mathcal{Y})$. If $P$ is represented by some admissible sheaf and $i: \mathcal{D}\hookrightarrow\mathcal{Y}$ is the inclusion, then $P_{0}:=i^*P\in F_{0}K(\mathcal{D})$. Using [\cite{Zhou1}, Corollary 5.3] and the identification between $\mathrm{Hom}(K(\mathcal{Y}),\mathbb{Q})$ and $K(\mathcal{Y})$, one can consider the moduli stack $\mathfrak{PT}^{P,P_{0}}_{\mathfrak{Y}/\mathfrak{A}}$ with  $P\in F_{1}K(\mathcal{Y})$ and $P_{0}\in F_{0}K(\mathcal{D})$. Again we have a virtual fundamental class $[\mathfrak{PT}^{P,P_{0}}_{\mathfrak{Y}/\mathfrak{A}}]^\mathrm{vir}$  for the stack $\mathfrak{PT}^{P,P_{0}}_{\mathfrak{Y}/\mathfrak{A}}\to\mathfrak{A}^{P,P_{0}}$. As in the absolute case, for $\gamma\in A_{\mathrm{orb}}^l(\mathcal{Y})$,  define the orbifold Chern character 
\ben
\widetilde{\mathrm{ch}}^\mathrm{orb}_{k+2}(\gamma): A_{*}(\mathfrak{PT}^{P,P_{0}}_{\mathfrak{Y}/\mathfrak{A}})\to A_{*-k+1-l}(\mathfrak{PT}^{P,P_{0}}_{\mathfrak{Y}/\mathfrak{A}})
\een
by
\ben
\widetilde{\mathrm{ch}}^\mathrm{orb}_{k+2}(\gamma)(\xi):=\pi_{2*}\left(\widetilde{\mathrm{ch}}^\mathrm{orb}_{k+2}(\mathbb{F})\cdot\iota^*\bar{\pi}_{1}^*\gamma\cap\pi_{2}^*\xi\right)
\een
where $\mathbb{F}$ is the universal sheaf on $\mathfrak{Y}^{P,P_{0}}\times_{\mathfrak{A}^{P,P_{0}}}\mathfrak{PT}^{P,P_{0}}_{\mathfrak{Y}/\mathfrak{A}}$ with $\widetilde{\mathrm{ch}}^{\mathrm{orb}}(\mathbb{F})\in A_{\mathrm{orb}}^*(\mathfrak{Y}^{P,P_{0}}\times_{\mathfrak{A}^{P,P_{0}}}\mathfrak{PT}^{P,P_{0}}_{\mathfrak{Y}/\mathfrak{A}})$ and the map $\pi_{i}$ is the projections from $I_{\mathfrak{A}^{P,P_{0}}}\mathfrak{Y}^{P,P_{0}}\times_{\mathfrak{A}^{P,P_{0}}}\mathfrak{PT}^{P,P_{0}}_{\mathfrak{Y}/\mathfrak{A}}$ to the $i$-th factor for $i=1,2$.
Here, $I_{\mathfrak{A}^{P,P_{0}}}\mathfrak{Y}^{P,P_{0}}$ is the inertia stack of $\mathfrak{Y}^{P,P_{0}}$ over $\mathfrak{A}^{P,P_{0}}$, the map $\bar{\pi}_{1}$ is the composition map of $\pi_{1}$ with the natural map $I_{\mathfrak{A}^{P,P_{0}}}\mathfrak{Y}^{P,P_{0}}\to I\mathcal{Y}$, and $\iota:I_{\mathfrak{A}^{P,P_{0}}}\mathfrak{Y}^{P,P_{0}}\to I_{\mathfrak{A}^{P,P_{0}}}\mathfrak{Y}^{P,P_{0}}$ is the canonical involution.
Denote by $\mathrm{ev}:\mathfrak{PT}^{P,P_{0}}_{\mathfrak{Y}/\mathfrak{A}}\to\mathrm{Hilb}^{P_{0}}_{\mathcal{D}}$ the envaluation morphism.

\begin{definition}\label{relative-def}
Assume $\gamma_{i}\in A^*_{\mathrm{orb}}(\mathcal{Y})$, $1\leq i\leq r$ and $C\in A^*(\mathrm{Hilb}^{P_{0}}_{\mathcal{D}})$,
the relative orbifold Pandharipande-Thomas invariant is defined by
\ben
\bigg\langle\prod_{i=1}^r\tau_{k_{i}}(\gamma_{i})\bigg|C\bigg\rangle_{\mathcal{Y},\mathcal{D}}^P:&=&\int_{[\mathfrak{PT}^{P,P_{0}}_{\mathfrak{Y}/\mathfrak{A}}]^\mathrm{vir}}\mathrm{ev}^*(C)\cdot\prod_{i=1}^r\widetilde{\mathrm{ch}}^\mathrm{orb}_{k_{i}+2}(\gamma_{i})\\
&=&\deg\bigg[\mathrm{ev}^*(C)\cdot\prod_{i=1}^r\widetilde{\mathrm{ch}}^\mathrm{orb}_{k_{i}+2}(\gamma_{i})\cdot[\mathfrak{PT}^{P,P_{0}}_{\mathfrak{Y}/\mathfrak{A}}]^\mathrm{vir}\bigg]_{0}.
\een
\end{definition}

\subsection{Cycle version of degeneration formula}
In this subsection, we will give a cycle version of degeneration formula following the similar argument for the Donaldson-Thomas case in  [\cite{LW}, Section 6.2] and [\cite{Zhou1}, Section 7.3].
Given a simple degeneration $\pi: \mathcal{X}\to\mathbb{A}^1$ which is a family of projective Deligne-Mumford stacks of relative dimension 3,  let $i_{c}: \{c\}\hookrightarrow\mathbb{A}^1$ be the inclusion for any point $c\in\mathbb{A}^1$. If $c\neq0$, the fiber $\mathcal{X}_{c}$ is a $3$-dimensional smooth projective Deligne-Mumford stack. As in Section 6.2, there is  a perfect obstruction theory for $\mathrm{PT}_{\mathcal{X}_{c}/\mathbb{C}}^P$ and a virtual fundamental class $[\mathrm{PT}_{\mathcal{X}_{c}/\mathbb{C}}^P]^\mathrm{vir}$.  Let the map $\mathfrak{PT}^P_{\mathfrak{X}/\mathfrak{C}}\to\mathbb{A}^1$ be the composition of $\mathfrak{PT}^P_{\mathfrak{X}/\mathfrak{C}}\to\mathfrak{C}^P$ with the canonical map $\mathfrak{C}^P\to\mathbb{A}^1:=\mathrm{Spec}\,\mathbb{C}[t]$. Now, we have two Cartesian diagrams as follows:
\ben
\xymatrixcolsep{3pc}\xymatrix{
	\mathrm{PT}_{\mathcal{X}_{c}/\mathbb{C}}^P\ar[d] \ar[r] &  \mathfrak{PT}^P_{\mathfrak{X}/\mathfrak{C}} \ar[d] \\
\{c\}\ar@{^{(}->}[r]^-{i_{c}} & \mathbb{A}^1 &
}
\een
and
\ben
\xymatrixcolsep{3pc}\xymatrix{
\mathfrak{PT}_{\mathfrak{X}_{0}^\dagger/\mathfrak{C}_{0}^\dagger}^\varrho\ar@{^{(}->}[r]^-{\varsigma_{\varrho}}	& \mathfrak{PT}_{\mathfrak{X}_{0}^\dagger/\mathfrak{C}_{0}^\dagger}^P\ar[d] \ar[r] & \mathfrak{PT}^P_{\mathfrak{X}/\mathfrak{C}} \ar[d] \\
	&\{0\}\ar@{^{(}->}[r]^-{i_{0}} & \mathbb{A}^1. 
}
\een
Combining the diagram \eqref{diagram} with Proposition \ref{split-data}, we have the following Cartesian diagram
\be\label{diagram2}
\xymatrixcolsep{3pc}\xymatrix{
\mathfrak{PT}_{\mathfrak{X}_{0}^\dagger/\mathfrak{C}_{0}^\dagger}^\varrho 	&\ar[l]_-{\Phi_{\varrho}}^-{\cong}\mathfrak{PT}_{\mathfrak{Y}_{-}/\mathfrak{A}}^{\varrho_{-},\varrho_{0}}\times_{\mathrm{Hilb}_{\mathcal{D}}^{\varrho_{0}}}\mathfrak{PT}_{\mathfrak{Y}_{+}/\mathfrak{A}}^{\varrho_{+},\varrho_{0}}\ar[d] \ar[r] &  	\mathfrak{PT}_{\mathfrak{Y}_{-}/\mathfrak{A}}^{\varrho_{-},\varrho_{0}}\times \mathfrak{PT}_{\mathfrak{Y}_{+}/\mathfrak{A}}^{\varrho_{+},\varrho_{0}} \ar[d] \\
&\mathfrak{C}_{0}^{\dagger,\varrho}\times	\mathrm{Hilb}_{\mathcal{D}}^{\varrho_{0}}\ar[r] \ar[d]& \mathfrak{C}_{0}^{\dagger,\varrho}\times\mathrm{Hilb}_{\mathcal{D}}^{\varrho_{0}}\times\mathrm{Hilb}_{\mathcal{D}}^{\varrho_{0}} \ar[d]&\\
	& \mathrm{Hilb}_{\mathcal{D}}^{\varrho_{0}}\ar[r]^-{\Delta}  &  \mathrm{Hilb}_{\mathcal{D}}^{\varrho_{0}}\times\mathrm{Hilb}_{\mathcal{D}}^{\varrho_{0}} &  
}
\ee
where we use  maps (to their bases) $\mathfrak{PT}_{\mathfrak{Y}_{\pm}/\mathfrak{A}}^{\varrho_{\pm},\varrho_{0}}\to\mathfrak{A}^{\varrho_{\pm},\varrho_{0}}$ and evaluation maps $\mathfrak{PT}_{\mathfrak{Y}_{\pm}/\mathfrak{A}}^{\varrho_{\pm},\varrho_{0}}\to\mathrm{Hilb}_{\mathcal{D}}^{\varrho_{0}}$. Denote  by $g$ the composition  $\mathfrak{PT}_{\mathfrak{Y}_{-}/\mathfrak{A}}^{\varrho_{-},\varrho_{0}}\times_{\mathrm{Hilb}_{\mathcal{D}}^{\varrho_{0}}}\mathfrak{PT}_{\mathfrak{Y}_{+}/\mathfrak{A}}^{\varrho_{+},\varrho_{0}}\to\mathrm{Hilb}_{\mathcal{D}}^{\varrho_{0}}\times\mathfrak{C}_{0}^{\dagger,\varrho}\to\mathrm{Hilb}_{\mathcal{D}}^{\varrho_{0}}$.

By Proposition \ref{linebd-split}-(ii) and Proposition \ref{split-data}-(ii), using the notion of refined Gysin homomorphism which is generalized to the case of Deligne-Mumford stacks in [\cite{Vis}] and to the case of Artin stacks in [\cite{Kre1}], as in [\cite{LW}, Section 6.2], one has the localized first Chern class  $c_{1}^{\mathrm{loc}}(L_{\varrho},s_{\varrho})$   satisfying
\ben
[\mathfrak{PT}_{\mathfrak{X}_{0}^\dagger/\mathfrak{C}_{0}^\dagger}^\varrho]^\mathrm{vir}=c_{1}^{\mathrm{loc}}(L_{\varrho},s_{\varrho})[\mathfrak{PT}^P_{\mathfrak{X}/\mathfrak{C}}]^\mathrm{vir}
\een
via the definition of 
$
c_{1}^{\mathrm{loc}}(L_{\varrho},s_{\varrho}): A_*(\mathfrak{PT}^P_{\mathfrak{X}/\mathfrak{C}})\to A_{*-1}(\mathfrak{PT}_{\mathfrak{X}_{0}^\dagger/\mathfrak{C}_{0}^\dagger}^\varrho)
$
as the refined Gysin homomorphism $0_{\pi_{P}^*L_{\varrho}}^!$ of some normal bundle as in [\cite{Li2}, Section 3.1] through the diagram
\ben
\xymatrixcolsep{5pc}\xymatrix{
	\mathfrak{PT}_{\mathfrak{X}_{0}^\dagger/\mathfrak{C}_{0}^\dagger}^\varrho\ar[d] \ar[r] &  \mathfrak{PT}^P_{\mathfrak{X}/\mathfrak{C}} \ar[d]^-{\pi_{P}^*s_{\varrho}} \\
	 \mathfrak{PT}^P_{\mathfrak{X}/\mathfrak{C}}\ar[r]^-{0_{\pi_{P}^*L_{\varrho}}} & \pi_{P}^*L_{\varrho} &
}
\een
and the pullback property in [\cite{BF}, Proposition 7.2].

By [\cite{BF}, Proposition 7.2] and Proposition \ref{split-data} again, we have
\begin{proposition}\label{split-cycles}
We have the following identity of cycle classes
\ben
i_{0}^![\mathfrak{PT}^P_{\mathfrak{X}/\mathfrak{C}}]^\mathrm{vir}=\sum_{\varrho\in\Lambda_{P}^{\mathrm{spl}}}\varsigma_{\varrho*}[\mathfrak{PT}_{\mathfrak{X}_{0}^\dagger/\mathfrak{C}_{0}^\dagger}^\varrho]^\mathrm{vir}
\een
where $\varsigma_{\varrho}:\mathfrak{PT}_{\mathfrak{X}_{0}^\dagger/\mathfrak{C}_{0}^\dagger}^\varrho\to\mathfrak{PT}_{\mathfrak{X}_{0}^\dagger/\mathfrak{C}_{0}^\dagger}^P$ is the inclusion.
\end{proposition}
By the similar argument in the proof of [\cite{Zhou1}, Theorem 7.11] (see also [\cite{LW}, Theorem 6.6]), we have the following degeneration formula in cycle version, which is the stacky case of [\cite{LW}, Theorem 6.9].
\begin{theorem}\label{cycle-degenerate}
We  have the following identities
\ben
&&i_{c}^![\mathfrak{PT}^P_{\mathfrak{X}/\mathfrak{C}}]^\mathrm{vir}=[	\mathrm{PT}_{\mathcal{X}_{c}/\mathbb{C}}^P]^\mathrm{vir},\\
&&i_{0}^![\mathfrak{PT}^P_{\mathfrak{X}/\mathfrak{C}}]^\mathrm{vir}=\sum_{\varrho\in\Lambda_{P}^\mathrm{spl}}\varsigma_{\varrho*}\Delta^!([\mathfrak{PT}_{\mathfrak{Y}_{-}/\mathfrak{A}}^{\varrho_{-},\varrho_{0}}]^\mathrm{vir}\times[\mathfrak{PT}_{\mathfrak{Y}_{+}/\mathfrak{A}}^{\varrho_{+},\varrho_{0}}]^\mathrm{vir}).
\een
\end{theorem}
\begin{proof}
 According to Proposition \ref{split-iso}, for each $\varrho\in\Lambda_{P}^\mathrm{spl}$,
we have the following isomorphism
\ben
&&\mathfrak{X}_{0}^{\dagger,\varrho}\times_{\mathfrak{C}_{0}^{\dagger,\varrho}}\mathfrak{PT}_{\mathfrak{X}_{0}^\dagger/\mathfrak{C}_{0}^\dagger}^\varrho\\
\cong&&((\mathfrak{Y}_{-}^{\varrho_{-},\varrho_{0}}\times\mathfrak{A}^{\varrho_{+},\varrho_{0}})\cup_{\mathfrak{A}^{\varrho_{-},\varrho_{0}}\times \mathcal{D}\times\mathfrak{A}^{\varrho_{+},\varrho_{0}}}(\mathfrak{A}^{\varrho_{-},\varrho_{0}}\times(\mathfrak{Y}_{+}^0)^{\varrho_{+},\varrho_{0}}))\times_{\mathfrak{C}_{0}^{\dagger,\varrho}}\mathfrak{PT}_{\mathfrak{X}_{0}^\dagger/\mathfrak{C}_{0}^\dagger}^\varrho\\
\cong&&((\mathfrak{Y}_{-}^{\varrho_{-},\varrho_{0}}\times\mathfrak{A}^{\varrho_{+},\varrho_{0}})\times_{\mathfrak{C}_{0}^{\dagger,\varrho}}\mathfrak{PT}_{\mathfrak{X}_{0}^\dagger/\mathfrak{C}_{0}^\dagger}^\varrho)\cup_{\mathcal{D}\times\mathfrak{PT}_{\mathfrak{X}_{0}^\dagger/\mathfrak{C}_{0}^\dagger}^\varrho}((\mathfrak{A}^{\varrho_{-},\varrho_{0}}\times(\mathfrak{Y}_{+}^0)^{\varrho_{+},\varrho_{0}})\times_{\mathfrak{C}_{0}^{\dagger,\varrho}}\mathfrak{PT}_{\mathfrak{X}_{0}^\dagger/\mathfrak{C}_{0}^\dagger}^\varrho).
\een
 Let $\mathrm{pr}_{\pm}: \mathfrak{PT}_{\mathfrak{Y}_{-}/\mathfrak{A}}^{\varrho_{-},\varrho_{0}}\times\mathfrak{PT}_{\mathfrak{Y}_{+}/\mathfrak{A}}^{\varrho_{+},\varrho_{0}}\to\mathfrak{PT}_{\mathfrak{Y}_{\pm}/\mathfrak{A}}^{\varrho_{\pm},\varrho_{0}}$ be the projections. Let $\mathrm{ev}_{\varrho}: \mathfrak{PT}_{\mathfrak{X}_{0}^\dagger/\mathfrak{C}_{0}^\dagger}^\varrho\to\mathrm{Hilb}_{\mathcal{D}}^{\varrho_{0}}$ be the evaluation morphism and $\iota_{\mathcal{D}}:\mathcal{D}\times\mathfrak{PT}_{\mathfrak{X}_{0}^\dagger/\mathfrak{C}_{0}^\dagger}^\varrho \hookrightarrow \mathfrak{X}_{0}^{\dagger,\varrho}\times_{\mathfrak{C}_{0}^{\dagger,\varrho}}\mathfrak{PT}_{\mathfrak{X}_{0}^\dagger/\mathfrak{C}_{0}^\dagger}^\varrho$ be the inclusion map of the universal distinguished divisor. Let $\bar{\mathbb{I}}_{0}$ be the universal ideal sheaf on $\mathcal{D}\times\mathrm{Hilb}_{\mathcal{D}}^{\varrho_{0}}$ of the Hilbert scheme $\mathrm{Hilb}_{\mathcal{D}}^{\varrho_{0}}$. Then on each fiber, $(\mathrm{id}_{\mathcal{D}}\times\mathrm{ev}_{\varrho})^*\bar{\mathbb{I}}_{0}$ is quasi-isomorphic to the restriction of the universal complex $\bar{\mathbb{I}}^\bullet_{\varrho}$ to $\mathcal{D}\times\mathfrak{PT}_{\mathfrak{X}_{0}^\dagger/\mathfrak{C}_{0}^\dagger}^\varrho$ due to the admissibility.
Combined with the above isomorphism, the diagram \eqref{diagram} and Proposition \ref{split-data}-(iii), we have the following short exact sequence
\ben
0\to\bar{\mathbb{I}}^\bullet_{\varrho}\to\Delta^*(\mathrm{pr}_{-}^*\bar{\mathbb{I}}^\bullet_{-}\oplus\mathrm{pr}_{+}^*\bar{\mathbb{I}}^\bullet_{+})\to \iota_{\mathcal{D}*}(\mathrm{id}_{\mathcal{D}}\times\mathrm{ev}_{\varrho})^*\bar{\mathbb{I}}_{0}\to0.
\een
where $\mathrm{pr}_{\pm}$ and $\Delta$ denote the base changes to the corresponding universal spaces.
Let $\pi_{\mathrm{Hilb}_{\mathcal{D}}^{\varrho_{0}}}: \mathcal{D}\times\mathrm{Hilb}_{\mathcal{D}}^{\varrho_{0}}\to\mathrm{Hilb}_{\mathcal{D}}^{\varrho_{0}}$ be the projection.
Applying $R\pi_{\mathfrak{PT}_{\mathfrak{X}_{0}^\dagger/\mathfrak{C}_{0}^\dagger}^\varrho*}R\mathcal{H}om(\bar{\mathbb{I}}^\bullet_{\varrho}, -)^\vee$ to the above short exact sequence, we have the following diagram of distinguished triangles
\ben
\xymatrixcolsep{2pc}\xymatrix{
 g^*R\pi_{\mathrm{Hilb}_{\mathcal{D}}^{\varrho_{0}}*}R\mathcal{H}om(\bar{\mathbb{I}}_{0},\bar{\mathbb{I}}_{0})_{0}^\vee\ar[r]	\ar[d]&\bigoplus R\pi_{\mathfrak{PT}_{\mathfrak{Y}_{\pm}/\mathfrak{A}}^{\varrho_{\pm},\varrho_{0}}*}R\mathcal{H}om(\bar{\mathbb{I}}^\bullet_{\pm},\bar{\mathbb{I}}^\bullet_{\pm})_{0}^\vee\ar[d] \ar[r] & R\pi_{\mathfrak{PT}_{\mathfrak{X}_{0}^\dagger/\mathfrak{C}_{0}^\dagger}^\varrho*}R\mathcal{H}om(\bar{\mathbb{I}}^\bullet_{\varrho},\bar{\mathbb{I}}^\bullet_{\varrho})_{0}^\vee \ar[d] \\
 \mathbb{L}_{\mathfrak{PT}_{\mathfrak{X}_{0}^\dagger/\mathfrak{C}_{0}^\dagger}^\varrho/\mathfrak{PT}_{\mathfrak{Y}_{-}/\mathfrak{A}}^{\varrho_{-},\varrho_{0}}\times\mathfrak{PT}_{\mathfrak{Y}_{+}/\mathfrak{A}}^{\varrho_{+},\varrho_{0}}}\ar[r]	& \bigoplus\mathbb{L}_{\mathfrak{PT}_{\mathfrak{Y}_{\pm}/\mathfrak{A}}^{\varrho_{\pm},\varrho_{0}}/\mathfrak{A}^{\varrho_{\pm},\varrho_{0}}}[1] \ar[r]  & \mathbb{L}_{\mathfrak{PT}_{\mathfrak{X}_{0}^\dagger/\mathfrak{C}_{0}^\dagger}^\varrho/\mathfrak{C}_{0}^{\dagger,\varrho}}[1] &
}
\een
where the second row is the distinguished triangle of truncated cotangent complexes for $\mathfrak{PT}_{\mathfrak{X}_{0}^\dagger/\mathfrak{C}_{0}^\dagger}^\varrho\to\mathfrak{PT}_{\mathfrak{Y}_{-}/\mathfrak{A}}^{\varrho_{-},\varrho_{0}}\times\mathfrak{PT}_{\mathfrak{Y}_{+}/\mathfrak{A}}^{\varrho_{+},\varrho_{0}}$ relative to $\mathfrak{C}_{0}^{\dagger,\varrho}$ and the vertical arrows are induced by perfect obstruction theories. As shown in [\cite{LW}, (6.13)] or [\cite{Zhou1}, Theorem 7.11], we have 
\ben
R\pi_{\mathrm{Hilb}_{\mathcal{D}}^{\varrho_{0}}*}R\mathcal{H}om(\bar{\mathbb{I}}_{0},\bar{\mathbb{I}}_{0})_{0}^\vee\cong \mathbb{L}_{\mathrm{Hilb}_{\mathcal{D}}^{\varrho_{0}}/\mathrm{Hilb}_{\mathcal{D}}^{\varrho_{0}}\times\mathrm{Hilb}_{\mathcal{D}}^{\varrho_{0}}}.
\een
Then combined with the Cartesian diagram \eqref{diagram2}, using Serre duality for the above distinguished triangles, two perfect relative obstruction theories for $\mathfrak{PT}_{\mathfrak{X}_{0}^\dagger/\mathfrak{C}_{0}^\dagger}^\varrho$ and 
$\mathfrak{PT}_{\mathfrak{Y}_{-}/\mathfrak{A}}^{\varrho_{-},\varrho_{0}}\times \mathfrak{PT}_{\mathfrak{Y}_{+}/\mathfrak{A}}^{\varrho_{+},\varrho_{0}}$ are compatible over the diagonal morphism $\Delta$ in the sense of [\cite{BF}, Section 7]. By [\cite{BF}, Proposition 7.5], we have
\ben
[\mathfrak{PT}_{\mathfrak{X}_{0}^\dagger/\mathfrak{C}_{0}^\dagger}^\varrho]^\mathrm{vir}=\Delta^!([\mathfrak{PT}_{\mathfrak{Y}_{-}/\mathfrak{A}}^{\varrho_{-},\varrho_{0}}]^\mathrm{vir}\times[\mathfrak{PT}_{\mathfrak{Y}_{+}/\mathfrak{A}}^{\varrho_{+},\varrho_{0}}]^\mathrm{vir}).
\een
Now, the proof of the second identity is completed by Proposition \ref{split-cycles}.
And the first identity is obtained by [\cite{BF}, Proposition 7.2].
\end{proof}

\subsection{Numerical  version of degeneration formula} In this subsection, we will derive the  degeneration formula relating absolute orbifold PT invariants with relative orbifold PT invariants following the similar argument in [\cite{Zhou1}, Section 8.2] (see also [\cite{LW}, Section 6.3]). With the same assumption on $\pi:\mathcal{X}\to\mathbb{A}^1$ as in Section 6.3, we consider the moduli stack
$\mathfrak{PT}^P_{\mathfrak{X}/\mathfrak{C}}$ where the generalized Hilbert polynomial with respect to  $P\in\mathrm{Hom}(K^0(\mathcal{X}),\mathbb{Z})$ is a  polynomial of degree one. On the smooth fiber $\mathcal{X}_{c}$ for $0\neq c\in \mathbb{A}^1$, one can view $P\in F_{1}K(\mathcal{X}_{c})$ as in Section 6.2. As $P$ is a fixed Hilbert homomorphism associated to some admissible pure sheaves of dimension one on $\mathcal{X}_{0}[k]$ for all possible $k\geq0$, using  [\cite{Zhou1}, Corollary 5.3], $P$ is also a fixed Hilbert homomorphism for the pushforward of the corresponding sheaves by the contraction map to $\mathcal{X}_{0}=\mathcal{Y}_{-}\cup_{\mathcal{D}}\mathcal{Y}_{+}$. Then we have the splitting data $(\varrho_{-},\varrho_{0},\varrho_{+})$ of $P$ on the fiber decomposition $\mathcal{Y}_{-}\cup_{\mathcal{D}}\mathcal{Y}_{+}$, i.e., $\varrho_{\pm}\in\mathrm{Hom}(K^0(\mathcal{Y}_{\pm}),\mathbb{Z})$ and $\varrho_{0}=\omega(\mathcal{D})$. By the assumption on $(\mathcal{Y}_{\pm},\mathcal{D})$, one can view $\varrho_{\pm}\in F_{1}K(\mathcal{Y}_{\pm})$ and $P_{0}:=\varrho_{0}=i_{-}^*\varrho_{-}=i_{+}^*\varrho_{+}\in F_{0}K(\mathcal{D})$ where $i_{\pm}:\mathcal{D}\hookrightarrow\mathcal{Y}_{\pm}$ are the inclusions. Then one  can take  $(\varrho_{-},P_{0},\varrho_{+})$ satisfying $\varrho_{-}+\varrho_{+}-P_{0}=P\in F_{1}K(\mathcal{X}_{0})$ as the splitting data on the fiber decompositions $\mathfrak{PT}_{\mathfrak{Y}_{-}/\mathfrak{A}}^{\varrho_{-},P_{0}}\times_{\mathrm{Hilb}_{\mathcal{D}}^{P_{0}}}\mathfrak{PT}_{\mathfrak{Y}_{+}/\mathfrak{A}}^{\varrho_{+},P_{0}}$ as in Section 6.2. 

Assume $\gamma\in A^{*}_{\mathrm{orb}}(\mathcal{X})$ and let $\gamma_{\pm}$ be the restriction of $\gamma$ on $\mathcal{Y}_{\pm}\subset\mathcal{X}_{0}$. Denote by $\gamma$ again the restriction of $\gamma$ on $\mathcal{X}_{c}$ with $c\neq0$.
As in [\cite{Zhou1}], a basis $\{C_{k}\}$ of $A^*(\mathrm{Hilb}^{P_{0}}_{\mathcal{D}})$ are chosen such that
\ben
\int_{\mathrm{Hilb}^{P_{0}}_{\mathcal{D}}}C_{k}\cup C_{l}=g_{kl}
\een
and we have the following Kunneth decomposition of the diagonal class 
\ben
[\Delta]=\sum_{k,l}g^{kl}C_{k}\otimes C_{l}\in A^*(\mathrm{Hilb}^{P_{0}}_{\mathcal{D}}\times\mathrm{Hilb}^{P_{0}}_{\mathcal{D}},\mathbb{Q})
\een
where  $(g^{kl})$ is the inverse matrix of $(g_{kl})$. 
By using the similar argument in  [\cite{Zhou1}, Theorem 8.2] (see also [\cite{LW}, Theorem 6.8]) for the PT case, we have 
\begin{theorem} \label{numerical-deg1}
Assume $\gamma_{i}\in A_{\mathrm{orb}}^*(\mathcal{X})$ and  $\gamma_{i,\pm}$ are disjoint with $\mathcal{D}$ for $1\leq i\leq r$. For a fixed $P\in F_{1}K(\mathcal{X}_{c})$, we have
\ben
\bigg\langle\prod_{i=1}^r\tau_{k_{i}}(\gamma_{i})\bigg\rangle^P_{\mathcal{X}_{c}}=\sum_{\substack{\varrho_{-}+\varrho_{+}-P_{0}=P\\T\subset\{1,\cdots,r\},k,l}}\bigg\langle\prod_{i\in T}\tau_{k_{i}}(\gamma_{i,-})\bigg|C_{k}\bigg\rangle_{\mathcal{Y}_{-},\mathcal{D}}^{\varrho_{-}}g^{kl}\bigg\langle\prod_{i\notin T}\tau_{k_{i}}(\gamma_{i,+})\bigg|C_{l}\bigg\rangle_{\mathcal{Y}_{+},\mathcal{D}}^{\varrho_{+}}
\een
where $\varrho_{\pm}\in F_{1}K(\mathcal{Y}_{\pm})$ in the sum is taken over all possible splitting data satisfying $\varrho_{-}+\varrho_{+}-P_{0}=P$.
\end{theorem}
\begin{proof}
Let $\varrho_{0}=P_{0}$. As in the proof of Theorem \ref{cycle-degenerate}, we have the following short exact sequence
\ben
0\to \mathbb{F}_{\varrho}\to\Delta^*(\mathrm{pr}_{-}^*\mathbb{F}_{-}\oplus\mathrm{pr}_{+}^*\mathbb{F}_{+})\to \iota_{\mathcal{D}*}\iota_{\mathcal{D}}^*\mathbb{F}_{\varrho}\to0.
\een
Given $\gamma\in A_{\mathrm{orb}}^*(\mathcal{X})$, the operator $\widetilde{\mathrm{ch}}_{k+2}^{\mathrm{orb}}(\gamma)$ can be similarly defined as in Section 6.2. Notice that the refined Gysin homomorphism defined in the case of Deligne-Mumford stacks (see [\cite{Vis,Kre1}]) has the formal properties of commuting with proper pushforward and flat pull back, and commuting with intersection products with Chern classes as in [\cite{Ful}, Proposition 6.3].
If $\gamma_{\pm}$ is disjoint with $\mathcal{D}$, by Theorem \ref{cycle-degenerate}, we have
\ben
&&\widetilde{\mathrm{ch}}_{k+2}^{\mathrm{orb}}(\gamma)\cdot i_{0}^![\mathfrak{PT}^P_{\mathfrak{X}/\mathfrak{C}}]^\mathrm{vir}\\
&=&\sum_{\varrho\in\Lambda_{P}^{\mathrm{spl}}}\pi_{2*}\left(\widetilde{\mathrm{ch}}^\mathrm{orb}_{k+2}(\mathbb{F}_{\varrho})\cdot\iota^*\bar{\pi}_{1}^*\gamma\cap\pi_{2}^*\varsigma_{\varrho*}\Delta^!([\mathfrak{PT}_{\mathfrak{Y}_{-}/\mathfrak{A}}^{\varrho_{-},\varrho_{0}}]^\mathrm{vir}\times[\mathfrak{PT}_{\mathfrak{Y}_{+}/\mathfrak{A}}^{\varrho_{+},\varrho_{0}}]^\mathrm{vir})\right)\\
&=&\sum_{\varrho\in\Lambda_{P}^{\mathrm{spl}}}\varsigma_{\varrho*}\pi_{2*}\left(\widetilde{\mathrm{ch}}^\mathrm{orb}_{k+2}(\Delta^*\mathrm{pr}_{-}^*\mathbb{F}_{-})\cdot\iota^*\bar{\pi}_{1}^*\gamma_{-}\cap\pi_{2}^*\Delta^!([\mathfrak{PT}_{\mathfrak{Y}_{-}/\mathfrak{A}}^{\varrho_{-},\varrho_{0}}]^\mathrm{vir}\times[\mathfrak{PT}_{\mathfrak{Y}_{+}/\mathfrak{A}}^{\varrho_{+},\varrho_{0}}]^\mathrm{vir})\right)\\
&&+\sum_{\varrho\in\Lambda_{P}^{\mathrm{spl}}}\varsigma_{\varrho*}\pi_{2*}\left(\widetilde{\mathrm{ch}}^\mathrm{orb}_{k+2}(\Delta^*\mathrm{pr}_{+}^*\mathbb{F}_{+})\cdot\iota^*\bar{\pi}_{1}^*\gamma_{+}\cap\pi_{2}^*\Delta^!([\mathfrak{PT}_{\mathfrak{Y}_{-}/\mathfrak{A}}^{\varrho_{-},\varrho_{0}}]^\mathrm{vir}\times[\mathfrak{PT}_{\mathfrak{Y}_{+}/\mathfrak{A}}^{\varrho_{+},\varrho_{0}}]^\mathrm{vir})\right)\\
&=&\sum_{\varrho\in\Lambda_{P}^{\mathrm{spl}}}\varsigma_{\varrho*}\pi_{2*}\Delta^!\left(\left(\widetilde{\mathrm{ch}}^\mathrm{orb}_{k+2}(\mathbb{F}_{-})\cdot\iota_{-}^*\bar{\pi}_{1,-}^*\gamma_{-}\cap\pi_{2,-}^*[\mathfrak{PT}_{\mathfrak{Y}_{-}/\mathfrak{A}}^{\varrho_{-},\varrho_{0}}]^\mathrm{vir}\right)\times\pi_{2,+}^*[\mathfrak{PT}_{\mathfrak{Y}_{+}/\mathfrak{A}}^{\varrho_{+},\varrho_{0}}]^\mathrm{vir}\right)\\
&&+\sum_{\varrho\in\Lambda_{P}^{\mathrm{spl}}}\varsigma_{\varrho*}\pi_{2*}\Delta^!\left(\pi_{2,-}^*[\mathfrak{PT}_{\mathfrak{Y}_{-}/\mathfrak{A}}^{\varrho_{-},\varrho_{0}}]^\mathrm{vir}\times\left(\widetilde{\mathrm{ch}}^\mathrm{orb}_{k+2}(\mathbb{F}_{+})\cdot\iota_{+}^*\bar{\pi}_{1,+}^*\gamma_{+}\cap\pi_{2,+}^*[\mathfrak{PT}_{\mathfrak{Y}_{+}/\mathfrak{A}}^{\varrho_{+},\varrho_{0}}]^\mathrm{vir}\right)\right)\\
&=&\sum_{\varrho\in\Lambda_{P}^{\mathrm{spl}}}\varsigma_{\varrho*}
\Delta^!\left(\left(\widetilde{\mathrm{ch}}^\mathrm{orb}_{k+2}(\gamma_{-})\cdot[\mathfrak{PT}_{\mathfrak{Y}_{-}/\mathfrak{A}}^{\varrho_{-},\varrho_{0}}]^\mathrm{vir}\right)\times[\mathfrak{PT}_{\mathfrak{Y}_{+}/\mathfrak{A}}^{\varrho_{+},\varrho_{0}}]^\mathrm{vir}\right)\\
&&+\sum_{\varrho\in\Lambda_{P}^{\mathrm{spl}}}\varsigma_{\varrho*} \Delta^!\left([\mathfrak{PT}_{\mathfrak{Y}_{-}/\mathfrak{A}}^{\varrho_{-},\varrho_{0}}]^\mathrm{vir}\times\left(\widetilde{\mathrm{ch}}^\mathrm{orb}_{k+2}(\gamma_{+})\cdot[\mathfrak{PT}_{\mathfrak{Y}_{+}/\mathfrak{A}}^{\varrho_{+},\varrho_{0}}]^\mathrm{vir}\right)\right)\\
\een
where the map $\pi_{i}$ is the projections from $I_{\mathfrak{C}_{0}^{\dagger,\varrho}}\mathfrak{X}_{0}^{\dagger,\varrho}\times_{\mathfrak{C}_{0}^{\dagger,\varrho}}\mathfrak{PT}_{\mathfrak{X}_{0}^\dagger/\mathfrak{C}_{0}^\dagger}^\varrho$ to the $i$-th factor for $i=1,2$, the map $\bar{\pi}_{1}$ is the composition map of $\pi_{1}$ with the natural map $I_{\mathfrak{C}_{0}^{\dagger,\varrho}}\mathfrak{X}_{0}^{\dagger,\varrho}\to I\mathcal{X}_{0}$ and $\iota: I_{\mathfrak{C}_{0}^{\dagger,\varrho}}\mathfrak{X}_{0}^{\dagger,\varrho}\to I_{\mathfrak{C}_{0}^{\dagger,\varrho}}\mathfrak{X}_{0}^{\dagger,\varrho}$ is the canonical involution. Similarly, the map $\pi_{i,\pm}$ is the projections from $I_{\mathfrak{A}^{\varrho_{\pm},\varrho_{0}}}\mathfrak{Y}^{\varrho_{\pm},\varrho_{0}}\times_{\mathfrak{A}^{\varrho_{\pm},\varrho_{0}}}\mathfrak{PT}^{\varrho_{\pm},\varrho_{0}}_{\mathfrak{Y}_{\pm}/\mathfrak{A}}$ to the $i$-th factor ($i=1,2$), the map $\bar{\pi}_{1,\pm}$ is the composition map with $I_{\mathfrak{A}^{\varrho_{\pm},\varrho_{0}}}\mathfrak{Y}^{\varrho_{\pm},\varrho_{0}}\to I\mathcal{Y}_{\pm}$, and $\iota_{\pm}: I_{\mathfrak{A}^{\varrho_{\pm},\varrho_{0}}}\mathfrak{Y}^{\varrho_{\pm},\varrho_{0}}\to I_{\mathfrak{A}^{\varrho_{\pm},\varrho_{0}}}\mathfrak{Y}^{\varrho_{\pm},\varrho_{0}}$ is the canonical involution.

On the other hand, we have $\widetilde{\mathrm{ch}}_{k+2}^{\mathrm{orb}}(\gamma)\cdot i_{c}^![\mathfrak{PT}^P_{\mathfrak{X}/\mathfrak{C}}]^\mathrm{vir}=\widetilde{\mathrm{ch}}_{k+2}^{\mathrm{orb}}(\gamma)\cdot [	\mathrm{PT}_{\mathcal{X}_{c}/\mathbb{C}}^P]^\mathrm{vir}$ by Theorem \ref{cycle-degenerate}.
Applying $\widetilde{\mathrm{ch}}_{k+2}^{\mathrm{orb}}(\gamma_{i})$ successively to both $i_{0}^![\mathfrak{PT}^P_{\mathfrak{X}/\mathfrak{C}}]^\mathrm{vir}$ and $i_{c}^![\mathfrak{PT}^P_{\mathfrak{X}/\mathfrak{C}}]^\mathrm{vir}$, and  as in the proof of [\cite{LW}, Theorem 6.8], using the following identity (where $[\cdot]_{1}
\in A_{1}(\mathfrak{PT}^P_{\mathfrak{X}/\mathfrak{C}})$ means taking degree one part)
\ben
\deg i_{c}^!\left[\prod_{i=1}^r\widetilde{\mathrm{ch}}_{k+2}^{\mathrm{orb}}(\gamma_{i})\cdot [\mathfrak{PT}^P_{\mathfrak{X}/\mathfrak{C}}]^\mathrm{vir}\right]_{1}=\deg i_{0}^!\left[\prod_{i=1}^r\widetilde{\mathrm{ch}}_{k+2}^{\mathrm{orb}}(\gamma_{i})\cdot [\mathfrak{PT}^P_{\mathfrak{X}/\mathfrak{C}}]^\mathrm{vir}\right]_{1}
\een
and the Kunneth decomposition of $[\Delta]$, the proof is completed by using Definitions \ref{absolute-def} and \ref{relative-def}.
\end{proof}

Next, we will consider the descendent Pandharipande-Thomas generating functions and the associated numerical version of degeneration formula as in [\cite{Zhou1}, Section 8.2].
We recall the notion of multi-regular classes and their properties in [\cite{Zhou1}] as follows.
Suppose $(\mathcal{Y},\mathcal{D})$ is a smooth pair.
\begin{definition}([\cite{Zhou1}, Definition 8.3])\label{multi-regular-class}
We say that a class $P\in K(\mathcal{Y})$ is multi-regular if it can be represented by some coherent sheaf such that the associated representation of the stabilizer group at the generic point is a multiple of the regular representation.
\end{definition}	 
Let $F_{1}^{\mathrm{mr}}K(\mathcal{Y})\subset F_{1}K(\mathcal{Y})$ and $F_{0}^{\mathrm{mr}}K(\mathcal{D})\subset F_{0}K(\mathcal{D})$ be the subgroups generated by multi-regular classes respectively. Assume  $P\in F_{1}^{\mathrm{mr}}K(\mathcal{Y})$. If $P$ is represented by some admissible sheaf, we have $P_{0}=i^*P\in F_{0}^{\mathrm{mr}}K(\mathcal{D})\cong F_{0}K(D)\cong\mathbb{Z}$ where $D$ is the coarse moduli space of $\mathcal{D}$. Suppose $(\beta,\epsilon)\in \left(F_{1}^{\mathrm{mr}}K(\mathcal{Y})/F_{0}K(\mathcal{Y})\right)\oplus F_{0}K(\mathcal{Y})$ 
is the graded K-classes of $P$, then $P_{0}=n[\mathcal{O}_{x}]=\beta\cdot\mathcal{D}$ where $x\in\mathcal{D}$ is the preimage of some point in $D$. Hence for a simple degeneration $\pi: \mathcal{X}\to\mathbb{A}^1$ which is also a family of projective Deligne-Mumford stacks, we have the following splitting data:
\ben
\beta_{-}+\beta_{+}=\beta;\;\;\;\;\epsilon_{-}+\epsilon_{+}=\epsilon+n
\een
where $(\beta_{\pm},\epsilon_{\pm})\in \left(F_{1}^{\mathrm{mr}}K(\mathcal{Y}_{\pm})/F_{0}K(\mathcal{Y}_{\pm})\right)\oplus F_{0}K(\mathcal{Y}_{\pm})$ and $(\beta,\epsilon)\in \left(F_{1}^{\mathrm{mr}}K(\mathcal{X}_{c})/F_{0}K(\mathcal{X}_{c})\right)\oplus F_{0}K(\mathcal{X}_{c})$.

Now, we have the following definition of  generating functions for PT case
as in [\cite{Zhou1}].
\begin{definition}\label{PT-gen-func}
The descendent Pandharipande-Thomas generating function of the smooth fiber $\mathcal{X}_{c}$ is defined by 
\ben
Z_{\beta}^{\mathrm{PT}}\bigg(\mathcal{X}_{c};q\bigg|\prod_{i=1}^r\tau_{k_{i}}(\gamma_{i})\bigg):=\sum_{\epsilon\in F_{0}K(\mathcal{X}_{c})}\bigg\langle\prod_{i=1}^r\tau_{k_{i}}(\gamma_{i})\bigg\rangle_{\mathcal{X}_{c}}^{(\beta,\epsilon)}q^\epsilon
\een
where $\gamma_{i}\in A_{\mathrm{orb}}^*(\mathcal{X}_{c})$. 

The relative descendent Pandharipande-Thomas  generating function of a smooth pair $(\mathcal{Y},\mathcal{D})$ is defined by 
\ben
Z_{\beta,C}^{\mathrm{PT}}\bigg(\mathcal{Y},\mathcal{D};q\bigg|\prod_{i=1}^r\tau_{k_{i}}(\gamma_{i})\bigg):=\sum_{\epsilon\in F_{0}K(\mathcal{Y})}\bigg\langle\prod_{i=1}^r\tau_{k_{i}}(\gamma_{i})\bigg|C\bigg\rangle_{\mathcal{Y},\mathcal{D}}^{(\beta,\epsilon)}q^\epsilon
\een
where $\gamma_{i}\in A^*_{\mathrm{orb}}(\mathcal{Y})$ and $C\in A^*\left(\mathrm{Hilb}^{\beta\cdot\mathcal{D}}_{\mathcal{D}}\right)$.	
\end{definition}	
\begin{remark}
There is another definition of the descendent Pandharipande-Thomas generating function of $\mathcal{X}_{c}$ given in [\cite{Lyj}, Definition 5.23], where the  sum is taken over all the Euler characteristic $\chi(F_{\mathcal{E}_{\mathcal{X}_{c}}}(\mathcal{F}))$ (different from $\epsilon$) for all orbifold PT pairs $\varphi:\mathcal{O}_{\mathcal{X}_{c}}\to\mathcal{F}$ in the moduli space $\mathrm{PT}_{\mathcal{X}_{c}/\mathbb{C}}^{(\beta,\epsilon)}$.
\end{remark}

Then we have the following form of  Theorem \ref{numerical-deg1} in multi-regular case.
\begin{theorem}\label{multi-deg-form}
Assume $\gamma_{i}\in A_{\mathrm{orb}}^*(\mathcal{X})$ and  $\gamma_{i,\pm}$ are disjoint with $\mathcal{D}$  for $1\leq i\leq r$. For a fixed $\beta\in F_{1}^{\mathrm{mr}}K(\mathcal{X}_{c})/F_{0}K(\mathcal{X}_{c})$, we have
\ben
\bigg\langle\prod_{i=1}^r\tau_{k_{i}}(\gamma_{i})\bigg\rangle_{\mathcal{X}_{c}}^{(\beta,\epsilon)}=\sum_{\substack{\beta_{-}+\beta_{+}=\beta\\ \epsilon_{-}+\epsilon_{+}=\epsilon+n\\ T\subset\{1,\cdots,r\},k,l}}\bigg\langle\prod_{i\in T}\tau_{k_{i}}(\gamma_{i,-})\bigg|C_{k}\bigg\rangle_{\mathcal{Y}_{-},\mathcal{D}}^{(\beta_{-},\epsilon_{-})}g^{kl}\bigg\langle\prod_{i\notin T}\tau_{k_{i}}(\gamma_{i,+})\bigg|C_{l}\bigg\rangle_{\mathcal{Y}_{+},\mathcal{D}}^{(\beta_{+},\epsilon_{+})}.
\een
\end{theorem}

And hence we have

\begin{theorem}\label{mul-deg-form-genfun}
Assume $\gamma_{i}\in A_{\mathrm{orb}}^*(\mathcal{X})$ and  $\gamma_{i,\pm}$ are disjoint with $\mathcal{D}$  for $1\leq i\leq r$. For a fixed $\beta\in F_{1}^{\mathrm{mr}}K(\mathcal{X}_{c})/F_{0}K(\mathcal{X}_{c})$, we have
\ben
&&Z_{\beta}^{\mathrm{PT}}\bigg(\mathcal{X}_{c};q\bigg|\prod_{i=1}^r\tau_{k_{i}}(\gamma_{i})\bigg)\\
&=&\sum_{\substack{\beta_{-}+\beta_{+}=\beta\\T\subset\{1,\cdots,r\},k,l}}\frac{g^{kl}}{q^n}Z_{\beta_{-},C_{k}}^{\mathrm{PT}}\bigg(\mathcal{Y}_{-},\mathcal{D};q\bigg|\prod_{i\in T}\tau_{k_{i}}(\gamma_{i,-})\bigg)\times Z_{\beta_{+},C_{l}}^{\mathrm{PT}}\bigg(\mathcal{Y}_{+},\mathcal{D};q\bigg|\prod_{i\notin T}\tau_{k_{i}}(\gamma_{i,+})\bigg).
\een
\end{theorem}

\end{document}